\documentclass[12pt,a4paper]{elsarticle}

\usepackage[a4paper,top=2cm,bottom=2cm,left=2.5cm,right=2.5cm,marginparwidth=1.75cm]{geometry}

\usepackage{bm}
\usepackage{amsmath,amssymb,amsthm,mathrsfs}
\usepackage{amsfonts}

\usepackage[colorlinks=true, allcolors=blue]{hyperref}
\usepackage{booktabs} % For \toprule, \midrule, \botrule
\usepackage{threeparttable} % For table footnotes

\usepackage{lmodern}  % Ensures default LaTeX font is available

\usepackage{setspace}
\usepackage{titlesec}
\titleformat{\section} % Redefine section numbering format
  {\normalfont\Large\bfseries}{\thesection.}{1em}{}
\titleformat{\subsection}{\normalfont\centering\scshape}{\thesubsection}{1em}{}  



\usepackage[OT2,T1]{fontenc}
\DeclareSymbolFont{cyrletters}{OT2}{wncyr}{m}{n}
\DeclareMathSymbol{\Sha}{\mathalpha}{cyrletters}{"58}

\theoremstyle{plain}

\newtheorem{thms}{Theorem}[section]
\newtheorem{lems}[thms]{Lemma}
\newtheorem{props}[thms]{Proposition}
\newtheorem*{Cor}{Corollary}
\newtheorem{cor}{Corollary}

\newtheorem*{conj*}{Conjecture}
\newtheorem*{thm*}{Theorem}
\newtheorem*{prop*}{Proposition}
\newtheorem*{mainthm*}{Main Theorem}

\theoremstyle{definition}

\newtheorem{defns}[thms]{Definition}
\newtheorem{defprops}[thms]{Definition and Proposition}
\theoremstyle{remark}
\newtheorem*{Rem}{Remark}

\newcommand{\disp}{\displaystyle}

\newcommand{\ben}{\begin{enumerate}}
\newcommand{\een}{\end{enumerate}}
\newcommand{\beqn}{\begin{eqnarray*}}
\newcommand{\eeqn}{\end{eqnarray*}}
\newcommand{\bcas}{\begin{cases}}
\newcommand{\ecas}{\end{cases}}

\newcommand{\mb}{\mathbb}
\newcommand{\mr}{\mathrm}
\newcommand{\maf}{\mathfrak}

\newcommand{\hana}{\mathcal} 
\newcommand{\mt}{\mathtt}
\renewcommand{\leq}{\leqq}

\renewcommand{\geq}{\geqq}

\newcommand{\ord}{\mr{ord}}
\newcommand{\SL}{\mr{SL}}

\newcommand{\A}{\mb{A}}

\newcommand{\I}{\mr{I}}
\newcommand{\lsla}{\backslash}
\newcommand{\GL}{\mr{GL}}

\newcommand{\tr}{\mr{tr}}
\newcommand{\Her}{\mr{Her}}
\newcommand{\Gal}{\mr{Gal}}
\newcommand{\dHer}{\widehat{\Her}}
\newcommand{\wHer}{\widetilde{\Her}}
\newcommand{\perd}[2]{\langle{#1,\,#2}\rangle}

\newcommand{\wG}{\widetilde{G}}
\newcommand{\Ind}{\mr{Ind}}
\newcommand{\Ge}{\hana{G}}
\newcommand{\wGe}{\widetilde{\hana{G}}}
\newcommand{\wPe}{\widetilde{\hana{P}}}
\newcommand{\Pe}{\hana{P}}
\newcommand{\m}{\bm{m}}
\newcommand{\n}{\bm{n}}
\newcommand{\Hom}{\mr{Hom}}
\newcommand{\wF}{\widetilde{F}}
\newcommand{\wP}{\widetilde{P}}

\newcommand{\K}{\hana{K}}

\newcommand{\D}{\hana{D}}

\newcommand{\wK}{\widetilde{\K}}
\newcommand{\prm}{\mr{prm}}

\newcommand{\M}{\hana{M}}

\newcommand{\Z}{\hana{Z}}
\newcommand{\Slk}{\maf{S}}
\newcommand{\Res}{\mr{Res}}

\newcommand{\ind}{\mr{ind}}
\newcommand{\Lmd}{\Lambda}
\newcommand{\wLmd}{\widetilde{\Lmd}}

\newcommand{\Ad}{\mr{Ad}}
\newcommand{\vol}{\mr{vol}}

\newcommand{\bmf}[1]{\bm{\mr{#1}}}
\newcommand{\wM}{\widetilde{\hana{M}}}
\newcommand{\Mat}{\mr{Mat}}
\newcommand{\W}{\hana{W}}
\newcommand{\R}{\hana{R}}

\title{On the Period formula for the Ikeda--Yamana Lift for the Unitary Group}
\author{Jin Higashitani}
\date{\today}

\begin{document}

\allowdisplaybreaks

\begin{abstract}
Let $F$ be a totally real field and let $E/F$ be a quadratic CM extension.
Let $n$ be an odd positive integer. Yamana constructed a lift from Hilbert modular forms to automorphic forms of $GU(n,n)$. We denote by $\I_n(f)$ the Ikeda--Yamana lift of a Hilbert cusp form $f$ of weight $(\kappa_v)_{v\in\bmf{a}}$ with level one.
We express the Petersson norm $\perd{\I_n(f)}{\I_n(f)}$ of $\I_n(f)$ for a Hecke eigenform $f$ in terms of special values of the adjoint $L$-functions attached to $f$.
This extends Katsurada's result concerning Ikeda's conjecture.
\end{abstract}

\maketitle

\textbf{Keywords}: automorphic forms, L-functions, periods, lifting.
%-
% Paper Body
%-

%- Section -%
\section{Introduction}\label{sec1}
A central problem in the theory of modular forms is to determine their periods explicitly. 
Such periods are often related to special values of $L$-functions through lifting constructions, which also shed light on Satake parameters and functional transfer of automorphic representations.
Periods of various lifts have been studied from an arithmetic viewpoint: 
Zagier \cite{Z} treated the Doi--Nakamura lift, 
Murase and Sugano \cite{MS} studied the Kudla lift, 
and Ibukiyama, Katsurada and Kojima \cite{IKK} investigated the Miyawaki--Ikeda lift.
For symplectic groups over $\mb{Q}$, Ikeda \cite{Ike2} constructed the Duke--Imamoglu--Ikeda lift, thereby settling the Duke--Imamoglu conjecture.
Ikeda constructed the unitary analogue in \cite{Ike1} and conjectured an explicit period formula; see ~\cite[Conjecture 17.5]{Ike1}.
Let $F$ be a totally real field and $E$ a quadratic CM extension of $F$.
The unitary similitude group $\wGe_n$ and its subgroup $\Ge_n$ are defined by
\[
\wGe_n = \{ g \in \mathrm{Res}_{E/F} \GL_{2n} \mid {}^t g^{\tau} J_n g = \nu_n(g) J_n, \ \nu_n(g) \in\mb{G}_m \}, 
\]
\[
\Ge_n = \{ g \in \wGe_n \mid \nu_n(g) = 1 \},
\]
where $\tau$ is the nontrivial automorphism in $\mathrm{Gal}(E/F)$.
When $F=\mb{Q}$, Katsurada~\cite{K1,K2,K3} obtained an explicit period formula for the Duke--Imamoglu--Ikeda lift on $\Ge_n(\mb{Q})$, refining Ikeda’s conjecture.
\begin{thm*}[{\cite[Theorem~2.2]{K3}}]
Assume that $n$ is odd and $F = \mb{Q}$. 
Let $f$ be a normalized Hecke eigenform of weight $k$ and full level.
Let $D_E$ be the discriminant of $E/\mb{Q}$ and $\epsilon_{E/\mb{Q}}$ the quadratic Hecke character attached to $E/\mb{Q}$.
Then
\[
\langle I_n(f), I_n(f) \rangle = 2^{(n-1)k} D_E^{\beta} \, \widetilde{\Lambda}(1, f, \mathrm{Ad}) 
\prod_{i=2}^n \widetilde{\Lambda}(i, \epsilon_{E/\mb{Q}}^i) \widetilde{\Lambda}(i, f, \mathrm{Ad}, \epsilon_{E/\mb{Q}}^{i-1}),
\]
where $\beta$ is a certain half-integer.
\end{thm*}

Yamana~\cite{Y} constructed the unitary analogue of the Duke–Imamoglu–Ikeda lift over totally real fields.
We refer to this lift as the \underline{\emph{Ikeda--Yamana lift}}.

Let $\pi$ be an irreducible cuspidal automorphic representation of $\GL_2(\A_F)$ associated with the adelization of a Hilbert cusp form $f$ of weight $\kappa$ and let $\omega$ be a central character of $\pi$.
Write $\pi \simeq \otimes'_{v\in\bmf{v}} \pi_v=(\otimes'_{v\in\bmf{f}} \pi_v)\otimes(\otimes_{v\in\bmf{a}}\pi_v)$ and $\pi_{\bmf{f}} \simeq \otimes'_{v\in\bmf{f}} \pi_v$.
Let $\chi$ be a Hecke character of $\mb{A}^\times_E$ such that $\chi|_{\mb{A}^\times_F} = \omega$.
Let $\pi^E$ be the base change of $\pi$ to $\GL_2(\A_E)$, and set $\pi[\chi] = \chi \otimes \pi^E$.
Via the isomorphism
\[
\wGe_1 \simeq (\mathrm{Res}_{E/F} \GL_1 \times \GL_2)/\{\text{diag}(\GL_1)\},
\]
we view $\chi^{-1} \boxtimes \pi$ as a representation of $\wGe_1$.
Let $\Pe_2$ be the parabolic subgroup of $\wGe_n$ whose Levi factor is $\mathrm{Res}_{E/F} \GL_2^{\,(n-1)/2} \times \wGe_1$.
For a finite place $v$, define
\[
J_n^{\chi_v}(\pi_v) = \Ind_{\Pe_2}^{\wGe_n} \big( \W(\pi_v[\chi_v])^{\boxtimes (n-1)/2} \boxtimes \W(\pi_v \boxtimes \chi_v^{-1}) \big),
\]
where $\W$ denotes the Whittaker model.
Let $A_n^{\chi_v}(\pi_v)$ be its irreducible subrepresentation. 
Yamana constructed a family of Shalika functionals $\{\mathfrak{S}_B^{\hat{\chi}}\}$ on $A_n^{\hat{\chi}}(\pi_{\bmf{f}})$; see~\cite{Y} for details.
For $l=(l_v)_{v\in\bmf{a}}\in\mb{Z}^{\bmf{a}}$, define $\varepsilon_{\infty}^l$ by
\begin{equation*}
    \varepsilon_{\infty}^{l}(a_\infty)=\prod_{v\in \bmf{a}}(\overline{a_v}/a_v)^{l_v/2}
\end{equation*}
where $a_{\infty}=(a_v)_{v\in \bmf{a}}\in E\otimes \mb{R}$ and $a_v\ne0$. 
Let $l(\chi)=(l_v)_{v\in \bmf{a}}\in\mb{Z}^{\bmf{a}}$ so that $\varepsilon_{\infty}^{l(\chi)}=\chi|_{E\otimes\mb{R}}$.
We define the archimedean Whittaker functions $W_B^{l,l'}$ by
\begin{equation*}
    W_B^{l,l'}(g)=(\det B)^{l/2}\exp(2\pi\sqrt{-1}\cdot\tr(B(g(\sqrt{-1}\cdot1_n))))\varepsilon^{l'}(\det g)J_n(g,\sqrt{-1}1_n)
\end{equation*}
where $J_n(g,Z)$ is the automorphy factor.
\begin{props}[{\cite[Theorem 1.1]{Y}}]
Let $\pi$, $\chi$, and $A_n^{\hat{\chi}}(\pi_{\bmf{f}})$ be as above. 
For $\phi = \otimes'_{v\in\bmf{f}} \phi_v \in A_n^{\hat{\chi}}(\pi_{\bmf{f}})$, set
\[
\I_n(g;\phi) = \sum_{B \in \mathrm{Her}_n^+(F)} \mathfrak{S}_B^{\hat{\chi}}(\Pi(g_{\bmf{f}})\phi) \, W_B^{\kappa+n-1,w}(g_\infty),\,\,w=\frac{1}{2}(\kappa+n-1+l(\chi))
\]
for $g = g_\infty g_{\bmf{f}} \in \wGe_n(\A)$, where we denote by $\Pi$ the action on $A_n^{\chi}(\pi_{\bmf{f}})$.
Here $\mathfrak{S}_B^{\hat{\chi}}$ is the Shalika functional from~\cite{Y}.

Then $\I_n$ defines an automorphic form on $\wGe_n(\A)$, called the \underline{\emph{Ikeda–Yamana lift}} of $\pi$ with respect to $\chi$.
\end{props}

 We now turn to our main theorem. 
Let $\perd{\cdot}{\cdot}$ denote the Petersson inner product on $\wGe_n(\A)$.
For a Hecke eigenform $f$ with full level and weight $(\kappa_v)$, let $\pi$ be the irreducible cuspidal automorphic representation generated by $f$.
Let $\phi_f$ be the unramified vector in $A_n^{\chi}(\pi_{\bmf{f}})$ and set
\[\I_n(f)(g)=\I_n(g;\phi_f).\]
Our goal is to give an explicit formula for 
\begin{equation*}
    \perd{\I_n(f)}{\I_n(f)}
\end{equation*}
in terms of special values of $L$-functions.
\begin{mainthm*}\label{thm:main_Intro}
Let $F$ be a totally real number field and $E/F$ a quadratic CM extension. Assume that $n$ is odd.
Let $\pi$ be an irreducible cuspidal automorphic representation of $\GL_2(\A_F)$ generated by a Hilbert cusp form $f$ of weight $\kappa$, with trivial central character $\omega_\pi = 1$.
Let $\chi$ be a Hecke character of $\A_E^\times$ such that $\chi|_{\mb{A}^\times_F} = \omega_\pi$.
Let $\I_n(f)$ be the Ikeda--Yamana lift defined in Proposition~\ref{YamanaMT}, and let $\perd{\cdot}{\cdot}$ denote the Petersson inner product on $\wGe_n(\A)$.
Then the Petersson norm of $\I_n(f)$ is given by
\begin{align*}
    &\perd{\I_n(f)}{\I_n(f)}\\
    ={}&2^{-n\kappa+(2n^2-n)d_F+1}|D_E|^{n(n-1)/4}|D_F|^{-n(4n-1)/2}\prod_{i=1}^n\Gamma_{\mb{C}}(i)^{d_F}\prod_{i=1}^n\dfrac{\Lmd(i,f,\Ad,\varepsilon_{E/F}^{i-1})}{\Lmd(n+i,\varepsilon_{E/F}^{i-1})}\\
    ={}&2^{-n\kappa+n^2d_F+1}|D_E|^{n(n-1)/4}|D_F|^{-n(4n-1)/2}\prod_{i=1}^n\dfrac{\Gamma_{\mb{C}}(n+i)^{d_F}}{\Gamma_{\mb{C}}(i)^{d_F}}\prod_{i=1}^n\dfrac{\wLmd(i,f,\Ad,\varepsilon_{E/F}^{i-1})}{\wLmd(n+i,\varepsilon_{E/F}^{i-1})}
\end{align*}
where $\varepsilon_{E/F}$ is the quadratic Hecke character associated with $E/F$.
\end{mainthm*}

The Main Theorem expresses the Petersson norm of the Ikeda--Yamana lift in terms of special values of $L$-functions.
In subsequent sections, we develop the local and global theory necessary to prove this result.
The proof relies on Yamana's explicit construction of the lift~\cite{Y}, the Rankin--Selberg integral introduced in Section \ref{sec4}, and a residue method involving Eisenstein series.
In Section 2, we recall some basic facts needed for the study of the Ikeda--Yamana lift and the computation of its period.
In Section 3, we describe the construction of the Ikeda--Yamana lift and then give an explicit expression for $\I_n(f)$.
We then review previous results on the periods of the Duke--Imamoglu--Ikeda lift and state our main theorem.
In Section 4, we introduce the Rankin--Selberg integral and compute the term corresponding to the period using the residue of an Eisenstein series. 
We also decompose the Rankin--Selberg integral into an Euler product.
In Section 5, we compute the local Euler factors using relations between special values of Siegel series and local densities, and reduce the resulting expressions to Koecher--Maass integrals and zeta integrals.
In Section 6, we prove the main theorem by combining the previous calculations.

\medskip
\noindent\textbf{Acknowledgments.}
I would like to thank Professor Tamotsu Ikeda and Dr. Nobuki Takeda for their helpful comments and discussions.

\subsection*{Notation.}

Fix a totally real number field $F$ and a quadratic CM extension $E$ of $F$.
Let $d_F=[F:\mb{Q}]$.
We denote by $\bmf{v}$ the set of all places of $F$, by $\bmf{f}$ the set of finite places of $F$, and by $\bmf{a}$ the set of infinite places of $F$.
For $v\in\bmf{v}$, we put $E_v=F_v\otimes_F E$.
We denote the nontrivial automorphism of $E$ over $F$ by $\rho_{E/F}$ and write $\rho_{E/F}(a\otimes x)=a\otimes \rho_{E/F}(x)$. 
If $E_v$ is split, fix an isomorphism $E_v\cong F_v\oplus F_v$ under which $\rho_{E/F}$ interchanges the two factors.
For $v\in\bmf{f}$, we denote by $\hana{O}_v$ the ring of integers of $F_v$ and by $\maf{p}_v$ the maximal ideal of $\hana{O}_v$.
We denote by $\hana{O}_{E_v}$ the integral closure of $\hana{O}_v$ in $E_v$ and if $v$ is split, we identify $\hana{O}_{E_v}=\hana{O}_{F_v}\oplus\hana{O}_{F_v}$.
We denote by $\maf{q}_v$ the maximal ideal of $\hana{O}_{E_v}$ if $v$ is not split, and we put $\maf{q}_v=\maf{p}_v\oplus\maf{p}_v$ where we consider $\hana{O}_{E_v}$ as $\hana{O}\oplus\hana{O}$ if $v$ is split.
We denote by $\A$ (resp. $\A_E$) the adele of $F$ (resp. $E$) and by $\A^{\times}$ (resp. $\A_E^{\times}$) the idele of $F$ (resp. $E$).
We denote by $\mb{A}_{\bmf{f}}$ the finite part of $\A$.
Let $\varepsilon_{E/F}=\prod_{v\in\bmf{v}}\varepsilon_{E_v/F_v}$ be the Hecke character $\A^{\times}$ attached to $E/F$.

For $v\in\bmf{f}$, we put $q_v=\#\hana{O}_v/\maf{p}_v$ and $q_{E_v}=\#\hana{O}_{E_v}/\maf{q}_v$.
We denote by $\hana{O}_{v}^\times$ the unit group of $\hana{O}_v$ and by $\hana{O}_{E_v}^\times$ the unit group of $\hana{O}_{E_v}$.
We denote by $\ord_v$ the additive valuation of $F_v$ so that $\ord_v(x)=1$ when $x$ is a uniformizer. 
We denote by $N_{E/F}$ the norm map from $E$ to $F$ and by $\tr_{E/F}$ the trace map from $E$ to $F$.
For an $F$-algebra $R$, we also denote by $N_{E/F}$ the norm map from $R\otimes_F E$ to $R$.
Put $\alpha_{F_v}(x)=q_v^{-\ord_v(x)}$ and $\alpha_{E_v}=\alpha_{F_v}\circ N_{E/F}$.
We choose a uniformizer $\varpi_v$ in $F_v$ and an element $\varpi_{E_v}$ of $E_v$ that satisfies
\begin{equation*}
\begin{cases}
\varpi_v=\varpi_{E_v}&\text{if $E_v$ is inert over $F_v$},\\
\varpi_{E_v}=(\varpi_v,\varpi_v)&\text{if $E_v$ is split over $F_v$},\\
\varpi_v=N_{E/F}(\varpi_{E_v})&\text{if $E_v$ is ramified over $F_v$}.
\end{cases}
\end{equation*}
We denote by $\maf{D}_{F/\mb{Q},v}$ the discriminant ideal of $F_v$ relative to $\mb{Q}_p$ and by $\maf{d}_{F/\mb{Q},v}$ the different ideal of $F_v$ relative to $\mb{Q}_p$ where $v|p$.
If $v$ is not split, we denote by $\maf{D}_{E/F,v}$ the discriminant ideal of $E_v$ relative to $F_v$, by $\maf{d}_{E/F,v}$ the different ideal of $E_v$ relative to $F_v$. 
We put $\maf{f}_v=\ord_v(\maf{D}_{E/F,v})=\min\{\ord_v(x)|x\in\maf{D}_{E/F,v}\}$ and $f_v=\ord_p(\maf{D}_{F/\mb{Q},v})$ where $v|p$. 
If $v$ is split, we put $\maf{f}_v=f_v=0$.\\

Let $D_E$ be the discriminant of $E$ over $\mb{Q}$ and $D_F$ be the discriminant of $F$ over $\mb{Q}$.

For $v\in\bmf{a}$, we let $\alpha_{F_v}$ be the usual absolute value under $F_v\cong\mb{R}$ and put $\alpha_{E_v}=\alpha_{F_v}\circ N_{E/F}$.

Put $\alpha_{\A}((x_v)_{v\in\bmf{v}})=\prod_{v\in\bmf{v}}\alpha_{F_v}(x_v)$ for $(x_v)_{v\in\bmf{v}}\in\A$ and put $\alpha_{\A_E}=\alpha_{\A}\circ N_{E/F}$.

We denote by $\alpha_R^s$ the $s$-power of $\alpha_R$ for $R=F_v$ and $\A$.
 
We denote the additive algebraic group of $m \times n$ matrices by $\Mat_{m,n}$ and set $\Mat_n=\Mat_{n,n}$.
For a commutative ring $R$ and a subset $S$ of $R$, we denote by $\Mat_{m,n}(S)$ the set of ($m,n$)-matrices with entries in $S$. 
For the subset $X$ of $\Mat_n(R)$, set $X^{nd}=\{x\in X|\det x\ne0\}$. 

For an $F$-algebra $R$, $a\in R$ and $x\in E$, we write $\overline{a\otimes x}=a\otimes\rho_{E/F}(x)$.
We put $X^*=(\overline{X}_{ji})_{1\leqq i\leqq n,1\leqq j\leqq m}$ for $X=(X_{ij})_{1\leqq i\leqq m,1\leqq j\leqq n}\in \Res_{E/F}\Mat_{m,n}$ where $\Res_{E/F}$ denotes Weil restriction. 
We write $A[X]=X^*AX$ for $A\in \Res_{E/F}\Mat_m$ and $X\in \Res_{E/F}\Mat_{m,n}$.

We put $\Her_n=\{X\in \Res_{E/F}\Mat_n|X^*=X\}$. 
For a $F$-algebra $R$ and a $\Gal(E/F)$-stable subset $S$ of $R\otimes_F E$, we denote by $\Her_n(R;S)$ the subset of $\Her_n(R)$ with entries in $S$.
For a subring $R'$ of $R\otimes_F E$, a Hermitian matrix $A$ of degree $n$ with entries in $R\otimes_F E$ is said to be semi-integral over $(R,R')$ if $\tr(AB)\in R\cap R'$ for any $B\in\Her_n(R;R')$.  
We denote by $\dHer_n(R; R')$ the set of semi-integral matrices of degree $n$ over $(R,R')$. 
We put $\dHer_{n,v}=\dHer_n(F_v;\hana{O}_{E_v})$.

Unless otherwise stated, Haar measures on the adelic points of algebraic groups are taken to be the Tamagawa measure.

%- Section -%
\section{Preliminaries}\label{sec2}
We recall some basic facts about algebraic groups, Hilbert-Hermitian modular forms, and automorphic forms.
\subsection{Algebraic Group $U(n,n)$ and Hermitian Matrices $\Her_n$}\label{subsec2-1}
Put \begin{equation*}
    J_n=
\begin{pmatrix}
0&-1_n\\
1_n&0
\end{pmatrix}. 
\end{equation*}
We denote by $\Ge_n$ the unitary group $U(n,n)$ over $F$, that is, 
\begin{equation*}
\Ge_n=\{g\in \Res_{E/F}\GL_{2n}|J_n[g]=J_n\}.
\end{equation*}
Let $\wGe_n$ be the similitude unitary group, that is,  
\begin{equation*}
\wGe_n(R)=\{g\in \Res_{E/F}\GL_{2n}(R)|J_n[g]=\nu_n(g)J_n,\nu_n(g)\in R^{\times}\}
\end{equation*}
for every $F$ algebra $R$.
We call $\nu_n$ the similitude character of $\wGe_n$.
We define 
\begin{equation*}
\Ge_n^1=\{g\in \Ge_n|\det g=1\}.
\end{equation*}
For $A\in\Res_{E/F}\GL_n$, $z\in\Her_n$, and $t\in \mb{G}_m$, we put
\begin{equation*}
\m_n(A)=\begin{pmatrix}
A&0\\
0&(A^*)^{-1}
\end{pmatrix}\text{, }
\n_n(z)=\begin{pmatrix}
1_n&z\\
0&1_n
\end{pmatrix}\text{, }
\bm{d}_n(t)=\begin{pmatrix}
1_n&0\\
0&t\cdot1_n
\end{pmatrix}.
\end{equation*}
Define the maximal parabolic subgroup $\wPe_n$ of $\wGe_n$ by
\begin{equation*}
\wPe_n=\{\m_n(A)\n_n(z)\bm{d}_n(t)|A\in\Res_{E/F}\GL_n,z\in\Her_n,t\in \mb{G}_m\}.
\end{equation*}
Put $\Pe_n=\wPe_n\cap\Ge_n$. 
Let $\wM_n$ be the Levi subgroup of $\wPe_n$ such that 
\begin{equation*}
    \wM_n=\{\m_n(A)\bm{d}_n(t)|A\in\Res_{E/F}\GL_n,t\in\mb{G}_m\}
\end{equation*}
and $U_{\Pe_n}$ be the unipotent radical of $\Pe_n$. 
Put $\M_n=\Pe_n\cap\wM_n$.
We denote by $\Z_n$ the center of $\wGe_n$.

For $v\in\bmf{f}$ and fractional ideals $\maf{a},\maf{b}$ of $E_v$, we define $\Gamma_{n,v}[\maf{a},\maf{b}]$ by
\begin{equation*}
    \Gamma_{n,v}[\maf{a},\maf{b}]=\left\{\begin{pmatrix}
        A_{11}&A_{12}\\
        A_{21}&A_{22}
    \end{pmatrix}\in \wGe_n(F_v)\middle|A_{12}\in \Mat_n(\maf{a}),A_{21}\in \Mat_n(\maf{b}),A_{11},A_{22}\in \Mat_n(\hana{O}_{E_v})\right\}.
\end{equation*}
 We define the maximal compact group of $\wGe_n(F_v)$ by $\wK_{n,v}=\Gamma[\maf{d}_{F_v/\mb{Q}_p}^{-1}\hana{O}_{E_v},\maf{d}_{F_v/\mb{Q}_p}\hana{O}_{E_v}]$ where $v|p$. 
 We put $\K_{n,v}^1=\Ge_n^1(F_v)\cap\K_{n,v}$ and $\K_{n,v}=\Ge_n(F_v)\cap\wK_{n,v}$. 
 We put $\K_{n,\bmf{f}}=\prod_{v\in\bmf{f}}\K_{n,v}$, $\wK_{n,\bmf{f}}=\prod_{v\in\bmf{f}}\wK_{n,v}$ and $\K_{n,\bmf{f}}^1=\prod_{v\in\bmf{f}}\K_{n,v}^1$. We write $\K_{n,v}^0=\GL_n(\hana{O}_{E_v})$ and $\K_{n,\bmf{f}}^0=\prod_{v\in\bmf{f}}\K_{n,v}^0$. 

 For each of the pairs $(G, K_v)=(\wGe_n, \wK_{n,v}),\, (\Ge_n, \K_{n,v}),\, (\Ge_n^1, \K_{n,v}^1), \,(\Res_{E/F}\GL_n,\K_{n,v}^0)\}$, we normalize the Haar measure $dg_v$ on $G(F_v)$ by $\mathrm{vol}(K_v, dg_v) = 1$ for $v\in\bmf{f}$. 

Then we have $\Ge_n^1(\A)=\Ge_n^1(F)\K_{n,\bmf{f}}^1\Ge_n^1(\mb{R}\otimes F)$ by the strong approximation theorem. 
Note that the determinant and the similitude character induce a homomorphism
\begin{equation*}
\wGe_n(\A) \longrightarrow \{(x,t)\in \A_E^{\times}\times\A^{\times}|N_{E/F}(x)t^{-2n}=1\},
\end{equation*}
given by $g \mapsto (\det(g), \nu(g))$, whose kernel is precisely $\Ge_n^1(\A)$. Since the order of the associated adelic class group coincides with $\maf{h}_E$, we can choose a complete set of representatives $\{\gamma_1,\cdots,\gamma_{\maf{h}_{E}}\}\subset\wGe_n(\A_{\bmf{f}})$ for this class group to obtain the double coset decomposition:
\begin{equation*}
\wGe_n(\A)=\bigsqcup_{i=1}^{\maf{h}_{E}}\wGe_n(F)\gamma_i\wGe_n(\mb{R}\otimes F)\wK_{n,\bmf{f}}.
\end{equation*}

For $v\in\bmf{a}$, let $\wGe_n^+(F_v)$ be the connected component of $1_{2n}$ in $\wGe_n(F_v)$. Put $\wGe_n^+(\mb{R}\otimes F)=\prod_{v\in\bmf{a}}\wGe_n^+(F_v)$, $\wGe_n^+(\A)=\wGe_n(\A_{\bmf{f}})\times\wGe_n^+(\mb{R}\otimes F)$ and $\wGe_n^+(F)=\wGe_n^+(\A)\cap\wGe_n(F)$.

For $R=F$, $F_v$ ($v\in\bmf{a}$) and $\A$, we put $\wPe_n^+(R)=\wGe_n^+(R)\cap\wPe_n(R)$, $\wM_n^+(R)=\wGe_n^+(R)\cap\wM_n(R)$. We have $\Z_{n}(R)\subset\wGe_n^+(R)$.

We denote by $\Her_n(F)^+$ the subset of $\Her_n(F)$ that are positive definite at every archimedean place.  
Then we have the following proposition.
\begin{props}\label{prop:pre2}
Let $\varepsilon_v:\Her_n^{nd}(F_v)\to\{-1,1\}$ be the function defined by
\begin{equation*}
\varepsilon_v(A)=\begin{cases}
1&\text{$\det A\in N_{E/F}(E_v^\times)$}\\
-1& \text{otherwise}
\end{cases}
\end{equation*}
for $A\in\Her_n^{nd}(F_v)$. Then the natural map
\begin{equation*}
\Her_n(F)^+/GL_n(E)\to\left\{(A_v)_{v\in\bmf{f}}\in\prod_{v\in\bmf{f}}\Her_n(F_v)^{nd}/\GL_n(E_v)\middle|\begin{matrix}\varepsilon_v(A_v)=1\text{ for almost all $v$}\\
\prod_{v\in\bmf{f}}\varepsilon_v(A_v)=1
\end{matrix}\right\}
\end{equation*}
induces a bijection.
\end{props}
\begin{proof}    
 Injectivity and surjectivity follow from \cite[Theorem 6.1, Corollary 6.6, Theorem 6.9]{WS}.
\end{proof}

To simplify the notation, we shall omit the subscript $n$ whenever no confusion can arise.

\subsection{Hilbert-Hermitian Modular Forms and Automorphic Forms}\label{subsec2-2}
We now recall the definitions of Hilbert-Hermitian modular forms and automorphic forms on $\wGe_n^+(\A)$ for later comparison with previous work. \\
First, let $\maf{H}_n$ be the Hermitian upper half-plane
\begin{equation*}
\left\{Z\in \Mat_n(\mb{C})\middle|\dfrac{1}{2\sqrt{-1}}(Z-Z^*)>0\right\}.
\end{equation*}
Let $v\in\bmf{a}$.
Then, the group $\wGe_n^+(F_v)$ acts on $\maf{H}_n$ by
\begin{align*}
hZ=(a(h)Z+b(h))(c(h)Z+d(h))^{-1}
\end{align*}
for $h=\begin{pmatrix}
a(h)&b(h)\\
c(h)&d(h)
\end{pmatrix}\in\wGe_n^+(F_v)$ and $Z\in\maf{H}_n$.

We call $\Gamma$ a congruence subgroup of $\wGe_n^+(F)$ if there exists a compact open subgroup $K_{\bmf{f}}$ of $\wGe_n(\A_{\bmf{f}})$ such that $\Gamma=\iota_{\bmf{f}}^{-1}(K_{\bmf{f}})\cap\iota_{\bmf{a}}^{-1}(\wGe_n^+(\mb{R}\otimes F))$.
Here $\iota_{\bmf{f}}:\wGe_n(F)\to\wGe_n(\A_{\bmf{f}})$, $\iota_{\bmf{a}}:\wGe_n(F)\to\wGe_n(\mb{R}\otimes F)$ is the natural embedding.\\
We define the automorphy factors for $\wGe_n(F_v)$ as follows. For $h=\begin{pmatrix}
a(h)&b(h)\\
c(h)&d(h)
\end{pmatrix}\in\wGe_n(F_v)$ and $Z\in\maf{H}_n$, we put 
\begin{align*}
\mu(h,Z)=c(h)Z+d(h)&\hspace{5mm}
j(h,Z)=\nu(h)^{-n/2}\det(\mu(h,Z)).
\end{align*}
Then we define the action of $\wGe_n^+(\mb{R}\otimes F)$ of weight $\kappa=(\kappa_v)_{v\in\bmf{a}}\in\mb{Z}^{\bmf{a}}$ on the space of holomorphic functions $f$ on $\maf{H}_n^{\bmf{a}}$ by
\begin{equation*}
(f||_\kappa g)(z)=\prod_{v\in \bmf{a}}j(g_v,z_v)^{-{\kappa_v}}f(gz)\text{ for $z=(z_v)_v\in\maf{H}_n^{\bmf{a}}$ and $g=(g_v)\in\wGe_n^+(\mb{R}\otimes F)$}.
\end{equation*}

We can now define Hilbert-Hermitian modular forms.

\begin{defns}\label{def: Hilbert_Hermitian_modular_form}
Let $\kappa=(\kappa_v)_{v\in \bmf{a}}\in\mb{Z}^{\bmf{a}}$ be a tuple of integers and $\Gamma$ be a congruence subgroup of $\wGe_n^+(F)$. A function $f:\maf{H}_n^{\bmf{a}}\to\mb{C}$ is called a Hilbert-Hermitian modular form of weight $\kappa$ with respect to $\Gamma$ if it satisfies the following condition:
\begin{itemize}
\item[(1)] $f$ is holomorphic on $\maf{H}_n^{\bmf{a}}$ and $f$ has the Fourier expansion of the form 
\begin{equation*}
f(Z)=\sum_{B\in L,B\geqq0}a(f,B)\exp\left(2\pi\sqrt{-1}\sum_{v\in \bmf{a}}\tr(BZ_v)\right).    
\end{equation*}
for some lattice $L$ of $\Her(F)$  where $B \geqq 0$ means that $B$ is a positive semidefinite hermitian matrix at all infinite places.
\item[(2)] $f||_{\kappa}\gamma=f$ for any $\gamma\in\Gamma$.
\item[(3)] If $n=1$ and $F=\mb{Q}$, $f$ is holomorphic at cusps.
\end{itemize}
We denote by $M_{\kappa}(\Gamma)$ the space of Hilbert-Hermitian modular forms of weight $\kappa$ with respect to a congruence subgroup $\Gamma$.
We call $f\in M_{\kappa}(\Gamma)$ a Hilbert cusp form if $a(f,B)=0$ unless $B$ is positive definite and we denote by $S_\kappa(\Gamma)$ its subspace consisting of cusp forms.
\end{defns}

\begin{defns}[adelic automorphic forms]
    A function $\phi:\wGe_n(F)\lsla\wGe_n(\A)\to \mb{C}$ is called an automorphic form on $\wGe_n(F)\lsla\wGe_n(\A)$ if it satisfies the following conditions:
    \begin{itemize}
        \item[(1)]$\phi$ is smooth and of moderate growth.
        \item[(2)]$\phi$ is right $\wK$-finite and $Z(\mr{Lie}\hspace{1mm}\wGe_n)(\mb{R}\otimes F)$-finite
    \end{itemize}
    where $Z(\mr{Lie}\hspace{1mm}\wGe_n)(\mb{R}\otimes F)$ is the center of the universal enveloping algebra of $\mr{Lie}\hspace{1mm}\wGe_n(\mb{R}\otimes F)$. See \cite{GJH} for details.
\end{defns}

\begin{Rem}[classical-adelic correspondence]\hspace{5mm}\\
Let $\phi$ be an automorphic form of weight $\kappa=(\kappa_v)_{v\in\bmf{a}}$ and let $K_{\bmf{f}}$ be the stabilizer of $\phi$ in $\wGe_n(\A_{\bmf{f}})$. Let $\{\gamma_i\in\wGe_n(\A_{\bmf{f}})\}_{i=1}^{h}$ be the representative set of $\wGe_n(F)\lsla\wGe_n(\A)/K_{\bmf{f}}\wGe_n(\mb{R}\otimes F)$. Put $\Gamma_i=\iota_{\bmf{f}}^{-1}(\gamma_i^{-1}K_{\bmf{f}}\gamma_i)$, $\bm{i}_n=\sqrt{-1}\cdot1_n$.

We call $\phi$ a holomorphic automorphic form if there exists a family of holomorphic functions $\{f_i:\maf{H}_n^{\bmf{a}}\to\mb{C}\}_{i=1}^h$ such that $f_i||_{\kappa}g_{\infty}((\bm{i}_n)_{v\in\bmf{a}})=\phi(\gamma_ig_{\infty})$ for $g_{\infty}\in\wGe_n^+(\mb{R}\otimes F)$.
Then if $\phi$ is a holomorphic automorphic form, for each $i$, $f_i$ is a Hilbert-Hermitian modular form of weight $(\kappa_v)_{v\in\bmf{a}}$ with respect to the congruence subgroup $\Gamma_i$.

Since $\wGe_n(F)\wGe_n^+(\A)=\wGe_n(\A)$, we obtain a correspondence between holomorphic automorphic forms and tuples of Hilbert-Hermitian modular forms.
\end{Rem}

We put $\K_{n,v}=\{g\in \Ge_n(F_v)|g\bm{i}_n=\bm{i}_n\}$. This is a compact subgroup. We denote by $d_{\K_{n,v}}k_v$ the Haar measure on $\K_{n,v}$ such that $\mr{vol}(\K_{n,v})=1$.
Put $\bmf{H}_n=\{X\in \Mat_n(\mb{C})|X-X^*=0\}$. Then, $\maf{H}_n=\{X+\sqrt{-1}Y|\text{$X,Y\in\bmf{H}_n$ $Y>0$}\}$.
The Haar measure $dg_{v}$ on $\Ge(F_v)$ is normalized such that $dg_{v}/d_{\K_{n,v}}k_v=|\det Y|^{-2n}dXdY$ where $dX$, $dY$ are the Haar measures on $\bmf{H}_n$ such that $\vol(\bmf{H}_n/\bmf{H}_n\cap\Mat_n(\mb{Z}[\sqrt{-1}])=1$.

We put $\wK_{n,v}=\{g\in \wGe_n^+( F_v)|g\bm{i}_n=\bm{i}_n\}$. There is an exact sequence $1\to\K_{n,v}\to\wK_{n,v}\xrightarrow{\nu_n}\mb{R}_{>0}\to1$. Under $\K_{n,v}\lsla\wK_{n,v}\cong\mb{R}_{>0}$, let the quotient measure be $dt/t$.

To explain our main theorem, we define the Petersson inner product $\perd{F}{G}$ of for automorphic forms $F,G$ with the same unitary central character by
\begin{equation*}
\perd{F}{G}=\int_{Z(\A)\wGe_n(F)\lsla\wGe_n(\A)}F(g)\overline{G(g)}d\mu_{\Z\lsla\wGe_n}g.
\end{equation*}
Here we denote by $\mu_{Z\lsla\wGe_n}$ the quotient measure on $\Z(\mathbb{A})\wGe_n(F) \backslash \wGe_n(\mathbb{A})$ for the Tamagawa measures $\mu_{\wGe_n}$ on $\wGe_n(F)\lsla\wGe_n(\A)$ and the Tamagawa measure $\mu_{\Z}$ on $\Z(\A) \cong \A_E^{\times}$.

\subsection{Siegel Series and Degenerate Whittaker Functions}\label{subsec2-3}
Siegel series play a crucial role in computing the Ikeda--Yamana lift and its period.
We first recall the definition of Siegel series. Let $v\in\bmf{f}$.
\begin{defns}\label{def:siegelseries}
We define the Siegel series $b_v(B;s)$ for $B\in\dHer_n(F_v;\hana{O}_{E_v})$ by
\begin{equation*}
b_v(B;s)=\sum_{z\in\Her_n(F_v)/\Her_n(F_v;\hana{O}_{E_v})}\psi_v(\mr{tr}(Bz))\mu_v(z)^{-s}
\end{equation*}
where $\mu_v(z)=[z\hana{O}_{E_v}^n+\hana{O}_{E_v}^n:\hana{O}_{E_v}^n]^{1/2}$ for $z\in\Her_n(F_v)$.
If $B\notin\dHer_n(F_v;\hana{O}_{E_v})$, we put $b_v(B;s)=0$.
\end{defns}
Put
\begin{equation}
\varepsilon_v=\begin{cases}
1&\text{if $E_v=F_v\oplus F_v$,}\\
-1&\text{if $E_v$ is an inert extension of $F_v$,}\\
0 &\text{if $E_v$ is a ramified extension of $F_v$.}
\end{cases}\label{xi}
\end{equation}
By \cite{S1}, there exists a unique polynomial $F_v(B;X)$ such that
\begin{equation*}
b_v(B;s)=F_v(B;q_v^{-s})\prod_{i=0}^{[(n-1)/2]}(1-q_v^{2i-s})^{-1}\prod_{i=1}^{[n/2]}(1-\varepsilon_vq_v^{2i-1-s})^{-1}
\end{equation*}
for $B\in\dHer_n(F_v;\hana{O}_{E_v})$. Put $F_v(B;X)=0$ for $B\notin\dHer_n(F_v;\hana{O}_{E_v})$.
We then define a Laurent polynomial $\wF_v(B;X)$ as
\begin{equation*}
\wF_v(B;X)=X^{-e_v(B)}F_v(B;q_v^{-n}X^2)
\end{equation*}
where $e_v(B)=\ord_v(\maf{D}_{E_v/F_v}^{[n/2]}\det B)$ for $B\in\Her_n(F_v)^{nd}$. 
Put $\gamma_0(B)=(-1)^{[n/2]}\det B$.
The following equation is well known (see \cite[Corollary 3.2]{Ike1}):
\begin{align*}
F_v(B;q_v^{-2n}X^{-1})={}&\begin{cases}
\varepsilon_{E_v/F_v}(\gamma_0(B))(q_v^nX)^{-e_v(B)}F_v(B;X)&(\text{if $n$ is even})\\
(q_v^nX)^{-e_v(B)}F_v(B;X)&(\text{if $n$ is odd})
\end{cases}
\end{align*}
for $B\in\dHer_n(F_v;\hana{O}_{E_v})^{nd}$.
Therefore, we have
\begin{align*}
\wF_v(B;X^{-1})={}&\begin{cases}
\varepsilon_{E_v/F_v}(\gamma_0(B))\wF_v(B;X)&(\text{if $n$ is even})\\
\wF_v(B;X)&(\text{if $n$ is odd}).
\end{cases}
\end{align*}
This equation follows from the functional equation of Whittaker functions.
\\
\\

To reduce the global period problem to local computations, we need an explicit formula for the degenerate Whittaker functions.
We recall some facts about Whittaker functions.

Let $\psi_v$ be an unramified additive character on $F_v$.
We denote by $\psi_{v,B}$ the character on $\Her_n(F_v)$ such that $\psi_{v,B}(u)=\psi_v(\tr_{E/F}(\tr(Bu)))$ for $B\in\Her_n(F_v)$, $u\in\Her_n(F_v)$.
Let $\W_v(\psi_{v,B})$ be the space of the functions on $\wGe_n(F_v)$ satisfying the following conditions:
\begin{itemize}
\item $W(\n(u)g)=\psi_{v,B}(u)W(g)$ for $g\in\wGe_n(F_v)$ and $u\in\Her_n(F_v)$,
\item There exists a compact open subgroup $K$ of $\wGe_n(F_v)$ such that
 \begin{equation*}
     W(gk)=W(g)\text{ for $g\in\wGe_n(F_v)$ and for $k\in K$}.
 \end{equation*}
\end{itemize}

We assume that $B\in\Her_n(F_v)^{nd}$.
When $\Pi_v$ is a degenerate principal series representation, it is well known that 
\begin{equation*}
\dim\mr{Hom}_{\wGe_n}(\Pi_v,\W(\psi_{v,B}))\leqq1.
\end{equation*}
Hence it suffices to construct a non-zero degenerate Whittaker function.

Let $\alpha_{\wPe_n}^s$ be the character on $\wPe_n$ such that $\alpha_{\wPe_n}^s(\n(u)\m(h)\bm{d}(t))=\alpha_{E_v}^s(\det h)\alpha_{F_v}^{-ns}(t)$ for $h\in\GL_n(E_v)$, $u\in\Her_n(F_v)$ and $t\in F_v^\times$.
Let $\ind_{\wPe_n(F_v)}^{\wGe_n(F_v)}(\alpha_{\wPe_n}^s)$ be the parabolic induced representation, that is, the space of the smooth functions satisfying
\begin{equation*}
\left\{\text{smooth function }f:\wGe_n(F_v)\to \mb{C}\middle|\begin{matrix}f(\n(X)\m(A)\bm{d}(t)g)=\alpha_{E_v}^s(\det A)\alpha_{F_v}^{-ns}(t)f(g)\end{matrix}\right\}.
\end{equation*} 
For $\varphi\in\ind_{\wPe_n(F_v)}^{\wGe_n(F_v)}(\alpha_{\wPe_n}^s)$, we set
\begin{equation*}
    W(g_v; B, \varphi)=\int_{\Her_n(F_v)}\varphi(w_n\n(Z_v)g_v)\psi_{v,B}^{-1}(Z_v)dZ_v
\end{equation*}
where $dZ_v$ is a Haar measure on $\Her_n(F_v)$ determined by $\vol(\Her_n(F_v;\hana{O}_{E_v}),dZ_v)=1$.
For $\Re(s)>n$, the above integral is absolutely convergent and defines a Whittaker functional on $\ind_{\wPe_n(F_v)}^{\wGe_n(F_v)}(\alpha_{\wPe_n}^s)$.
We put $\varphi^{(s)}(g)=\alpha_{\wPe_n}^s(p)$ for $g=pk\in \wGe_n(F_v)$ with the Iwasawa decomposition $\wGe_n(F_v)=\wPe_n(F_v)\wK_{n,v}$.
The following lemma gives an explicit formula for its Fourier coefficients:
\begin{lems}\label{lem:whfunc1}
Let $B\in\Her_n(F_v)^{nd}$.
Then
\begin{equation*}
W(g_v; B, \varphi^{(s)})=\alpha_{E_v}(\det h)^{-s+n}\alpha_{F_v}^{-n^2}(t)b_v(t^{-1}B[h];2s)\psi_{v,B}(u)
\end{equation*}
where $g_v\in\n(u)\bm{d}(t)\m(h)\wK_{n,v}$ for $h\in\GL_n(E_v)$, $t\in F_v^\times$ and $u\in\Her_n(F_v)$.
\end{lems}
\begin{proof}
Since the function $\varphi^{(s)}$ is left $U_\Pe$-invariant and right $\wK_{n,v}$-invariant, it suffices to consider the case where $g_v=\bm{d}(t)\m(h)$ for $h\in\GL_n(E_v)$, $t\in F_v^\times$.
Then, we have
\begin{align*}
&\int_{\Her_n(F_v)}\varphi^{(s)}(J\n(z)\bm{d}(t)\m(h))\overline{\psi_{v,B}(z)}dz\\
={}&\int_{\Her_n(F_v)}\varphi^{(s)}(\n^{-}(-z)\bm{d}(t)\m(t\cdot1_n)\m((h^*)^{-1}))\overline{\psi_{v,B}(z)}dz\\
={}&\alpha_{E_v}^s(\det(t(h^*)^{-1}))\alpha_{F_v}^{-2ns}(t)\int_{\Her_n(F_v)}\varphi^{(s)}(\n^{-}(-tz[(h^*)^{-1}]))\overline{\psi_{v,B}(z)}dz\\
={}&\alpha_{E_v}^{-s}(\det(h))\int_{\Her_n(F_v)}\varphi^{(s)}(\n^{-}(-z))\overline{\psi_{v,B}(t^{-1}z[h^*])}\alpha_{F_v}^{-n^2}(t)\alpha_{E_v}^{-n}(\det(h^{-1}))dz\\
={}&\alpha_{E_v}^{-s+n}(\det(h))\alpha_{F_v}^{-n^2}(t)\int_{\Her_n(F_v)}\varphi^{(s)}(\n^{-}(-z))\overline{\psi_{v,t^{-1}B[h]}(z)}dz.
\end{align*}
By Iwasawa decomposition, for each $z\in\Her_n(F_v)$ there exist $a\in\GL_n(E_v)$ and $u\in\Her_n(F_v)$ such that $\n^-(-z)\in\m(a)\n(u)\wK_{n,v}$.
This means that $a^{-1}+ua^*z$, $ua^*$, $a^*z$ and $a^*$ belong to $\Mat_n(\hana{O}_{E_v})$. 
Then, we have $\varphi^{(s)}(\n^-(-z))=\alpha_{E_v}^s(\det a)$.
On the other hand, the following computation identifies $\mu_v(z)^{-2s}$ with $\varphi^{(s)}(\n^-(-z))$.
Since $a^*z$, $a^*\in \Mat_n(\hana{O}_{E_v})$, we have
\begin{align*}
\mu_v(z)^2={}&\int_{E_v^{n}}1_{z\hana{O}_{E_v}^n+\hana{O}_{E_v}^n}(r)dr\\
={}&\int_{E_v^{n}}1_{z\hana{O}_{E_v}^n+\hana{O}_{E_v}^n}((a^*)^{-1}r')dr'|\det(a^*)^{-1}|_{E_v}\\
={}&\int_{E_v^{n}}1_{a^*z\hana{O}_{E_v}^n+a^*\hana{O}_{E_v}^n}(r')dr'|\det(a^*)|_{E_v}^{-1}\\
\leqq&\int_{E_v^{n}}1_{\hana{O}_{E_v}^n}(r')dr'|\det(a^*)|_{E_v}^{-1}.
\end{align*}
On the other hand, since $(a^*)^{-1}+zau=(a^{-1}+ua^*z)^*$ and $au=(u^*a)^*$ belong to $\Mat_n(\hana{O}_{E_v})$, we have
\begin{equation*}
(a^*)^{-1}\hana{O}_{E_v}^n\subset zau\hana{O}_{E_v}^n+\hana{O}_{E_v}^n\subset z\hana{O}_{E_v}^n+\hana{O}_{E_v}^n
\end{equation*}
and
\begin{align*}    
\mu_v(z)^2\geqq&\int_{E_v^{n}}1_{(a^{*})^{-1}\hana{O}_{E_v}^n}(r)dr\\
={}&\int_{E_v^{n}}1_{\hana{O}_{E_v}^n}(a^*r)dr\\
={}&\int_{E_v^{n}}1_{\hana{O}_{E_v}^n}(r')dr'|\det(a^*)|_{E_v}^{-1}
\end{align*}
Therefore, $\mu_v(z)^{-2s}=\alpha_{E_v}^{s}(\det a^*)=\varphi^{(s)}(\n^-(-z))$ and we obtain 
\begin{equation*}
W(g_v; B, \varphi^{(s)})=\psi_{v,B}(u)|\det h|_{E_v}^{-s+n}|t|_v^{-n^2}b_v(t^{-1}B[h],2s).
\end{equation*}
\end{proof}
For $v\in\bmf{f}$, we put  
\begin{equation*}
W_{v,s}(\n(u)\bm{d}(t)\m(h)k;B)=\psi_{v,B}(u)|\det (t^{-1}B[h])|_v^{-s+n/2}\wF_v(t^{-1}B[h],q_v^{-s})
\end{equation*}
for $u\in\Her_n(F_v)$, $h\in\GL_n(E_v)$, and $t\in F_v^\times$.
Then, this function is a Whittaker function belonging to the Whittaker model of the degenerate principal series $\Ind_{\wPe_n}^{\wGe_n}(\alpha_{\wPe_n}^s)$.

We define the character $\epsilon_v:E_v^{\times}\to\mb{C}^{\times}$ by $\epsilon_v(x)=\frac{x}{\overline{x}}$ for $v\in\bmf{a}$.

For $B\in\Her_n(F)^+$, $l,\kappa\in\mb{Z}^{\bmf{a}}$ and $g\in\wGe_n^+(\mb{R}\otimes F)$, we define the archimedean Whittaker function $W_\infty^{l,\kappa}$ by
\begin{align*}
W_{\infty}^{l,\kappa}(g;B)={}&\prod_{v\in \bmf{a}}\left((\det B)^{l_v/2}\exp(2\pi\sqrt{-1}\tr(B\cdot g_v\bm{i}_n))\epsilon_v\left(\det g_v\right)^{\kappa_v/2}\right)\prod_{v\in\bmf{a}}j_v(g_v,\bm{i}_n)^{-l_v}.
\end{align*}
We define the Whittaker function $W_{\maf{s}}^{l,\kappa}$ by 
\begin{align*}
W_{\maf{s}}^{l,\kappa}(g;B)={}&\left(\prod_{v\in\bmf{f}}W_{v,\maf{s}_v}(g_v;B)\right)\times W_{\infty}^{l,\kappa}(g_{\infty};B)
\end{align*}
for $g=((g_v)_{v\in\bmf{f}},g_{\infty})\in\wGe_n^+(\A)=\wGe_n(\A_{\bmf{f}})\times\wGe_n^+(\mb{R}\otimes F)$, $\maf{s}=(\maf{s}_v)_{v\in\bmf{f}}$ and $l,\kappa\in\mb{Z}^{\bmf{a}}$.

%- Section -%
\section{Ikeda--Yamana Lift and Main Theorem}\label{sec3}

\subsection{Ikeda--Yamana Lift}\label{subsec3-1}
We first recall Yamana's construction of the Ikeda--Yamana lift in \cite{Y} and its explicit formula in the classical setting.
Throughout this section, we assume $n$ is odd.

Let $\pi \simeq\otimes'_v\pi_v$ be an irreducible automorphic representation of $\GL_2(\A)$ associated with the adelization of a Hilbert Hecke eigenform, and let $\omega_{\pi}$ be its central character.
Fix a Hecke character $\chi_\pi=\otimes_v\chi_v$ of $\A_E^\times$ whose restriction to $\mb{A}^\times_F$ equals $\omega_\pi$.

Let $V$ denote a $2n$-dimensional vector space over $E$, equipped with a split skew-Hermitian structure given by $\langle x, y \rangle_V = x^* J_n y$ for $x,y\in V\cong E^{2n}$.
Fix an $E$-basis $\{e_i,f_i\}_{i=1}^{n}$ consisting of isotropic vectors such that $\langle{e_i,f_j}\rangle_V=\delta_{i,j}$.
Let $V^{(i)}$ be the isotropic subspace of $V$ spanned by $e_1,f_2,e_3,f_4,\dots,e_{2i-1},f_{2i}$ and let $V_{(i)}$ be the isotropic subspace of $V$ spanned by $f_1,e_2,f_3,e_4,\cdots,f_{2i-1},e_{2i}$.
We define the parabolic subgroup $\wPe^{(2)}$ of $\wGe_n$ stabilizing the flag of isotropic subspaces $V^{(1)}\subset V^{(2)}\subset \dots\subset V^{((n-1)/2)}\subset V$ and the flag of isotropic subspaces $V_{(1)}\subset V_{(2)}\subset\dots\subset V_{((n-1)/2)}\subset V$.
Its standard Levi group of $\wPe^{(2)}$ is isomorphic to $\Res_{E/F}\GL_2^{(n-1)/2}\times\wGe_1$.
Put 
\begin{align*}
&J_n^{\chi_v}(\pi_v)=\Ind_{\wPe^{(2)}}^{\wGe_n}\delta_{\wPe^{(2)}}^{-1/4}\otimes \left(W(\pi_v[\hat{\chi_v}])^{\boxtimes(n-1)/2}\boxtimes(\chi_v^{-1}\boxtimes W(\pi_v))\right)
\end{align*}
where $\Ind$ denotes normalized induction.
Let $A_n^{\chi_v}(\pi_v)$ be the unique irreducible subrepresentation of $J_n^{\chi_v}(\pi_v)$.
We next define the $B$-th Shalika functional with respect to a character $\chi$ on $E_v^\times$.
\begin{defns}
    Let $\Pi$ be an admissible representation of $\wGe_n(F_v)$ and let $\chi$ be a character on $E_v^\times$.
    Let $B\in\Her_n(F_v)^{nd}$ and $w_B$ be a Whittaker functional. We say that $w_B$ is a $B$-th Shalika functional with respect to $\chi$ if it satisfies the following equivariance property:
\begin{equation*}
(w_B \circ \Pi)(m(A) d_n(\lambda_B(A))) = \chi(\lambda_B(A)^{-[n/2]} \det A) w_B
\end{equation*}
for all $A \in \mathrm{GL}_n(E_v)$ such that $B[A] = \lambda_B(A) B$ for some $\lambda_B(A) \in F_v^\times$.
For a single Shalika functional $w_{B_0}$ of $\Pi$ with respect to $\chi$, we obtain a family of Shalika functionals $w_B$ with respect to $\chi$ if 
    \begin{equation*}
        (w_B\circ\Pi)(\m(A)\bm{d}_n(\lambda_B(A)))=\chi(\lambda_B(A)^{-[n/2]}\det A)w_B
    \end{equation*}
    for any $(A,t)\in\GL_n(E_v)\times F_v^{\times}$ such that $t^{-1}B_0[A]=B$. This is well-defined since $w_{B_0}$ is a Shalika functional.\\
\end{defns}
The same definition applies to an admissible representation of $\wGe_n(\mb{A}_{\bmf{f}})$.
Yamana then constructed a family of Shalika functionals $\{\Slk_B\}_{B\in\Her_n(F)^+}$ on $A_n^{\chi_{\bmf{f}}}(\pi_{\bmf{f}})=\otimes'_{v\in\bmf{f}}A_n^{\chi_v}(\pi_v)$ and proved the following theorem:
\begin{props}[{\cite[Theorem 1.1]{Y}}]\label{YamanaMT}
The series
\begin{equation*}
\I_n(g;\phi)=\sum_{B\in \Her_n^+(F)}W_B^{\kappa+n-1,w}(g_\infty)\Slk_B(\Pi_{\bmf{f}}(g_{\bmf{f}})\phi)
\end{equation*}
defines a $\wGe_n(\A_{\bmf{f}})$-intertwining embedding from $A_n^{\chi_{\bmf{f}}}(\pi_{\bmf{f}})$ into the space of Hilbert-Hermitian cusp forms on $\Ge_n(\A_F)$ of weight $\kappa+n-1$ with the character $\otimes_{v\in\bmf{a}}\epsilon_{v}^{w_v}$ where $w_v=\frac{\kappa_v+n-1+l(\chi_v)}{2}$ and $w=(w_v)_{v\in\bmf{a}}$.
\end{props}

%adelic-classical

To compute the Petersson norm of $\I_n(g;\phi)$, we require an explicit formula for the Shalika functional on $A_n^{\chi_{\bmf{f}}}(\pi_{\bmf{f}})$.
We therefore specialize to the classical setting considered in \cite[Section 11]{Y}.\\

Let $\pi=\bigotimes'_v\pi_v$ be an irreducible cuspidal
automorphic representation of $\GL_2(\A_F)$. 
At each archimedean place $v$, assume that $\pi_v$ is the discrete series representation of weight $\kappa_v$. 
Then there is a tuple of Hilbert cusp forms $f=(f_i)_i$ of weight $(\kappa_v)_{v\in\bmf a}$ whose adelization generates $\pi$.

For each finite place $v$ at which $\pi_v$ is unramified,
choose $\maf{s}_v\in\mathbb C$ such that $\pi_v$ is the unique
spherical irreducible constituent of the normalized principal series
\begin{equation*}
\Ind_{B_{\GL_2}(F_v)}^{\GL_2(F_v)}
\left(
  \alpha_{F_v}^{\maf{s}_v}
  \otimes
  \alpha_{F_v}^{-\maf{s}_v}
\right).
\end{equation*}
We call the unordered pair
\begin{equation*}
\left\{
q_v^{\maf{s}_v},q_v^{-\maf{s}_v}
\right\}
\end{equation*}
the Satake parameters of $\pi_v$, or equivalently those of $f$ at $v$.
Put $S(f)=\{v\in\bmf{f}|\Re\maf{s}_v\ne0\}$.

We denote by $J_n(\chi_v,\mu_v)$ the normalized induced representation $\Ind_{\wPe_n}^{\wGe_n}(\chi_v\circ\det_{\Pe}\boxtimes\mu_v\circ\nu_n)$ for a character $\chi_v$ on $E_v^{\times}$ and for a character $\mu_v$ on $F_v^{\times}$, where $\det_{\Pe}$ is a character $\Pe\to\Res_{E/F}\mb{G}_m$ such that $\det_{\Pe}\m(A)=\det A$.
We denote by $A_n(\chi_v,\mu_v)$ the unique irreducible subrepresentation of $J_n(\chi_v,\mu_v)$.
For a character $\chi_v$ on $E_v^{\times}$, we denote by $\chi_{v,F}$ the restriction to $F_v^\times$.
Then Yamana constructed the intertwining operator $M_{\mr{Yam}}$ from $J_n((\mu_v\circ N_{E/F})\overline{\chi}_v^{-1},\mu_v^{-n}\chi_{v,F}^{(n+1)/2})$ to $J_n^{\chi_v}(\Ind_{B_{\GL_2}}^{\GL_2}(\mu_v\otimes \mu_v^{-1}\chi_{v,F}))$ such that $M_{\mr{Yam}}(h)(1_{2n})=h(1_{2n})$ in \cite[Proposition 6.2.]{Y}.
Here, set $\overline{\chi}_v=\chi_v\circ\rho_{E/F}$ where $\rho_{E/F}\in\Gal(E/F)$.
In our case, we put $\chi_\pi=1$. Then, we have 
\begin{equation*}
h_v(g)=\begin{cases}
    \alpha_{\wPe_n}^{\maf{s}_v+n/2}(p)&\text{if $g=pk\in\wPe_n(F_v)\wK_{n,v}$}\\
    0&\text{otherwise}
\end{cases}.
\end{equation*}
Set $\phi_{f,v}=M_{\mr{Yam}}(h_v)$ and $\phi_f=\otimes'_{v\in\bmf{f}}\phi_{f,v}$.
We define the classical realization of the Ikeda--Yamana lift $\I_n(f)$ by $\I_n(f)(g)=\I_n(g;\phi_f)$.
Put 
\begin{equation*}
\maf{w}_B^{\chi_v}(h)=\alpha_{F_v}^{n/2}(\det B)\prod_{j=1}^nL_v(j,\chi_v|_{F_v}\cdot\varepsilon_{E/F}^{n+j})\int_{\Her_n(F_v)}h(J_n\n_n(z))\overline{\psi_B(z)}dz
\end{equation*}
for $h\in I_n(\chi_v)$ where $I_n(\chi_v)=\{f|_{\Ge_n}|f\in J_n(\chi_v,\mu_v)\}$.
By \cite[Lemma 9.2]{Y}, we have
\begin{align*}
\maf{w}_B^{\alpha_{E_v}^{\maf{s}_v}}(h_v)={}&\alpha_{F_v}^{-2n\maf{s}_v}(\maf{d}_{F/\mb{Q},v})\alpha_{F_v}^{n/2}(\det B)F_v(B,q_v^{-2\maf{s}_v-n})\\
={}&\alpha_{F_v}^{-2n\maf{s}_v}(\maf{d}_{F/\mb{Q},v})\alpha_{F_v}^{n/2}(\det B)\wF_v(B,q_v^{-\maf{s}_v})(q_v^{-\maf{s}_v})^{-(n-1)\maf{f}_v/2+\ord_v(\det B)}\\
={}&\alpha_{F_v}^{-2n\maf{s}_v}(\maf{d}_{F/\mb{Q},v})q_v^{(n-1)\maf{s}_v\maf{f}_v/2}\alpha_{F_v}^{n/2+\maf{s}_v}(\det B)\wF_v(B,q_v^{-\maf{s}_v}).
\end{align*}

By \cite[Corollary 6.4]{Y}, we have
\begin{align*}
\Slk_{v,B}(\phi_{f,v})={}&\alpha_{F_v}^{-\maf{s}_v}((-1)^{(n-1)/2}\det B)\maf{w}_B^{|\cdot|_E^{\maf{s}_v}}(h_v)\\
={}&\alpha_{F_v}^{-2n\maf{s}_v}(\maf{d}_{F/\mb{Q},v})q_v^{(n-1)\maf{s}_v\maf{f}_v/2}\alpha_{F_v}^{n/2}(\det B)\wF_v(B,q_v^{-\maf{s}_v}).
\end{align*}
We put $C_{n,v}(f,B)=N_{F/\mb{Q}}(\maf{d}_{F/\mb{Q},v})^{-2n\maf{s}_v}q_v^{(n-1)\maf{s}_v\maf{f}_v/2}|\det B|_v^{\maf{s}_v}$.
Then $|C_{n,v}(f,B)|=1$ if $v\notin S(f)$ and we have
\begin{align*}
\Slk_{v,B}(\phi_{f,v})=C_{n,v}(f,B)W_{v,\maf{s}_v}(1_{2n};B).
\end{align*}

Since $f$ is of level 1, all $\kappa_v$ are even, $f$ satisfies the condition of the main theorem in \cite{BlD}. Thus, we have $S(f)=\emptyset$.
%Ramanujan conjecture for Hilbert modular forms. 
We put $C_n(f,B)=\prod_{v\in\bmf{f}}C_{n,v}(f,B)$ for $B\in\Her_n(F)^{nd}$.
The Duke--Imamoglu--Ikeda lift $\I_n(f)$ of the Hilbert modular form $f$ therefore has the Fourier expansion
\begin{align*}
\I_n(f)(g)={}&\sum_{B\in\Her_n(F)^+}C_n(f,B)W_{\maf{s}}^{n+\kappa-1,w}(g; B)
\end{align*}
and $\I_n(f)$ belongs to the space of the Hilbert-Hermitian cusp forms 
of weight $\kappa+n-1$ with character $\epsilon^{(\kappa+n-1)/2}$ defined in \cite{Y}.
\begin{Rem}
    The Ikeda–Yamana lift and the Duke–Imamoglu–Ikeda lift differ by a scalar multiple when $F=\mb{Q}$.
    In \cite[Section 11]{Y}, Yamana gives the Fourier coefficients of the Ikeda--Yamana lift as a function on hermitian upper half planes.
    Comparing these Fourier coefficients with those of the Duke--Imamoglu--Ikeda lift, one sees that the two lifts differ by a scalar factor. We compute this scalar explicitly in Section \ref{subsec6-2}.
\end{Rem}

\subsection{Main Theorem}\label{subsec3-2}
In \cite[Conjecture 17.5]{Ike1}, Ikeda conjectured the explicit formula for the Petersson norm of the Duke--Imamoglu--Ikeda lift.
This conjecture was proved in \cite{K1, K2, K3} with some modifications. 
 
For $f\in S_{2k}(SL_2(\mb{Z}))$, put 
\begin{align*}
\Gamma_{\mb{C}}(s)={}&2(2\pi)^{-s}\Gamma(s)\\
\wLmd(s,\varepsilon_{E/F}^i)={}&\Gamma_{\mb{C}}(s)L(s,\varepsilon_{E/F}^i)\\
\wLmd(s,f,\Ad,\varepsilon_{E/F}^i)={}&\Gamma_{\mb{C}}(s)\Gamma_{\mb{C}}(s+2k-n)L(s,f,\Ad,\varepsilon_{E/F}^i)\\
L(s,f,\Ad,\varepsilon_{E/F}^i)={}&\prod_{v\in \bmf{f}}\dfrac{1}{(1-\varepsilon_{E/F}^i(\varpi_v)q_v^{-2\maf{s}_v-s})(1-\varepsilon_{E/F}^i(\varpi_v)q_v^{-s})(1-\varepsilon_{E/F}^i(\varpi_v)q_v^{-2\overline{\maf{s}}_v-s})}
\end{align*}

Assume $n$ is odd.
\begin{thms}\cite[Theorem 2.2(2)]{K3}
Let $F=\mb{Q}$ and $f\in S_{2k}(SL_2(\mb{Z}))$. Let $F_f$ be the Duke--Imamoglu--Ikeda lift defined by \cite{Ike2}.
Then, we have
\begin{equation*}
\perd{F_f}{F_f}=2^{-2nk-(n-1)(n+2)}D_E^{(n-1)k+5(n-1)n/4}\prod_{i=1}^n\wLmd(i,f,\Ad,\varepsilon_{E/F}^{i-1})\prod_{i=2}^n\wLmd(i,\varepsilon_{E/F}^{i})
\end{equation*}
where this Petersson norm is defined on the Hermitian upper half planes as defined in \cite{K3}.
\end{thms}

For a tuple of Hilbert forms $f$ of weight $(\kappa_v)_{v\in\bmf{a}}$ with full level, we put 
\begin{align*}
\Gamma_{\mb{R}}(\sigma)={}&\pi^{-\sigma/2}\Gamma(\sigma/2),\\
\Lmd(s,\varepsilon_{E/F}^i)={}&L(s,\varepsilon_{E/F}^i)\Gamma_{\mb{R}}(s,\varepsilon_{E/F}^i)^{d_F},\\
\wLmd(s,\varepsilon_{E/F}^{i})={}&L(s,\varepsilon_{E/F}^i)\Gamma_{\mb{C}}(s)^{d_F},\\
\Lmd(s,f,\Ad,\varepsilon_{E/F}^i)={}&\prod_{v\in \bmf{a}}\left(\Gamma_{\mb{C}}(s+\kappa_v-1)\Gamma_{\mb{R}}(s,\varepsilon_{E/F}^i)\right)L(s,f,\Ad,\varepsilon_{E/F}^i),\\
\wLmd(s,f,\Ad,\varepsilon_{E/F}^i)={}&\Gamma_{\mb{C}}(s)^{d_F}\prod_{v\in \bmf{a}}\Gamma_{\mb{C}}(s+\kappa_v-1)L(s,f,\Ad,\varepsilon_{E/F}^i),\\
L(s,f,\Ad,\varepsilon_{E/F}^i)={}&\prod_{v\in
\bmf{f}}\dfrac{1}{(1-\varepsilon_{E/F}^i(\varpi_v)q_v^{-2\maf{s}_v-s})(1-\varepsilon_{E/F}^i(\varpi_v)q_v^{-s})(1-\varepsilon_{E/F}^i(\varpi_v)q_v^{-2\overline{\maf{s}}_v-s})}
\end{align*}
where $\{q_v^{\pm \maf{s}_v}\}$ for $v\in\bmf{f}$ are the Satake parameters of $f$ at $v$.

We put
\begin{equation*}
    \Gamma_{\mb{R}}(\sigma,\varepsilon_{E/F}^i)=\begin{cases}
        \Gamma_{\mb{R}}(\sigma)&(\text{if $\varepsilon_{E/F}^i=1$}),\\
        \Gamma_{\mb{R}}(\sigma+1)&(\text{if $\varepsilon_{E/F}^i\ne1$})
    \end{cases}
\end{equation*}

If $\varepsilon_{E/F}^i=1$, we write
\begin{align*}
    \wLmd(s,f,\Ad)={}&\wLmd(s,f,\Ad,\varepsilon_{E/F}^i)\\
L(s,f,\Ad)={}&L(s,f,\Ad,\varepsilon_{E/F}^i).
\end{align*}

Then our main theorem is as follows:
\begin{thms}\label{MT}
For a tuple of Hilbert eigenforms $f$ of weight $(\kappa_v)_{v\in\bmf{a}}$ and with full level, the Petersson norm of the Ikeda–Yamana lift $\mathcal{I}_n(f)$ is given by:  
\begin{align*}
&\perd{\I_n(f)}{\I_n(f)}\\
={}&2^{-n\kappa+(2n^2-n)d_F+1}|D_E|^{n(n-1)/4}|D_F|^{-n(4n-1)/2}\prod_{i=1}^n\Gamma_{\mb{C}}(i)^{d_F}\prod_{i=1}^n\dfrac{\Lmd(i,f,\Ad,\varepsilon_{E/F}^{i-1})}{\Lmd(n+i,\varepsilon_{E/F}^{i-1})}\\
={}&2^{-n\kappa+n^2d_F+1}|D_E|^{n(n-1)/4}|D_F|^{-n(4n-1)/2}\prod_{i=1}^n\dfrac{\Gamma_{\mb{C}}(n+i)^{d_F}}{\Gamma_{\mb{C}}(i)^{d_F}}\prod_{i=1}^n\dfrac{\wLmd(i,f,\Ad,\varepsilon_{E/F}^{i-1})}{\wLmd(n+i,\varepsilon_{E/F}^{i-1})}.
\end{align*}
where the Petersson norm is defined using the Tamagawa measure introduced above and $\kappa=\sum_{v\in\bmf{a}}\kappa_v$.
\end{thms}

Although this formula appears to differ from the previously known formula, we show in Section \ref{subsec6-2} that the two formulas are equivalent.
To prove the main theorem, we define the Rankin--Selberg integrals in the next section.

\section{Rankin--Selberg Integrals}\label{sec4}
In this section, we define the Rankin--Selberg integral and prove a proposition giving an explicit formula for the period, which is used to prove the main theorem.

For $v\in\bmf{f}$, let $\varphi_v^{(s)}$ be the $\wK_{n,v}$-invariant function in $\ind_{\wPe_n(F_v)}^{\wGe_n(F_v)}(\alpha_{\wPe}^s)$ for $v\in\bmf{f}$ such that $\varphi_v^{(s)}(1_{2n})=1$. 
For $v\in\bmf{a}$, let $\varphi_v^{(s)}$ be the $\K_{n,v}$-invariant function in $\ind_{\wPe_n(F_v)}^{\wGe_n(F_v)}(\alpha_{\wPe}^s)$ such that $\varphi_v^{(s)}(1_2n)=1$.
Set $\varphi^{(s)}=\prod_{v\in\bmf{v}}\varphi_v^{(s)}$. 
\begin{defns}
Let $\phi_1$ and $\phi_2$ be automorphic cusp forms on $\wGe_n(F)\lsla\wGe_n(\A)$ with the same central unitary character.
We define the Rankin--Selberg integral for $\phi_1$ and $\phi_2$ by
\begin{equation}
\R(s;\phi_1,\,\phi_2)=\int_{\Z(\A)\wGe_n(F)\lsla \wGe_n(\A)}\phi_1(g)\overline{\phi_2(g)}E_{\wPe}(g, s)d\mu_{\Z\lsla\wGe_n}(g),
\end{equation}
where
\begin{equation}
E_{\wPe}(g, s):=\sum_{\gamma\in \wPe(F)\lsla \wGe_n(F)}\varphi^{(s)}(\gamma g).\label{Eisenstein}
\end{equation}
We write $\R(s;\phi)=\R(s;\phi,\phi)$ for short.
\end{defns}

Let $D_F$ (resp.$D_E$) be the discriminant of $F$ (resp.$E$) over $\mb{Q}$. Let $\varepsilon_{E/F}$ be the Hecke character of $\A^{\times}$ attached to $E/F$. We denote by $\Lmd(s,\varepsilon_{E/F})$ the complete L-function and by $\Lmd(s,id)=\zeta_{F}(s)$ the complete Dedekind zeta function of $F$.

Put 
\begin{equation}
R_n=\dfrac{\Res_{\sigma=1}\zeta_{F}(\sigma)\prod_{k=2}^{n}\Lmd(k,\varepsilon_{E/F}^{k+1})}{2|D_F|^{n/2}|D_E|^{n(n-1)/4}\prod_{l=1}^{n}\Lmd(2n-l+1,\varepsilon_{E/F}^{l-1})}.\label{residue}
\end{equation}

\begin{props}\label{prop:period1}
Let $E(g,s)$ be defined by \eqref{Eisenstein}. Then, we have
\begin{equation*}
\Res_{s=n}E_{\wPe}(g,s)=R_n.
\end{equation*}
\end{props}
\begin{proof}
   It follows from \cite[Theorem 19.7]{S1} by identifying $E(\bm{d}_n(\nu_n(g)^{-1})g,s)$ with the Eisenstein series on $\Ge_n(F)\lsla\Ge_n(\A)$.
\end{proof}

By Proposition \ref{prop:period1}, we obtain the relation between the residue of the Rankin--Selberg integral and the period.
\begin{cor}\label{cor:period1}
Let $\phi_1$, $\phi_2$ be automorphic cusp forms on $\wGe_n(F)\lsla\wGe_n(\A)$ with the same central unitary character.
Then, $\R(s;\phi_1,\phi_2)$ has a meromorphic continuation to the whole $s$-plane and has at most a simple pole at $s=n$ with residue $R_n\perd{\phi_1}{\phi_2}$ where $R_n$ is defined in \eqref{residue}.
\end{cor}

When $n$ is odd and $\phi=\phi_1=\phi_2=\I_n(f)$ for a Hecke eigenform $f$ in $S(\Gamma)$, $\R(s;\phi)$ admits an Euler product. We now unfold the integral and reduce its evaluation to local computations.

Let $\phi(g)$ be a holomorphic automorphic cusp form and put 
\begin{equation*}
    a_\phi(g;B)=\int_{U(F)\lsla U(\A)}\phi(ug)\psi_B^{-1}(u)du.
\end{equation*}
Since $\phi$ is a cusp form, $a_{\phi}(g,B)=0$ unless $B$ is totally positive definite.
Then, we have the Fourier expansion
\begin{equation*}
    \phi(g)=\sum_{B\in\Her_n(F)^+}a_\phi(g;B)
\end{equation*}
We can compute $\R(s;\phi)$ by the Fourier expansion of $\phi$.
By unfolding the Eisenstein series, we have
\begin{align*}
\R(s;\phi)={}&\int_{\Z(\A)\wGe_n(F)\lsla\wGe_n(\A)}|\phi(g)|^2E_{\wPe}(g,s)d\mu_{\Z\lsla\wGe}(g)\\
={}&\int_{\Z(\A)\wPe_n(F)\lsla\wGe_n(\A)}|\phi(g)|^2\varphi^{(s)}(g)d\mu'_{\Z\lsla\wGe}(g).
\end{align*}
Here $\mu'_{\Z\lsla\wGe_n}$ is the quotient measure on $\Z(\A)\wPe(F)\lsla\wGe_n(\A)$ for the adelic Haar measure defined in \ref{ap:Tamagawa}, the canonical counting measure on $\wPe_n(F)$ and the Tamagawa measure of $\Z$.
Since $\wPe_n(F)\lsla\wGe_n(\A)\cong \wPe_n^+(F)\lsla\wGe_n^+(\A)$, we have
\begin{align*}
\R(s;\phi)
={}&\int_{\Z(\A)\wPe_n^+(F)\lsla\wGe_n^+(\A)}|\phi(g)|^2\varphi^{(s)}(g)d\mu'_{\Z\lsla\wGe}(g).
\end{align*}

 Then we have
\begin{align*}
\R(s;\phi)
={}&\int_{\Z(\A)\wPe_n^+(F)\lsla\wGe_n^+(\A)}\sum_{A,B\in\Her_n(F)^+}a_{\phi}(g;A)\overline{a_{\phi}(g;B)}\varphi^{(s)}(g) d\mu'_{\Z\lsla\wGe_n}g\\
={}&\sum_{A,B\in\Her_n(F)^+}\int_{\Z(\A)U_{\Pe}(\A)\wM_n^+(F)\lsla \wGe_n^+(\A)}\\&\times\left(\int_{U_\Pe(F)\lsla U_\Pe(\A)}a_{\phi}(ug;A)\overline{a_{\phi}(ug;B)}du\right)\varphi^{(s)}(g) d\mu_{\Z U_{\wPe}\lsla\wGe_n}(g)
\end{align*}
where $\mu_{\Z U_{\wPe}\lsla\wGe_n}$ is the quotient measure on $\Z(\A)U_{\Pe}(\A)\wM(F)\lsla \wGe(\A)$ for the measure $\mu'_{\Z\lsla\wGe_n}$ and the Tamagawa measure of $U$.

Since $a_{\phi}(ug;B)=\psi_B(u)a_{\phi}(g;B)$, we have
\begin{equation*}
\int_{U_\Pe(F)\lsla U_\Pe(\A)}a_{\phi}(ug;A)\overline{a_{\phi}(ug;B)}du=\begin{cases}
|a_{\phi}(g;A)|^2&\text{if $A=B$},\\
0&\text{otherwise}.
\end{cases}
\end{equation*}
Thus, we have
\begin{align*}
\R(s;\phi)={}&\sum_{B\in\Her_n(F)^+}\int_{\Z(\A)U_{\Pe}(\A)\wM_n^+(F)\lsla \wGe_n^+(\A)}\left|a_{\phi}\left(g;B\right)\right|^2\varphi^{(s)}(g)d\mu_{\Z U_{\wPe}\lsla\wGe_n}(g).
\end{align*}
 
We also put $B[m]=t^{-1}B[h]$ for $m=\m(h)\bm{d}(t)\in\wM$. 
Put $\wM_B=\{m\in\wM|B[m]=B\}$ for $B\in\Her_n(F)^{nd}$ 
We remark that $\wM_B(R)\subset\wM_n^+(R)$ for $R\in\{F\}\cup\{F_v|v\in\bmf{a}\}\cup\{\A\}$. 

Then, we have
\begin{align*}
\R(s;\phi)={}&\sum_{B\in\Her_n(F)^+/\wM^+(F)}\sum_{\gamma\in \wM_B(F)\lsla \wM^+(F),}\\
&\int_{\Z(\A)U_{\Pe}(\A)\wM^+(F)\lsla \wGe^+(\A)}\left|a_{\phi}\left(
g;B[\gamma]\right)\right|^2\varphi^{(s)}(g)d\mu_{\Z U_{\wPe}\lsla\wGe_n}(g).
\end{align*}
Then, since $\varphi^{(s)}$ is left $\wPe(F)$-invariant and $a_{\phi}(g;B[\gamma])=a_{\phi}(\m(\gamma)g;B)$ for any $\gamma\in \wM^+(F)$, we have
\begin{align*}
&\R(s;\phi)\\
={}&\sum_{B\in\Her_n(F)^+/\wM^+(F)}\,\sum_{\gamma\in \wM_B(F)\lsla \wM^+(F)}\\
&\int_{\Z(\A)U_{\Pe}(\A)\wM^+(F)\lsla \wGe_n^+(\A)}\left|a_{\phi}\left(
\m(\gamma)g;B\right)\right|^2\varphi^{(s)}(\m(\gamma)g) d\mu_{\Z U_{\wPe}\lsla\wGe_n}(g).
\end{align*}
By unfolding, we have 
\begin{align*}
\R(s;\phi)={}&\sum_{B\in\Her_n(F)^+/\wM^+(F)}\int_{\Z(\A)U_{\Pe}(\A)\wM_B(F)\lsla \wGe_n^+(\A)}\left|a_{\phi}\left(
g;B\right)\right|^2\varphi^{(s)}(g)d\mu^B_{\Z U_{\wPe}\lsla\wGe_n}(g)
\end{align*}
where $\mu^B_{\Z U_{\wPe}\lsla\wGe_n}$ is the measure on $\Z(\A)U_{\wPe}(\A)\M_B(F)\lsla\wGe_n(\A)$ determined by unfolding.

Let $\mu_{\wM_B}$ be the Tamagawa measure of $\wM_B$. If $\gamma\in \wM_B(\A)$, we have $\alpha_{\wPe}^s(\det \gamma)=1$ and $a_{\phi}(\m_n(\gamma h);B)=a_{\phi}(\m_n(h);B)$ for any $h\in \wM(\A)$.

Therefore, we have
\begin{align*}
\R(s;\phi)&=\sum_{B\in\Her_n(F)^+/\wM^+(F)}\mr{vol}(\Z(\A)\wM_B(F)\lsla\wM_B(\A))\\
&\times\int_{U_{\Pe}(\A)\wM_B(\A)\lsla \wGe_n^+(\A)}\left|a_{\phi}\left(
g;B\right)\right|^2\varphi^{(s)}(g) d\mu_{U\wM_B\lsla \wGe_n}(g).
\end{align*}
where $\mu_{U\wM_B\lsla \wGe_n}$ is the quotient measure on $U_{\Pe}(\A)\wM_B(\A)\lsla \wGe_n^+(\A)$ for $\mu^B_{\Z U_{\wPe}\lsla\wGe_n(\A)}$ and the Tamagawa measure $\mu_{\wM_B}$ of $\wM_B$.

Put $\wK=\wK_{n,\bmf{f}}\times\K_{n,\infty}$. By Iwasawa decomposition, $\Ge(\A)=\Pe(\A)(\Ge(\A)\cap\wK)$.
Hence for any $g\in\wGe_n(\A)$, there exists $p\in\Pe(\A_F),k\in \wK$ such that $\bm{d}(\nu(g))^{-1}g=pk$.
Since $\bm{d}_n(\nu_n(g))\in\wPe(\A)$ for any $g\in\wGe_n(\A)$, we have $\wGe(\A)=\wPe(\A)\wK$.

Let $dk$ be the measure on $\wK_{n,\bmf{f}}\times\K_{n,\infty}$ such that $\int_{\wK_{n,\bmf{f}}\times\K_{n,\infty}}dk=1$.

Then, there exists $c_0$ such that $dg=c_0dpdk$ where $dg,dp$ are the Tamagawa measures. 

By \ref{ap:Tamagawa}, we have
\begin{equation*}
    c_0=2^{n(n+1)d_F}|D_E|^{-n(n-1)/4}|D_F|^{-n^2-n/2}\lim_{s\to1}(s-1)L(s,1)\dfrac{\prod_{i=2}^nL(i,\varepsilon_{E/F}^{i-1})}{\prod_{i=n+1}^{2n}L(i,\varepsilon_{E/F}^i)}\times\prod_{i=1}^n\Gamma_{\mb{C}}(i)^{-d_F}
\end{equation*}

We next compute the Tamagawa number $\tau(\Z\lsla\wM_B)$ of $\Z\lsla\wM_B$.
Put 
\begin{align*}
SU_B={}&\{g\in\Res_{E/F}\SL_n|B[g]=B\}, \\
Z_B={}&\{z\cdot1_n|z\in\Res_{E/F}\GL_1,z^n=1,z\overline{z}=1\}
\end{align*}
for $B\in\Her_n(F)^{nd}$. Since $n$ is odd, we have the exact sequence
\begin{equation*}
    1\to Z_B\to SU_B\to \Z\lsla\wM_B\to1.
\end{equation*}

Then we can compute the Tamagawa number of $\Z\lsla\wM_B$ by \cite[Theorem 10.1]{SanJ}.
\begin{props}\cite[Theorem 10.1]{SanJ}\label{prop:SanJ}
    Let $G$ be a connected reductive group defined over the number field $F$. 
    Let $\mr{Pic}(G)$ be the Picard group of $G$ and let $\Sha(G)$ be the Tate-Shafarevich group of $G$.
    Then, we have the formula 
    \begin{equation*}
        \tau(G)=\dfrac{\#\mr{Pic}(G)}{\#\Sha(G)}.
    \end{equation*}
\end{props}
Since $SU_B$ is connected and semisimple and $Z_B$ is finite and central, the above exact sequence shows that
$\Z\lsla\wM_B\simeq SU_B/Z_B$ is connected and semisimple, and hence reductive.
In addition, we have $\tau(SU_B)=1$ by \cite[Theorem 3]{Kott}. 
Thus, we have $\tau(\Z\lsla\wM_B)=\frac{\#\mr{Pic}(\Z\lsla\wM_B)}{\#\Sha(\Z\lsla\wM_B)}$.

 Since $SU_B$ is a semisimple algebraic group and we have the exact sequence $1\to Z_B\to SU_B\to\Z\lsla\wM_B\to1$, we have 
 \begin{equation*}
     \mr{Pic}(\Z\lsla\wM_B)\cong X^*(Z_B)^{\Gal(\overline{F}/F)}
 \end{equation*}
 where we denote by $X^*(Z_B)$ the algebraic character group of $Z_B$.
Since $Z_B(\overline{F})\cong\{x\in\overline{F}|x^n=1\}$ and $F$ is a totally real field, $\Hom(Z_B(\overline{F}),\overline{F}^{\times})^{\Gal(\overline{F}/F)}=\{1\}$ if $n$ is odd and $\#\mr{Pic}(\Z\lsla\wM_B)=1$.
Since $B$ is a non-degenerate hermitian matrix, we have the Hasse principle for $U_B$ by the Landherr theorem.
By the weak approximation theorem for $SU_B$, we obtain the Hasse principle for $SU_B$. 
Then, by the exact sequence $1\to Z_B\to SU_B\to\Z\lsla\wM_B\to1$ and the Poitou-Tate duality for $Z_B$, we obtain $\#\Sha(\Z\lsla\wM_B)=1$.
Therefore, we have
\begin{equation*}
    \tau(\Z\lsla\wM_B)=\dfrac{\#\mr{Pic}(\Z\lsla\wM_B)}{\#\Sha(\Z\lsla\wM_B)}=1
\end{equation*}
by \cite[Theorem 10.1]{SanJ}.

By the definition of the Ikeda--Yamana lift for Hilbert-Hermitian modular forms, we have
\begin{align*}
a_{\I_n(f)}(g;B)={}&C_{n}(f,B)W_{\maf{s}_v}^{\kappa+n-1,(\kappa+n-1)/2}(g;B)\\
={}&\prod_{v\in\bmf{f}}C_{n,v}(f,B)W_{v,\maf{s}_v}(g_v;B)W_{\infty}^{\kappa+n-1,(\kappa+n-1)/2}(g_{\infty};B).
\end{align*}
We define the quotient measure $\mu_{\wM_BU\lsla\wPe}$ on $U(\A)\wM_B(\A)\lsla\wPe(\A)$ for the measure $\mu_{\wM_B,\A}\mu_{U,\A}$ on $\wM_B(\A)U(\A)$ and the measure $\mu_{\wPe,\A}$ on $\wPe(\A)$. 
Here we define these adelic measures $\mu_{\wM_B,\A}$, $\mu_{U,\A}$ and $\mu_{\wPe,\A}$ in \ref{ap:Tamagawa}.

We also put the quotient measure $\mu_{\wM_B\lsla\wM}$ on $\wM_B(\A)\lsla\wM(\A)$ for the measure $\mu_{\wM_B,\A}$ on $\wM_B(\A)$ and the measure $\mu_{\wM,\A}$ on $\wM(\A)$.
Here, the adelic measure $\mu_{\wM,\A}$ is also defined in the \ref{ap:Tamagawa}.

By Iwasawa decomposition, we obtain
\begin{align*}
&\R(s;\I_n(f))\\
={}&c_0\sum_{B\in\Her_n(F)^+/\wM^+(F)}\int_{U_{\Pe}(\A)\wM_B(\A)\lsla\wPe_n^+(\A)}\left|a_{\I_n(f)}\left(
p;B\right)\right|^2\varphi^{(s)}(p) d\mu_{\wM_BU\lsla\wPe}(p).
\end{align*}

Since $\mu_{\wPe_n,\A}$ is a left-invariant measure on $\wPe(F)\lsla\wPe_n(\A)$, we have 
\begin{equation*}
\mu_{\wM_BU\lsla\wPe}=\delta_{\wPe}(m)^{-1}\mu_{\wM_B\lsla\wM}=\alpha_{\A}^{-n}(\det B[m])\mu_{\wM_B\lsla\wM},
\end{equation*}
and
\begin{align*}
&\R(s;\I_n(f))\\
={}&c_0\sum_{B\in \Her_n(F)^+/\wM^+(F)}
\int_{\wM_B(\A)\lsla \wM^+(\A)}\alpha_{\A}^{s-n}(\det(B[m]))\\
&\times
\prod_{v\in\bmf{f}}\left|C_{n,v}(f,B)W_{v,\maf{s}_v}(m_v;B)\right|^2\times|W_{\infty}^{\kappa+n-1,(\kappa+n-1)/2}(m_{\infty};B)|^2 d\mu_{\wM_B\lsla\wM}(m)\\
={}&c_0\sum_{B\in \Her_n(F)^+/\wM^+(F)}
\int_{\wM_B(\A)\lsla \wM^+(\A)}\alpha_{\A}^{s}(\det(B[m]))\\
&\times
\prod_{v\in\bmf{f}}\left|\wF_{v}(B[m_v],q_v^{-\maf{s}_v})\right|^2\times\prod_{v\in\bmf{a}}\left|\exp(-2\pi\tr(B[m_{v}]))\right|^2(\alpha_{F_v}^{\kappa_v-1})(\det B[m_v])d\mu_{\wM_B\lsla\wM}(m).
\end{align*}

By \ref{ap:locden}, we have
\begin{equation*}
\mu_{\wM_B\lsla\wM}=\dfrac{2^{n(n+1)d_F/2}|D_E|^{-n(n-1)/4}|D_F|^{-n/2}}{\lim_{s\to1}(s-1)L(s,1)}\prod_{v\in\bmf{f}}(1-q_v^{-1})^{-1}\mu^*_{\Her_n,v}\prod_{v\in\bmf{a}}\mu_{Q,v}
\end{equation*}
where $\mu_{\Her_n,v}$ is the Haar measure on $\Her_n(F_v)$ such that $\int_{\Her_n(F_v;\hana{O}_{E_v})}d\mu_{\Her_n,v}=1$, $d\mu^*_{\Her_n,v}(Z)=\alpha_{F_v}^{-n}(\det Z)d\mu_{\Her_n,v}$ for $v\in\bmf{f}$, put $Q_v=\{b:\text{upper triangular matrix in $\GL_n(E_v)$}|b_{ii}\in \mb{R}_{>0}\}$ and $\mu_{Q,v}$ is the right Haar measure on $Q_v$ defined in \ref{ap:locden}.

Put \begin{align}
    c={}&c_0\dfrac{2^{n(n+1)d_F/2}|D_E|^{-n(n-1)/4}|D_F|^{-n/2}}{\lim_{s\to1}(s-1)L(s,1)}\nonumber\\
    ={}&2^{3n(n+1)d_F/2}|D_E|^{-n(n-1)/2}|D_F|^{-n^2-n}\dfrac{\prod_{i=2}^nL(i,\varepsilon_{E/F}^{i-1})}{\prod_{i=n+1}^{2n}L(i,\varepsilon_{E/F}^i)}\times\prod_{i=1}^n\Gamma_{\mb{C}}(i)^{-d_F}.\label{eq:constantofRS}
\end{align}

For each $v\in\bmf{a}$, let $dh$ denote the left Haar measure on $Q_v$ uniquely determined by $\delta_{B_{\GL_n}}(h)dh=d\mu_{Q,v}(h)$.
Here $\delta_{B_{\GL_n}}$ is the modulus character on the Borel subgroup of $\GL_n$ consisting of upper triangular matrices.
With this choice of measures, we have
\begin{align*}
&\R(s;\I_n(f))\\
={}&
c\prod_{v\in\bmf{f}}\left((1-q_v^{-1})^{-1}\int_{\Her_n(F_v)^{nd}}\left|\wF_{v}(Z_v,q_v^{-\maf{s}_v})\right|^2\alpha_{F_v}^{s}(\det Z_v)d^*Z_v\right)\\
&\prod_{v\in\bmf{a}}\int_{Q_v}\left|\exp(-2\pi\tr(h^*h))\right|^2\alpha_{F_v}^{s+\kappa_v-1}(\det(hh^*))\delta_{B_{\GL_n}}(h)dh.
\end{align*}

For $v\in\bmf{a}$, we have
\begin{align*}
&\int_{Q_v}\left|\exp(-2\pi\tr(h^*h))\right|^2|\det(h^*h)|^{s+\kappa_v-1}\delta_{B_{\GL_n}}(h)dh\\
={}&\prod_{i=1}^n\int_{\mb{R}_{>0}}\exp(-4\pi t_i^2)|t_i^2|^{s+\kappa_v-1-i+1}d^{\times}t_i\times\left(\int_{\mb{R}^2}\exp(-4\pi (x^2+y^2))dxdy\right)^{n(n-1)/2}.\\
={}&\prod_{i=1}^n\dfrac{2}{2^{s+\kappa_v-i+2}\cdot(2\pi)^{s+\kappa_v-i}}\Gamma(s+\kappa_v-i)\times2^{-n(n-1)}\\
={}&2^{-ns-n\kappa_v-n(n+1)/2}\prod_{i=1}^n\Gamma_{\mb{C}}(s+\kappa_v-i).
\end{align*}
\begin{defns}[the local Rankin--Selberg integral]

    For $X,Y\in\mb{C}$ and for $s\in\mb{C}$ such that $\Re s\gg 0$, we define the local Rankin--Selberg integral by
    \begin{equation*}        
    \hat{\R}_v(X,Y,q_v^{n-s})=\int_{\Her_n(F_v)^{nd}}\wF_v(Z_v;X)\wF_v(Z_v;Y)|\det Z_v|^{s}d^*Z_v.
    \end{equation*}
\end{defns}
\begin{Rem}
The notation $\hat{\R}(X, Y; q_v^{n-s})$ is chosen to anticipate the fact that this integral is later shown to be a rational function of $X$, $Y$, and $q_v^{-s}$ through its connection to zeta functions.
\end{Rem}
Then for $\kappa=\sum_{v\in \bmf{a}}\kappa_v$, we have
\begin{align}
&\R(s;\I_n(f))\nonumber\\
={}&c\prod_{v\in\bmf{f}}(1-q_v^{-1})^{-1}\hat{\R}_v(q_v^{-\maf{s}_v},q_v^{-\overline{\maf{s}_v}},q_v^{n-s})\times\left(2^{-nd_Fs-n(n+1)d_F/2-n\kappa}\prod_{v\in \bmf{a}}\prod_{i=1}^n\Gamma_{\mb{C}}(s+\kappa_v-i)\right)\label{eq:locfactorofRS}
\end{align} 
using the local Rankin--Selberg integral.

%- Section -%
\section{Local Computation}\label{sec5}

\subsection{Local Densities and Siegel Series}\label{subsec5-1}
For the local computations needed in the proof of the main theorem, we review the relation between 
local densities and local Siegel series and define a series corresponding to the primitive local density.

Throughout this section, we fix a finite place $v\in\bmf{f}$, write $F=F_v$ and $E=E_v$ and omit the subscript $v$ from the local notation.
Put $|x|=\alpha_F(x)$ for $x\in F^{\times}$ and $|x|_E=\alpha_E(x)$ for $x\in E^{\times}$.
We remark that $E$ is either a quadratic extension of $F$ or $E\cong F\oplus F$.

We first define the local density $\alpha$, the primitive local density $\beta$, and the $\K_n^0$-orbit sets $\Omega$ and $\Omega'$, and recall their basic properties.
For $n\leqq m$ and $X\in \Mat_{m,n}(\hana{O}_E)$, we call $X$ a primitive matrix if there exists $Y\in \Mat_{m,(m-n)}(\hana{O}_E)$ such that $(XY)\in\K_m^0$.

\begin{defprops}\label{def:locden1}
Let $A\in\dHer_m(F;\hana{O}_E), B\in\dHer_n(F;\hana{O}_E)$ and $m\geqq n$. For an integer $e$ and a compact open subset $S\subset \Mat_{m,n}(E)$, put
\begin{align*}
\hana{A}_e(A,B;S)={}&\left\{X\in S\middle|A[X]-B\in\maf{p}^e\dHer_n(F;\hana{O}_E)\right\},\\
\hana{A}_\infty(A,B;S)={}&\left\{X\in S\middle|A[X]=B\right\}\\
\overline{\hana{A}}_e(A,B;S)={}&\left\{X\in \Mat_{m,n}(\hana{O}_E)/\maf{p}^e\Mat_{m,n}(\hana{O}_E)|X+\maf{p}^e\Mat_{m,n}(\hana{O}_E)\subset \hana{A}_e(A,B;S)\right\}.
\end{align*}
Then, $q^{e(-2mn+n^2)}\#\overline{\hana{A}}_e(A,B;S)$ converges if $A$ is non-degenerate. We put 
\begin{equation*}
\alpha(A,B;S)=\lim_{e\to\infty}q^{e(-2mn+n^2)}\#\overline{\hana{A}}_e(A,B;S)
\end{equation*}
If $S=\Mat_{m,n}(\hana{O}_E)$, then we write $\hana{A}_e(A,B)=\hana{A}_e(A,B;S)$, $\overline{\hana{A}_e(A,B)}=\overline{\hana{A}_e(A,B;S)}$ and $\alpha(A,B)=\alpha(A,B;S)$ and we call $\alpha(A,B)$ the local density representing $B$ by $A$.
If $B=A$ we write $\alpha(A)=\alpha(A,A)$.\\
If $S=\Mat_{m,n}(\hana{O}_E)^{\prm}$ where $\Mat_{m,n}(\hana{O}_E)^{\prm}$ is the subset of primitive matrices in $\Mat_{m,n}(\hana{O}_E)$, we write $\hana{B}_e(A,B)=\hana{A}_e(A,B;S)$, $\overline{\hana{B}}_e(A,B)=\overline{\hana{A}}_e(A,B;S)$ and $\beta(A,B)=\alpha(A,B;S)$ and we call $\beta(A,B)$ the primitive local density representing $B$ by $A$.
If $B=A$, we write $\beta(A)=\beta(A,B)$.
\end{defprops}
\begin{proof}
Let $S$ be a compact open subset of $\Mat_{m,n}(E)$.
Then there exists a sufficiently large integer $e_S$ such that for any $s\in S$, $s+\maf{p}^{e_S}\Mat_{m,n}(\hana{O}_E)$ is contained in $S$.
Let $A\in\dHer_m(F;\hana{O}_E)^{nd}$ and $B\in\dHer_n(F;\hana{O}_E)$. Let $a$ be a sufficiently large integer such that $S\subset \maf{p}^{-a}\Mat_{m,n}(\hana{O}_E)$.
Then there exists an integer $b$ such that $\{A[X]-B|X\in S\}\subset \maf{p}^{b}\dHer_n(F;\hana{O}_E)$.
This implies that $\int_{\maf{p}^{-e}\Her_n(F;\hana{O}_E)}\psi(\tr(Y(A[X]-B)))dY$ is independent of $e$ if $e$ is a sufficiently large integer and $X\in S$.
Since $1_{\dHer_n(F;\hana{O}_E)}(Z)=q^{en^2}\int_{\maf{q}^{-e}\Her_n(F;\hana{O}_F)}\psi_{Y}(Z)dY$, for any sufficiently large integer $e$, we have
\begin{align*}
q^{-e(2mn-n^2)}\#\overline{\hana{A}}_e(A,B;S)=\int_{\Mat_{m,n}(E)}\bm{1}_S(X)\int_{\maf{p}^{-e}\Her_n(F;\hana{O}_E)}\psi_{Y}(A[X]-B)dYdX
\end{align*}
where $\bm{1}_T$ is the characteristic function of $T$ for the subset $T\subset \Mat_{m,n}(E)$.
Since there exists an integer $b$ such that $\{A[X]-B|X\in S\}\subset \maf{p}^{b}\dHer_n(F;\hana{O}_E)$, $q^{e(-2mn+n^2)}\#\overline{\hana{A}}_e(A,B;S)$ is independent of $e$ if $e$ is sufficiently large.
Thus, $q^{e(-2mn+n^2)}\#\overline{\hana{A}}_e(A,B;S)$ converges and $\alpha(A,B;S)=q^{e_0(-2mn+n^2)}\#\overline{\hana{A}}_{e_0}(A,B;S)$ for some integer $e_0$.
\end{proof}

We first relate between the local density $\alpha(A,B)$ and the primitive local density $\beta(A,B)$ for $A\in\Her_n(F;\hana{O}_E)$ and $B\in\Her_m(F;\hana{O}_E)$.
For this purpose, we need to define the M\"{o}bius function on $\Mat_n(E)$.
When $E$ is not split, for $0\leq i\leq r$, we put 
\begin{align*}
\hana{D}_{r,i}=\K_{r}^0(1_{r-i}\perp\varpi1_i)\K_{r}^0\text{ ($0\leq i\leq r$)}
\end{align*}
and the function $\mt{m}:\Mat_r(\hana{O}_E)\to\mb{C}$ such that
\begin{equation*}
\mt{m}(W)=\begin{cases}
(-1)^iq_E^{i(i-1)/2}&\text{($W\in\D_{r,i},0\leqq i\leqq r$)}\\
0&\mr{otherwise.}
\end{cases}
\end{equation*}

When $E$ is split, for $i=(i_1,i_2)$ such that $0\leqq i_1\leqq r$ and $0\leq i_2\leq r$, we put 
\begin{eqnarray*}
\hana{D}_{r,i}=\K_{r}^0(1_{r-i_1}\perp\varpi1_{i_1},1_{r-i_2}\perp \varpi1_{i_2})\K_{r}^0
\end{eqnarray*}
and the function $\mt{m}:\Mat_r(\hana{O}_E)\to\mb{C}$ such that
\begin{equation*}
\mt{m}(W)=\begin{cases}
(-1)^{i_1+i_2}q^{i_1(i_1-1)/2+i_2(i_2-1)/2}&\text{($W\in\D_{r,i}$, $0\leqq i_1\leqq r$, $0\leqq i_2\leqq r$,)}\\
0&\mr{otherwise.}
\end{cases}
\end{equation*}

We put $S^{\prm}=S\cap \Mat_{m,n}(\hana{O}_E)^\prm$ for $S\subset \Mat_{m,n}(\hana{O}_E)$.

\begin{lems}\label{lem:locden1}
Let $S$ be a subset of $\Mat_{m,n}(E)$ that satisfies $S\K_n^0=S$.
Let $A\in\dHer_m(F;\hana{O}_E)^{nd}$ and $B\in\dHer_n(F;\hana{O}_E)$. For $W\in \Mat_n(\hana{O}_E)^{nd}$, we put $S_{W}^{\prm}=(SW^{-1})^\prm$ and $\widehat{S_{W}}=(SW^{-1})\cap \Mat_{m,n}(\hana{O}_E)$. Then we have 
\begin{align*}
&\alpha(A,B;S)=\sum_{W\in\K_n^0\lsla \Mat_n(\hana{O}_E)^{nd}}\alpha(A,B[W^{-1}];S_W^\prm)q^{-(m-n)\ord(N_{E/F}(\det W))},\\
&\alpha(A,B;S^{\prm})\\
={}&\sum_{i=0}^n\sum_{W\in\K_n^0\lsla \D_{n,i}}\mt{m}(W)\alpha(A,B[W^{-1}];\widehat{S_W})q^{-(m-n)\ord(N_{E/F}(\det W))}
\intertext{ if $E$ is not split over $F$, and}
&\alpha(A,B;S^{\prm})\\
={}&\sum_{i_1=0}^n\sum_{i_2=0}^n\sum_{W\in\K_n^0\lsla \D_{n,(i_1,i_2)}}\mt{m}(W)\alpha(A,B[W^{-1}];\widehat{S_W})q^{-(m-n)\ord(N_{E/F}(\det W))}
\end{align*}
if $E$ is split over $F$.

In particular,
\begin{align*}
\alpha(A,B)={}&\sum_{W\in\K_n^0\lsla \Mat_n(\hana{O}_E)^{nd}}\beta(A,B[W^{-1}])q^{-(m-n)\ord(N_{E/F}(\det W))}
\end{align*}
and
\begin{align*}
\beta(A,B)={}&\sum_{i=0}^n(-1)^iq^{i(i-1)/2}\sum_{W\in\K_n^0\lsla \D_{n,i}}\alpha(A,B[W^{-1}])q^{-(m-n)\ord(N_{E/F}(\det W))}
\intertext{ if $E$ is not split over $F$, and}
\beta(A,B)
={}&\sum_{i_1=0}^n\sum_{i_2=0}^n\sum_{W\in\K_n^0\lsla \D_{n,(i_1,i_2)}}\mt{m}(W)\alpha(A,B[W^{-1}])q^{-(m-n)\ord(N_{E/F}(\det W))}
\end{align*}
if $E$ is split over $F$.
\end{lems}
\begin{proof}
We assume that $E$ is not split over $F$.
We put $\bm{1}_{m,n}=\bm{1}_{\Mat_{m,n}(\hana{O}_E)}$ , $\bm{1}_{m,n}^\prm=\bm{1}_{\Mat_{m,n}(\hana{O}_E)^\prm}$ and let $\varphi_S$ be $(\bm{1}_S\cdot\bm{1}_{m,n}^\prm)(X)$.
Since $\Mat_{m,n}(\hana{O}_E)\setminus(\bigsqcup_{W\in\K_n^0\lsla \Mat_n(\hana{O}_E)^{nd}}\Mat_{m,n}(\hana{O}_E)^{\prm}W)$ is a null set with respect to $dX$, we have
\begin{align*}
&\alpha(A,B;S)\\
={}&\lim_{e\to\infty}\int_{\maf{p}^{-e}\Her_n(F;\hana{O}_E)}\sum_{W\in\K_n^0\lsla \Mat_n(\hana{O}_E)^{nd}}\int_{\Mat_{m,n}(E)}\bm{1}_{S}(X)\cdot\bm{1}_{m,n}^{\prm}(XW^{-1})\psi_{Y}(A[X]-B)dXdY\\
={}&\lim_{e\to\infty}\sum_{W\in\K_n^0\lsla \Mat_n(\hana{O}_E)^{nd}}\int_{\maf{p}^{-e}\Her_n(F;\hana{O}_E)[(W^*)^{-1}]}\int_{\Mat_{m,n}(E)}\varphi_{SW^{-1}}(Z)\\
&\times\psi_{Y'}(A[Z]-B[W^{-1}])dZdY'|\det W|_E^{m-n}\\
={}&\sum_{W\in\K_n^0\lsla \Mat_n(\hana{O}_E)^{nd}}\alpha(A,B[W^{-1}];(SW^{-1})^{\prm})q^{-(m-n)\ord(N_{E/F}(\det W))}.
\end{align*}
On the other hand, when $\varphi$ is a right $\K_n$-invariant function on $\Mat_n(E)$ whose support is compact and $\widetilde{\varphi}(X)=\sum_{W\in\K_n\lsla \Mat_n(\hana{O}_E)^{nd}}\varphi(XW^{-1})$, we have
\begin{equation*}
\sum_{W\in\K_n^0\lsla \Mat_n(\hana{O}_E)^{nd}}\mt{m}(W)\widetilde{\varphi}(XW^{-1})=\varphi(X)
\end{equation*}
by the $q$-binomial theorem.

When $\varphi=\bm{1}_{m,n}^\prm$, $\widetilde{\varphi}$ is the characteristic function of $\Mat_n(\hana{O}_E)^{nd}$. Thus, if $e$ is sufficiently large, we have
\begin{align*}
&\alpha(A,B;S^{\prm})\\
={}&\lim_{e\to\infty}\int_{\maf{p}^{-e}\Her_n(F;\hana{O}_E)}\int_{\Mat_{m,n}(E)}\bm{1}_S(X)\bm{1}_{m,n}^\prm(X)\psi_{Y}(A[X]-B)dXdY\\
={}&\lim_{e\to\infty}\sum_{W\in\K_n^0\lsla \Mat_n(\hana{O}_E)^{nd}}\mt{m}(W)\int_{\maf{p}^{-e}\Her_n(F;\hana{O}_E)}\int_{\Mat_{m,n}(E)}\bm{1}_{SW^{-1}}(Z)\bm{1}_{m,n}(Z)\\
&\times\psi_{Y}(A[ZW]-B)|\det W|_E^{m}dZdY\\
={}&\sum_{W\in\K_n^0\lsla \Mat_n(\hana{O}_E)^{nd}}\mt{m}(W)\lim_{e\to\infty}\int_{\maf{p}^{-e}\Her_n(F;\hana{O}_E)[(W^*)^{-1}]}\int_{\Mat_{m,n}(E)}\bm{1}_{SW^{-1}}(Z)\\
&\times\bm{1}_{m,n}(Z)\psi_{Y'}(A[Z]-B[W^{-1}])dZdY'|\det W|_E^{m-n}.
\end{align*}
Therefore, we have
\begin{align*}
&\alpha(A,B;S^{\prm})\\
={}&\sum_{W\in\K_n^0\lsla \Mat_n(\hana{O}_E)^{nd}}\mt{m}(W)\alpha(A,B[W^{-1}];(SW^{-1})\cap \Mat_{m,n}(\hana{O}_E))q^{-(m-n)\ord(N_{E/F}(\det W))}.
\end{align*}
In particular, if $S=\Mat_{m,n}(\hana{O}_E)$, 
\begin{equation*}
(SW^{-1})^\prm=\Mat_{m,n}(\hana{O}_E)^\prm\text{ and }\Mat_{m,n}(\hana{O}_E)=(SW^{-1})\cap \Mat_{m,n}(\hana{O}_E).
\end{equation*} 
Thus, we obtain the above equations.
When $E$ is split over $F$, we can prove similarly.
\end{proof}

We second relate among the local density $\alpha(A,B)$, the number $\#\Omega(A,B)/\K_m^0$ and $\K_m^0\lsla\Omega'(A,B)$ for $A,B\in\Her_m(F;\hana{O}_E)$.

\begin{defns}\label{def:locden2}
For $A, B\in\dHer_m(F;\hana{O}_E)$ and a compact subset $S$ of $\Mat_m(E)$, we put
\begin{align*}
\Omega(A,B;S)={}&\{W\in S|A[W]\in B[\K_m^0]\}\\
\Omega'(A,B;S)={}&\{W\in S^{nd}|B[W^{-1}]\in A[\K_m^0]\}
\end{align*}
and if $S=\Mat_m(\hana{O}_E)$, we write $\Omega(A,B;S)=\Omega(A,B)$ and $\Omega'(A,B;S)=\Omega'(A,B)$.
\end{defns}

\begin{lems}\label{lem:locden2}
Assume $S=\K_m^0 S\K_m^0$ and $S\cap\K_m^0\ne\emptyset$. Let $A,B$ be elements of $\dHer_m(F;\hana{O}_E)^{nd}$.
Then we have that
\begin{itemize}
\item[(1)]$\dfrac{\alpha(A,B;S)}{\alpha(A)}=\#(\Omega(A,B;S)/\K_m^0)q^{-m(\ord(\det B)-\ord(\det A))}$,
\item[(2)]$\dfrac{\alpha(A,B;S)}{\alpha(A)}=\#(\K_m^0\lsla\Omega'(A,B;S))$.
\end{itemize}
\end{lems}
\begin{proof}
We can prove this in a similar way as in \cite[Lemma 4.1.3]{K1} and \cite[Lemma 5.1.1]{K2}.
\end{proof}

To obtain explicit formulas for these densities, we set
\begin{align*}
\Theta_{2k}={}&
\begin{cases}
1_{2k}&(\text{$E$ is inert over $F$ or $E=F\oplus F$}),\\
\begin{pmatrix}
0&\varpi_E^{-\maf{f}}\\
\overline{\varpi_E}^{-\maf{f}}&0
\end{pmatrix}
^{\perp k}
&(\text{$E$ is ramified over $F$})
\end{cases}
\end{align*}
Then $\Theta_{2k}$ is called a regular element in \cite{S1}.
We put $\Theta_{2k+1}=1_{2k+1}$ if $E$ is not ramified over $F$.
We compute the local density and the primitive local density concerning direct sum Hermitian matrices.
\begin{lems}\label{lem:locden3}
Let $m_1$, $m$, $n_1$, $n$ be integers such that $n_1\leqq m_1\leqq m$, $n_1\leqq n\leqq m$ and $m_1, n_1$ are even if $E$ is ramified over $F$.
Then for $A\in\dHer_{m-m_1}(F;\hana{O}_E)^{nd}$ and $B\in\dHer_{n-n_1}(F;\hana{O}_E)$, we have
\begin{align*}
\beta(\Theta_{m_1}\perp A,\Theta_{n_1}\perp B)={}&\beta(\Theta_{m_1}\perp A,\Theta_{n_1})\beta(\Theta_{m_1-n_1}\perp A,B).
\intertext{Moreover, if $B$ is non-degenerate,}
\alpha(\Theta_{m_1}\perp A,\Theta_{n_1})={}&\beta(\Theta_{m_1}\perp A,\Theta_{n_1})\\
\alpha(\Theta_{m_1}\perp A,\Theta_{n_1}\perp B)={}&\alpha(\Theta_{m_1}\perp A,\Theta_{n_1})\alpha(\Theta_{m_1-n_1}\perp A,B)
\end{align*}
\end{lems}

\begin{lems}\label{lem:locden4}
\begin{itemize}
\item[(a)]Assume that $E$ is not ramified over $F$.
Let $m_1$, $m$, $n_1$, $n_2$ be integers such that $2n_1\leqq m_1\leqq m$.
Then for $A\in\dHer_{m-m_1}(F;\hana{O}_E)$, $B_1\in\dHer_{n_1}(F;\hana{O}_E)$ and $B_2\in\dHer_{n_2}(F;\hana{O}_E)$, we have
\begin{align*}
\beta(1_{m_1}\perp A,B_1\perp B_2)={}&\beta(1_{m_1}\perp A,B_1)\beta((-B_1)\perp1_{m_1-2n_1}\perp A,B_2).
\end{align*}
\item[(b)]Assume that $E$ is ramified over $F$.
Let $m_1$, $m$, $n_1$, $n_2$ be integers such that $2n_1\leqq2m_1\leqq m$.
Then for $A\in\dHer_{m-2m_1}(F;\hana{O}_E)$, $B_1\in\dHer_{n_1}(F;\hana{O}_E)$ and $B_2\in\dHer_{n_2}(F;\hana{O}_E)$, we have
\begin{align*}
\beta(\Theta_{2m_1}\perp A,B_1\perp B_2)={}&\beta(\Theta_{2m_1}\perp A,B_1)\beta((-B_1)\perp\Theta_{2m_1-2n_1}\perp A,B_2).
\end{align*}
\end{itemize}
\end{lems}

\begin{Rem}
    The proofs of Lemmas \ref{lem:locden3} and \ref{lem:locden4} follow the argument of \cite[Lemma 4.1.6, Lemma 4.1.7]{K1}. 
    The same argument also covers $p=2$ by \cite[Lemma 8.1]{F}.
\end{Rem}

To compute the special value of the Siegel series, we compute the primitive density representing the Hermitian matrix $B$ by $\Theta_{2k}$.
We put 
\begin{equation*}
\Her_{m,*}=\begin{cases}
\maf{p}\dHer_m(F;\hana{O}_E)&\text{if $E$ is not ramified over $F$,}\\
\left\{(a_{ij})_{1\leqq i,j\leqq m}\in\Her_m(F)\middle|
\begin{matrix}
\text{$a_{ij}\in\maf{q}\maf{d}_{E/F}^{-1}$ if $i\ne j$},\\
\text{$a_{ii}\in\hana{O}$}
\end{matrix}
\right\}&\text{if $E$ is ramified over $F$.}
\end{cases}\label{eq:qsemiint}
\end{equation*}
Note that $\Theta_{2k}\notin\Her_{2k,*}$ for every $k\geqq 1$.

The following proposition gives the required formulas for the primitive local densities.
\begin{props}\label{prop:locden1}
\begin{itemize}
\item[(1)]Suppose that $E$ is inert over $F$. If $b\in\hana{O}_F$ we have
\begin{equation*}
\beta(1_m,\varpi b)=(1-(-q)^{-m})(1-(-q)^{-m+1}).
\end{equation*}
If $b\in\hana{O}_F^\times$, we have
\begin{align*}
\alpha(1_{2m},b)=\beta(1_{2m},b)={}&1-q^{-2m},\\
\alpha(1_{2m-1},b)=\beta(1_{2m-1},b)={}&1+q^{-2m+1}.
\end{align*}
\item[(2)]Assume that $E$ is ramified over $F$. If $B\in\Her_{l,*}^{nd}$ with $l\leqq2$, we have
\begin{align*}
\beta(\Theta_{2m},B)=\prod_{i=0}^{l-1}(1-q^{-2m+2i}).
\intertext{If $B=\Theta_{2}
$, we have}
\beta(\Theta_{2m},B)=1-q^{-2m}.
\end{align*}
\end{itemize}
\end{props}
\begin{proof}
We can prove (1) and (2) when $B\in\Her_{1,*}^{nd}=\maf{q}_E\maf{d}_{E/F}^{-1}\cap F^{\times}$ or $B=\Theta_2$ by \cite[Theorem 8.2.1]{F} and the explicit formula between the modular lattices and local density as in \cite{K1}.
Thus, we should prove the equation of $\hana{B}_e(\Theta_{2m},B)$ for $B\in\Her_{2,*}^{nd}$.
We put $B=
\begin{pmatrix}
b_{11}&b_{12}\\
\overline{b}_{12}&b_{22}
\end{pmatrix}
$ and for $b\in\hana{O}$ we put $\Theta(b)=\begin{pmatrix}
b&\varpi_E^{-\maf{f}}\\
\overline{\varpi}_E^{-\maf{f}}&0
\end{pmatrix}$, and $\Theta_{2m}(b)=\Theta(b)\perp\Theta_{2(m-1)}$. Then there exists $k_b\in \GL_{2m}(\hana{O}_E)$ such that $\Theta_{2m}[k_b]=\Theta_{2m}(b)$.
Therefore, we have $\beta(\Theta_{2m},A)=\beta(\Theta_{2m}(b),A)$ for any $b\in\hana{O}$ and any Hermitian matrices $A$.\\
Considering the decomposition of $\hana{B}_e(\Theta_{2m}(b_{11}),B)$, we obtain
\begin{align*}
\hana{B}_e(\Theta_{2m}(b_{11}),B)=
\bigsqcup_{W_1\in\overline{\hana{B}}_e(\Theta_{2m}(b_{11}),b_{11})}\hana{A}_e(\Theta_{2m}(b_{11}),B;S_{W_1}^{\prm})
\end{align*}
where 
\begin{equation*}
S_{W_1}=\{W\in \Mat_{2m,2}(\hana{O}_E)|W\equiv\begin{pmatrix}
W_1&W_2
\end{pmatrix}\mod\maf{p}^e\Mat_{2m,2}(\hana{O}_E)\text{ for some matrix $W_2$}\}.
\end{equation*}
For $W_1\in\overline{\hana{B}}_e(\Theta_{2m}(b_{11}),b_{11})$, we fix the primitive matrix $Z\in\Mat_{2m,2m-1}(\hana{O}_E)$ such that 
\begin{align*}
&\text{$\begin{pmatrix}
W_1&Z
\end{pmatrix}$ is primitive}\\
&W_1^*\Theta_{2m}(b_{11})Z\equiv
\begin{pmatrix}
\varpi_E^{-\maf{f}}&0
\end{pmatrix}
\mod\maf{p}^e\maf{q}^{-\maf{f}}\\
&\Theta_{2m}(b_{11})[Z]\equiv 0\perp\Theta_{2m-2}\mod\maf{p}^e\dHer_{2m-1}(F;\hana{O}_E).
\end{align*}
We put $e=\begin{pmatrix}
1\\
0
\end{pmatrix}\in \Mat_{2m,1}(\hana{O}_E)$ and $Z_{W_1}=\begin{pmatrix}
W_1&Z
\end{pmatrix}$.
Then we have $Z_{W_1}\in\GL_{2m}(\hana{O}_E)$ and
\begin{equation*}
Z_{W_1}\hana{A}_e(\Theta_{2m}(b_{11}),B;S_e^{\prm})=\hana{A}_e(\Theta_{2m}(b_{11}),B;S_{W_1}^{\prm}).
\end{equation*}
This gives $\#\overline{\hana{A}}_e(\Theta_{2m}(b_{11}),B;S_e^{\prm})=\#\overline{\hana{A}}_e(\Theta_{2m}(b_{11}),B;S_{W_1}^{\prm})$.
Therefore, it suffices to compute $\#\overline{\hana{A}}_e(\Theta_{2m}(b_{11}),B;S_e^{\prm})$.
Let 
\begin{equation*}
\begin{pmatrix}
1&a\\
0&b\\
0&X
\end{pmatrix}
\in\hana{A}_e(\Theta_{2m}(b_{11}),B;S_e^{\prm}).
\end{equation*} Then we have
\begin{align*}
b\equiv&\varpi_E^{\maf{f}}b_{12}-\varpi_E^{\maf{f}}ab_{11}\mod\maf{p}^e\maf{q}^{\maf{f}}\\
b_{22}\equiv&\overline{a}b_{12}+a\overline{b}_{12}-\overline{a}b_{11}a+\Theta_{2m-2}[X]\mod\maf{p}^e.
\end{align*}
This means that $b$ is completely determined by $a$ and $B$ modulo $\maf{p}^e\hana{O}_E$.\\
Assume $B\in\Her_{2,*}^{nd}$.
Then we have $b_{12}\varpi_E^{\maf{f}}-\varpi_E^{\maf{f}}ab_{11}\in\maf{q}$ and this implies that $X$ must be a primitive matrix.
\\
Therefore, there exists a bijection between $\overline{\hana{A}}_e(\Theta_{2m}(b_{11}),B;S_e^{\prm})$ and the multi-set 
\begin{equation*}
\bigsqcup_{a\in\hana{O}_E/\maf{p}^e\hana{O}_E}\overline{\hana{B}}_e(\Theta_{2m-2},b_{22}-b_{12}\overline{a}-\overline{b}_{12}a+b_{11}N_{E/F}(a)).
\end{equation*}
and this induces
\begin{equation*}
\#\overline{\hana{A}}_e(\Theta_{2m}(b_{11}),B;S_e^{\prm})=\sum_{a\in\hana{O}_E/\maf{p}^e\hana{O}_E}\#\overline{\hana{B}}_e(\Theta_{2m-2},b_{22}-b_{12}\overline{a}-\overline{b}_{12}a+b_{11}N_{E/F}(a)).
\end{equation*}
By the case $l=1$, for sufficiently large $e$, we have
\begin{equation*}
q^{-(4m-5)e}\#\overline{\hana{A}}_e(\Theta_{2m}(b_{11}),B;S_e^{\prm})=q^{2e}\cdot (1-q^{-2m+2})
\end{equation*}
and
\begin{align*}
\#\overline{\hana{B}}_e(\Theta_{2m}(b_{11}),B)={}&q^{(4m-1)e}(1-q^{-2m})\cdot q^{(4m-3)e}(1-q^{-2m+2})\\
={}&q^{(8m-4)e}(1-q^{-2m})(1-q^{-2m+2}).
\end{align*}
\end{proof}
\begin{lems}\label{lem:locden5}
Let $B\in\Her_{n,*}^{nd}$ and let $k$ be an integer such that $2k\geqq n$.
Then we have the following equations:
\begin{equation*}
\beta(\Theta_{2k},B)=\begin{cases}\prod_{i=0}^{2n-1}(1-(-1)^iq^{-2k+i})&\text{if $E$ is inert over $F$,}\\
\prod_{i=0}^{n-1}(1-q^{-2k+2i})&\text{if $E$ is ramified over $F$,}\\
\prod_{i=0}^{2n-1}(1-q^{-2k+i})&\text{if $E=F\oplus F$}
\end{cases}
\end{equation*}
\end{lems}
\begin{proof}
By Lemma \ref{lem:locden3}, Lemma \ref{lem:locden4} and Proposition \ref{prop:locden1}, we can prove this by induction in a similar way as in \cite[Lemma 4.10]{K1}.
\end{proof}

\begin{lems}\label{lem:locden6}
Let $m$ and $n$ be integers such that $m\geqq n$ and $m$, $n$ are even if $E$ is ramified over $F$.
Then we have
\begin{equation*}
\alpha(\Theta_m,\Theta_n)=\beta(\Theta_m,\Theta_n)=
\begin{cases}
\prod_{i=0}^{n-1}(1-(-q)^{-m+i})&\text{if $E$ is inert over $F$},\\
\prod_{i=0}^{n-1}(1-q^{-m+i})&\text{if $E=F\oplus F$},\\
\prod_{i=0}^{n/2-1}(1-q^{-m+2i})&\text{if $E$ is ramified over $F$}.
\end{cases}
\end{equation*}
\end{lems}
\begin{proof}
We can prove this similarly in \cite[Lemma 4.1.11]{K1}.
\end{proof}

\begin{lems}\label{lem:locden7}
Let $m,n$ be positive integers with $n\leqq m$.
If $E$ is ramified, we assume $n$ is even. Then for any $A\in\Her_{m-n,*}^{nd}$, we have
\begin{align*}
\dfrac{\alpha(\Theta_n\perp A)}{\alpha(A)}={}&\begin{cases}\prod_{i=1}^{n}(1-(-q)^{-i})&\text{if $E$ is inert over $F$}\\
\prod_{i=1}^{n}(1-q^{-i})&\text{if $E=F\oplus F$}\\
\prod_{i=1}^{n/2}(1-q^{-2i})&\text{if $E$ is ramified over $F$}
\end{cases}
\end{align*}
\end{lems}
\begin{proof}
We can prove this similarly in \cite[Lemma 4.2.1]{K1}.
\end{proof}
\begin{lems}\label{lem:locden8}
Let $k,m,r$ be non-negative integers such that $r\leqq m\leqq2k$. If $E$ is ramified over $F$, we assume that $m-r$ is even.
Then, for any $A\in\Her_{r,*}^{nd}$, we have
\begin{align*}
\beta(\Theta_{2k},\Theta_{m-r}\perp A)={}&\begin{cases}
\prod_{i=0}^{m+r-1}(1-(-q)^{-2k+i})&\text{if $E$ is inert over $F$,}\\
\prod_{i=0}^{m+r-1}(1-q^{-2k+i})&\text{if $E=F\oplus F$,}\\
\prod_{i=0}^{(m+r)/2-1}(1-q^{-2k+2i})&\text{if $E$ is ramified over $F$.}
\end{cases}
\end{align*}
\end{lems}
\begin{proof}
We can prove this similarly in \cite[Lemma 4.2.2]{K1}.
\end{proof}

These identities yield a decomposition formula for the polynomial $F(A, X)$ by the following lemma proved by \cite[Lemma 14.8 (1)]{S1}.
\begin{lems}\cite[Lemma14.8(1)]{S1}\label{lem:locden9}
For any $A\in\dHer_n(F;\hana{O}_E)^{nd}$ and any integer $k\geqq n$, we have
$b(A,2k)=\alpha(\Theta_{2k},A)$
\end{lems}

On the other hand, for $A\in\dHer_n(F;\hana{O}_E)$ we put 
\begin{eqnarray*}
G(A;X)&={}&\sum_{W\in \K_n^0\lsla \Mat_n(\hana{O}_E)^{nd}}(Xq^n)^{\ord(N_{E/F}(\det W))}\mt{m}(W) F(A[W^{-1}];X)\\
\wG(A;X,t)&={}&\sum_{W\in\K_n^0\lsla \Mat_n(\hana{O}_E)^{nd}}t^{\ord(N_{E/F}(\det W))}\mt{m}(W)\wF(A[W^{-1}];X)
\end{eqnarray*}
Then, the polynomial $G(A;X)$ has explicit values at $X=q^{-2k}$ for $A\in\dHer_n(F;\hana{O}_E)$ and an integer $k\geqq n$ and give the important decomposition of $F(A;X)$ by Lemma \ref{lem:locden1} and Lemma \ref{lem:locden9} as follows: 
\begin{lems}\label{lem:locden10}
For any $A\in\dHer_n(F;\hana{O}_E)^{nd}$ and any integer $k\geqq n$, we have
\begin{equation*}
\beta(\Theta_{2k},A)=G(A;q^{-2k})\prod_{i=0}^{[(n-1)/2]}(1-q^{2i-2k})\prod_{i=1}^{[n/2]}(1-\varepsilon q^{2i-1-2k})
\end{equation*}
\end{lems}
\begin{proof}
This lemma immediately follows from Lemma \ref{lem:locden1} and Lemma \ref{lem:locden9}.
\end{proof}

\begin{props}\label{prop:locden2}
\begin{itemize}
\item[(1)] Assume that $E$ is unramified over $F$. For $B\in\Her_{r,*}^{nd}$, we have
\begin{equation*}
G(1_{m-r}\perp B;X)=\prod_{j=0}^{r-1}(1-(\varepsilon q)^{m+j}X)
\end{equation*}
\item[(2)] Assume that $E$ is ramified over $F$ and $m-r$ is even.
For $B\in\Her_{r,*}^{nd}$, we have
\begin{equation*}
G(\Theta_{m-r}\perp B;X)=\prod_{j=0}^{[r/2]-1}(1-q^{2j+2[(m+1)/2]}X)
\end{equation*}
\end{itemize}
\end{props}
\begin{proof}
We can similarly prove this in \cite[Corollary of Lemma 4.2.2]{K1}.
\end{proof}

\begin{props}\label{prop:locden3}
For $A\in\dHer_n(F;\hana{O}_E)^{nd}$, we have 
\begin{equation*}
F(A;q^{-x})=\int_{\GL_n(E)}1_{\Mat_n(\hana{O}_E)}(h)G(A[h^{-1}];q^{-x})|\det h|_E^{n+x}dh
\end{equation*}
for any $\Re x>0$.
Moreover, we have 
\begin{equation*}
\wF(A;X)=X^{-[n/2]\ord(\maf{D}_{E/F})}\sum_{B\in\dHer_n(F;\hana{O}_E)/\K_n^0}\dfrac{\alpha(B,A)}{\alpha(B)}G(B;q^{-n}X^2)X^{\ord(\det A)-2\ord(\det B)}
\end{equation*}
\end{props}
\begin{proof}
We can similarly prove this in \cite[Lemma 4.2.3, Corollary]{K1}.
\end{proof}

We also define the rational function $\maf{B}(A,t)$ for $A\in\dHer_n(F;\hana{O}_E)^{nd}$
\begin{align*}
\maf{B}(A;t)={}&\dfrac{\prod_{j=0}^{n-1}(1-\varepsilon(n+j)q^{n+j}t^2)}{G(A,t^2)}
\intertext{where $\varepsilon$ is defined in the former part of this subsection and}
\varepsilon(j)=\begin{cases}
\varepsilon&\text{if $j$ is odd}\\
1&\text{if $j$ is even}
\end{cases}
\end{align*}
for integer $j$.
Then we have the explicit formula as in the following lemma. 
\begin{lems}\label{lem:locden11}
\begin{itemize}
\item[(1)] Assume that $E$ is inert or split over $F$.
For $B\in\Her_{r}(F;\hana{O}_E)^{nd}$, we have
\begin{equation*}
\maf{B}(1_{m-r}\perp\varpi B;X)=\prod_{i=r}^{m-1}(1-(\varepsilon q)^{m+i}X^2)
\end{equation*}
\item[(2)] Assume that $E$ is ramified over $F$ and $m-r$ is even.
For $B\in\Her_{r,*}^{nd}$, we have
\begin{equation*}
\maf{B}(\Theta_{m-r}\perp B;X)=
\prod_{i=[r/2]}^{[(m-2)/2]}(1-q^{2([(m+1)/2]+i)}X^2)
\end{equation*}
\end{itemize}
\end{lems}
\begin{proof}
As in \cite[Lemma 5.2.3]{K2}, we can prove this by Proposition \ref{prop:locden2}.
\end{proof}
Besides, we have one of the most important formulas for calculating the local factor.
Put 
\begin{equation*}
\maf{L}(X,t)=\begin{cases}
\prod_{i=1}^n(1-q^{-n+2i-1}X^2t^2)^{-1}(1-q^{-n+2i-1}X^{-2}t^2)^{-1}&\text{$E$ is inert over $F$}\\
\prod_{i=1}^n(1-q^{-n/2+i-1/2}Xt)^{-2}(1-q^{-n/2+i-1/2}X^{-1}t)^{-2}&\text{$E$ is split over $F$}\\
\prod_{i=1}^n(1-q^{-n/2+i-1/2}Xt)^{-1}(1-q^{-n/2+i-1/2}X^{-1}t)^{-1}&\text{$E$ is ramified over $F$}.
\end{cases}
\end{equation*}

\begin{lems}\label{lem:locden12}
Let $f$ be a Hecke eigenform in $\mr{ind}_{\Pe(F)}^{\Ge_n(F)}|\det_{\Pe}|^k$ and $f(1_{2n})=1$ and $A\in\dHer_n(F;\hana{O}_E)$. 
Put
\begin{align*}
\varphi(g)={}&\int_{\GL_n(E)}\bm{1}_n(h)|\det h|_E^s\int_{\Ge_n(F)}f(gy^{-1})\bm{1}_{\K_n\m_n(h^{-1})\K_n}(y)dydh\\
\varphi_A(\theta)={}&\int_{\Her_n(F)}\varphi(J_n\n_n(z)\m_n(\theta))\psi(-\tr(Az))dz\\
C(s)={}&\prod_{i=0}^{[(n-1)/2]}(1-q^{2i-s})\prod_{i=1}^{[n/2]}(1-\varepsilon q^{2i-1-s}),\\
C_A(k,s)={}&|\det A|^{k-n/2}q^{-(k-n/2)\maf{f}[n/2]}C(2k)C(2s)G(A,q^{-2s}).
\end{align*}
Then we have
\begin{align*}
&\int_{\GL_n(E)}\mt{m}(\theta)|\det \theta|^s\varphi_A(\theta^{-1})d\theta\\
={}&C_A(k,s)\maf{L}(q^{-k+n/2},q^{-s+n-1/2})\maf{B}(A,q^{-s})\wG(A;q^{-k+n/2};q^{-s+n/2})
\end{align*}
\end{lems}
\begin{proof}
Put 
\begin{align*}
f_A(\theta)={}&\int_{\Her_n(F)}f(J_n\n_n(z)\m_n(\theta))\psi(-\tr(Az))dz
\end{align*}
for $\theta\in\GL_n(E)$.
Then by Lemma \ref{lem:whfunc1}, we have
\begin{equation*}
\int_{\Her_n(F)}f(J_n\n_n(z))\psi(-\tr(Az))dz=b(A;2k).
\end{equation*}
Let $X=q^{-k+n/2}$.
Since $f$ is a Hecke eigenvector
\begin{align*}
\varphi(g)
={}&\maf{L}(X,q^{-s+n-1/2})\prod_{j=1}^{2n}(1-\varepsilon(j-1)q^{j-1}q^{-2s})f(g)
\end{align*}
by \cite[Lemma 19.8]{S2}.
Therefore we have
\begin{align*}
\varphi_A(\theta)={}&\int_{\Her_n(F)}\varphi(J_n\m_n(\theta)\n_n(z))\psi(-\tr(Az[\theta^*]))|\det\theta|_E^ndz\\
={}&\maf{L}(q^{-k+n/2},q^{-s+n-1/2})\prod_{j=1}^{2n}(1-\varepsilon(j-1)q^{j-1}q^{-2s})|\det\theta|_E^{-k+n}b(A[\theta];2k)
\end{align*}
Since 
$b(B;2\sigma)=C(2\sigma)|\det B|^{\sigma-n/2}q^{-(\sigma-n/2)\maf{f}[n/2]}\wF(B;q^{-\sigma+n/2})$, we have
\begin{align*}
&\int_{\GL_n(E)}\mt{m}(\theta)|\det\theta|_E^{s+k-n}b(A[\theta^{-1}];2k)d\theta\\
={}&C(2k)q^{-(k-n/2)\maf{f}[n/2]}\int_{\GL_n(E)}\mt{m}(\theta)|\det\theta|_E^{s+k-n}|\det A[\theta^{-1}]|^{k-n/2}\wF(A[\theta^{-1}];q^{-k+n/2})d\theta\\
={}&C(2k)q^{-(k-n/2)\maf{f}[n/2]}|\det A|^{k-n/2}\int_{\GL_n(E)}\mt{m}(\theta)|\det\theta|_E^{s-n/2}\wF(A[\theta^{-1}];q^{-k+n/2})d\theta\\
={}&\dfrac{C_A(k,s)}{C(2s)G(A;q^{-2s})}\wG(A;q^{-k+n/2},q^{-s+n/2}).
\end{align*}
Therefore we have
\begin{align*}
&\int_{\GL_n(E)}\mt{m}(\theta)|\det\theta|^s\varphi_A(\theta^{-1})d\theta\\
={}&\dfrac{C_A(k,s)}{C(2s)G(A,q^{-2s})}\maf{L}(q^{-k+n/2},q^{-s+n-1/2})\prod_{j=1}^{2n}(1-\varepsilon(j-1)q^{j-1}q^{-2s})\wG(A;q^{-k+n/2};q^{-s+n/2})\\
={}&C_A(k,s)\maf{L}(q^{-k+n/2},q^{-s+n-1/2})\maf{B}(A,q^{-s})\wG(A;q^{-k+n/2};q^{-s+n/2}).
\end{align*}
\end{proof}

\begin{props}\label{prop:locden4}
For $A\in\Her_n(F)$, we have
\begin{equation*}
\sum_{W\in \Mat_n(\hana{O}_E)^{nd}/\K_n^0}\wF(A[W];X) t^{\ord(N_{E/F}(\det W))}=\maf{B}(A,q^{-n/2}t)\wG(A;X,t)\maf{L}(X,q^{n/2-1/2}t).
\end{equation*}
\end{props}
\begin{proof}
Since $\wF(A; X)$ is a  Laurent polynomial for any $A$ and $\wF(A;X)=0$ if $A\notin\dHer_n(F;\hana{O}_E)$, we can show the assertion by proving when $A\in\dHer_n(F;\hana{O}_E)$, $X=q^{-k+n/2}$, $t=q^{-s+n/2}$ for any positive integer $k\geqq n$ and any complex number $s$ whose real part of $s$ is sufficiently large.
\\
Put $f$, $f_A$, $\varphi$, $\varphi_A$, $C(s)$, $C_A(k,s)$ as in Lemma \ref{lem:locden12}.\\
Let $\hana{S}$ be the representative set of $\sqcup_{h\in\K_n^0\lsla \Mat_n(\hana{O}_E)^{nd}/\K_n^0}\K_n\m_n(h^{-1})\K_n$ such that
\begin{equation*}
\hana{S}=\left\{\begin{pmatrix}
h_1^{-1}h_2&h_1^{-1}z_1(h_2^*)^{-1}\\
0&h_1^*(h_2^*)^{-1}
\end{pmatrix}
\middle|\begin{matrix}
(h_1,h_2)\in\K_n^0\lsla (\Mat_n(\hana{O}_E)^{nd}\times \Mat_n(\hana{O}_E)^{nd})/(\K_n^0\times1_n),\\
h_1\hana{O}_E^n+h_2\hana{O}_E^n=\hana{O}_E^n\text{ and }z_1\in\Her_n(F;\hana{O}_E)/(\Her_n(F;\hana{O}_E)[h_1^*])
\end{matrix}\right\}
\end{equation*}
by \cite[Lemma 19.2]{S2}. 
Let $\hana{S}_h=\K_n\m_n(h^{-1})\K_n\cap\hana{S}$. 
Then we have
\begin{align*}
\varphi(g)={}&\int_{\GL_n(E)}\bm{1}_n(h)|\det h|_E^s\sum_{y'\in\hana{S}_h}\int_{\Ge_n(F)}f(gy^{-1})\bm{1}_{\K_n}(yy'^{-1})dydh\\
={}&\int_{\GL_n(E)}\bm{1}_n(h)|\det h|_E^s\sum_{y'\in \hana{S}_h}\int_{\Ge_n(F)}f(gy'^{-1}y^{-1})\bm{1}_{\K_n}(y)dydh.
\end{align*}
Let $\hana{S}_h^1=\{(h_1,h_2)\in\GL_n(E)\times\GL_n(E)|\m_n(h_1^{-1})\n_n(z_1)\m_n(h_2)\in\hana{S}_h\text{ for some $z_1$}\}$ and $X_h=\Her_n(F;\hana{O}_E)/(\Her_n(F;\hana{O}_E)[h^*])$.
Then we have
\begin{align*}
\varphi(g)={}&\int_{\GL_n(E)}\bm{1}_n(h)|\det h|_E^s\sum_{(h_1,h_2)\in\hana{S}_h^1}\sum_{z_1\in X_{h_1}}f(g\m_n(h_2^{-1})\n_n(-z_1)\m_n(h_1))dh.
\end{align*}
If $\begin{pmatrix}
h_1^{-1}h_2&h_1^{-1}\sigma (h_2^*)^{-1}\\
0&h_1^*(h_2^*)^{-1}
\end{pmatrix}
\in \K_n\m_n(h^{-1})\K_n$, by \cite[Proposition 3.9]{S1} we have $|\det h|_E=|\det h_1|_E|\det h_2^*|_E\mu(\sigma)^{-2}$ where $\mu$ is defined in Definition \ref{def:siegelseries}.
Put 
\begin{equation*}
\widetilde{\hana{S}}=\{(h_1,h_2)\in \Mat_n(\hana{O}_E)^{nd}\times \Mat_n(\hana{O}_E)^{nd}|h_1\hana{O}_E^n+h_2\hana{O}_E^n=\hana{O}_E^n\}.
\end{equation*}
Then we have
\begin{align*}
&\varphi(g)\\
={}&\int_{\GL_n(E)}\bm{1}_n(h)\sum_{(h_1,h_2)\in\hana{S}_h^1}\sum_{z_1\in X_{h_1}}|\det h_1|_E^s|\det h_2^*|_E^s\mu(z_1)^{-2s}f(g\m_n(h_2^{-1})\n_n(-z_1)\m_n(h_1))dh\\
={}&\int_{\GL_n(E)}\int_{\GL_n(E)}\bm{1}_{\widetilde{\hana{S}}}(h_1,h_2)|\det h_1|_E^s|\det h_2^*|_E^s\\
&\times\sum_{z_1\in X_{h_1}}\mu(z_1)^{-2s}f(g\m_n(h_2^{-1})\n_n(-z_1)\m_n(h_1))dh_1dh_2.
\end{align*}
Therefore we have
\begin{align*}
\varphi_A(\theta)={}&\int_{\GL_n(E)}\int_{\GL_n(E)}\bm{1}_{\widetilde{\hana{S}}}(h_1,h_2)|\det h_1|_E^s|\det h_2^*|_E^s\sum_{z_1\in X_{h_1}}\mu(z_1)^{-2s}\\
&\times\int_{\Her_n(F)}f(J_n\n_n(z)\m_n(\theta)\m_n(h_2^{-1})\n_n(-z_1)\m_n(h_1))\psi(-\tr(Az))dzdh_1dh_2.
\end{align*}
Replacing $z$ by $z+z_1[(h_2^*)^{-1}\theta^*]$, we have
\begin{align*}
\varphi_A(\theta)
={}&\int_{\GL_n(E)}\int_{\GL_n(E)}\bm{1}_{\widetilde{\hana{S}}}(h_1,h_2)|\det h_1|_E^s|\det h_2^*|_E^s\sum_{z_1\in X_{h_1}}\mu(z_1)^{-2s}\\
&\int_{\Her_n(F)}f(J_n\n_n(z)\m_n(\theta h_2^{-1}h_1))\psi(-\tr(A(z+z_1[(h_2^*)^{-1}\theta^*])))dzdh_1dh_2\\
={}&\int_{\GL_n(E)}\int_{\GL_n(E)}\bm{1}_{\widetilde{\hana{S}}}(h_1,h_2)|\det h_1|_E^s|\det h_2|_E^{s}\\
&\times\sum_{z_1\in X_{h_1}}\mu(z_1)^{-2s}\psi(-\tr(A[\theta h_2^{-1}]z_1))f_A(\theta h_2^{-1}h_1)dh_1dh_2.
\end{align*}
As \cite[Lemma 19.6]{S2}, if $\zeta\in\Her_n(F)$ and $\zeta[h]\in\dHer_n(F;\hana{O}_E)$
we have
\begin{equation*}
\sum_{z_1\in X_{h}}\psi(-\tr(\zeta z_1))\mu(z_1)^{-2s}=
|\det h|_E^{-n}b(\zeta;2s).
\end{equation*}
Therefore, we have
\begin{align*}
&\varphi_A(\theta)\\
={}&\int_{\GL_n(E)}\int_{\GL_n(E)}\bm{1}_{\widetilde{\hana{S}}}(h_1,h_2)|\det h_1|_E^{s-n}|\det h_2|_E^{s}b(A[\theta h_2^{-1}];2s)f_A(\theta h_2^{-1}h_1)dh_1dh_2.
\end{align*}
By replacing $h_1$ by $h_2h_3$, we have
\begin{align*}
\varphi_A(\theta)={}&\int_{\GL_n(E)}\int_{\GL_n(E)}\bm{1}_{\widetilde{\hana{S}}}(h_2h_3,h_2)|\det h_3|_E^{s-n}|\det h_2|_E^{2s-n}b(A[\theta h_2^{-1}];2s)f_A(\theta h_3)dh_3dh_2.
\end{align*}
Therefore we have
\begin{align*}
&\int_{\GL_n(E)}\mt{m}(\theta)|\det\theta|_E^{s}\varphi_A(\theta^{-1})d\theta\\
={}&\int_{\GL_n(E)}\mt{m}(\theta)|\det\theta|_E^{s}\int_{\GL_n(E)}\int_{\GL_n(E)}\bm{1}_{\widetilde{\hana{S}}}(h_2h_3,h_2)|\det h_3|_E^{s-n}\\
&\times|\det h_2|_E^{2s-n}b(A[\theta^{-1} h_2^{-1}];2s)f_A(\theta^{-1} h_3)dh_3dh_2d\theta.
\intertext{Replacing $\theta^{-1} h_3$ by $h_4$ and $h_2\theta$ by $h_5$, we have}
&\int_{\GL_n(E)}\mt{m}(\theta)|\det\theta|_E^{s}\varphi_A(\theta^{-1})d\theta\\
={}&\int_{\GL_n(E)}\mt{m}(\theta)\int_{\GL_n(E)}\int_{\GL_n(E)}\bm{1}_{\widetilde{\hana{S}}}(h_5h_4,h_5\theta^{-1})|\det h_4|_E^{s-n}\\
&\times|\det h_5|_E^{2s-n}b(A[h_5^{-1}];2s)f_A(h_4)dh_5dh_4d\theta.
\end{align*}
As the proof of \cite[Lemma 19.14]{S2}, we have
\begin{equation*}
\int_{\GL_n(E)}\mt{m}(\theta)\bm{1}_{\widetilde{\hana{S}}}(x,y\theta^{-1})d\theta=1_n(y^{-1}x)\mt{m}(y)
\end{equation*}
Thus, we have
\begin{align*}
&\int_{\GL_n(E)}\mt{m}(\theta)|\det\theta|_E^{s}\varphi_A(\theta^{-1})d\theta\\
={}&\int_{\GL_n(E)}\bm{1}_n(h_4)|\det h_4|_E^{s-n}f_A(h_4)dh_4\times\int_{\GL_n(E)}\mt{m}(h_5)|\det h_5|_E^{2s-n}b(A[h_5^{-1}];2s)dh_5.
\end{align*}
Since 
\begin{align*}
b(B;2s)={}&C(2s)F(B;q^{-2s})\\
f_A(\theta)
={}&q^{-(k-n/2)\maf{f}[n/2]}|\det A|^{k-n/2}\wF(A[\theta];q^{-k+n/2})C(2k)|\det\theta|_E^{n/2},
\end{align*}
we have 
\begin{align*}
&\int_{\GL_n(E)}\mt{m}(\theta)|\det\theta|^{s}\varphi_A(\theta^{-1})d\theta\\
={}&
\int_{\GL_n(E)}\bm{1}_n(h_4)|\det h_4|_E^{s-n/2}q^{-(k-n/2)\maf{f}[n/2]}|\det A|^{k-n/2}\wF(A[h_4];q^{-k+n/2})C(2k)dh_4\\
&\times\int_{\GL_n(E)}\mt{m}(h_5)|\det h_5|_E^{2s-n}C(2s)F(A[h_5^{-1}];q^{-2s})dh_5\\
={}&q^{-(k-n/2)\maf{f}[n/2]}|\det A|^{k-n/2}C(2k)C(2s)G(A;q^{-2s})\\
&\times\int_{\GL_n(E)}\bm{1}_n(h_4)|\det h_4|_E^{s-n/2}\wF(A[h_4];q^{-k+n/2})dh_4.
\end{align*}
Thus, we have
\begin{align*}
&\int_{\GL_n(E)}\mt{m}(\theta)|\det\theta|^{s}\varphi_A(\theta^{-1})d\theta\\
={}&C_A(k,s)\int_{\GL_n(E)}\bm{1}_n(h_4)|\det h_4|_E^{s-n/2}\wF(A[h_4];q^{-k+n/2})dh_4.
\end{align*}
By Lemma \ref{lem:locden12}, we have
\begin{align*}
&\maf{L}(q^{-k+n/2},q^{-s+n-1/2})\maf{B}(A,q^{-s})\wG(A;q^{-k+n/2};q^{-s+n/2})\\
={}&\int_{\GL_n(E)}\bm{1}_n(h_4)|\det h_4|_E^{s-n/2}\wF(A[h_4];q^{-k+n/2})dh_4.
\end{align*}
Therefore, the assertion holds. 
\end{proof}

\subsection{The Local Rankin--Selberg Integral}\label{subsec5-2}
To compute $\hat{\hana{R}}(X,Y,t)$, we give a different form of $\hat{\hana{R}}(X,Y,t)$. We assume that $n$ is odd in the remainder of this section.
By comparing the stabilizer of the action of $\GL_n(\hana{O}_E)/1+\maf{p}^e\Mat_n(\hana{O}_E)$ on $\overline{\hana{A}}_e(B,B)$ for sufficiently large $e$ and the local density, we have $\mr{vol}(B[\K_{n}^0])=\alpha(B)^{-1}\prod_{i=1}^nL(i,1)^{-1}L(i,\varepsilon_{E/F})^{-1}$. Therefore, we have
\begin{align*}
\hat{\hana{R}}(X,Y,q^{-s})={}&\int_{\Her_n(F)^{nd}}\wF(Z;q^{-x})\wF(Z;q^{-y})|\det Z|^{s}dZ\\
={}&\prod_{i=1}^nL(i,1)^{-1}L(i,\varepsilon_{E/F})^{-1}\sum_{Z\in\Her_n(F)^{nd}/\K_{n}^0}\dfrac{\wF(Z;q^{-x})\wF(Z;q^{-y})}{\alpha(Z)}|\det Z|^{s}.
\end{align*}
Let $X=q^{-x}$, $Y=q^{-y}$, and $t=q^{-s}$ for $x,y\in\mb{C}$ and we assume that $\Re s>0$.
For $\bm{X}\subset \Her_n(F)$ and for $S\subset E$, put 
\begin{equation*}
\bm{X}(S)=\{x\in\bm{X}|\det x \in S\}
\end{equation*}
and if $S=\cup_{i=-\infty}^\infty d\varpi^{i}N_{E/F}(\hana{O}_{E}^\times)$, we write $\bm{X}(S)=\bm{X}_d$ for $d\in F^\times$.
Put 
\begin{align*}
\hana{R}(d;q^{-x},q^{-y},q^{-s})=\int_{\Her_n(F)_d^{nd}}\wF(Z;q^{-x})\wF(Z;q^{-y})|\det Z|^{s}dZ
\end{align*}
The relation between $\hat{\hana{R}}$ and $\hana{R}$ is given by the following proposition.
\begin{props}\label{prop:RS1}
If $E$ is inert over $F$ or $E=F\oplus F$, we have 
\begin{equation*}
\hat{\hana{R}}(X,Y,t)=\hana{R}(d;X,Y,t)
\end{equation*}
for any $d\in F^\times$.\\
If $E$ is ramified over $F$, we have
\begin{equation*}
\hat{\hana{R}}(X,Y,t)=\sum_{d\in \hana{O}_F^{\times}/N_{E/F}(\hana{O}_{E}^\times)}\hana{R}(d;X,Y,t)
\end{equation*}
\end{props}
\begin{proof}
If $E$ is inert over $F$, $N_{E/F}(\hana{O}_{E}^\times)=\hana{O}_F^{\times}$.
Thus, we have $\cup_{i=-\infty}^{\infty} d\varpi^{i}N_{E/F}(\hana{O}_{E}^\times)=F^\times$ for any $d\in F^{\times}$, and we have
\begin{align*}
\hat{\hana{R}}(X,Y,t)={}&\hana{R}(d;X,Y,t)
\end{align*}
If $E=F\oplus F$, $N_{E/F}(E^\times)=F^\times$ so we have
\begin{align*}
\hat{\hana{R}}(X,Y,t)={}&\hana{R}(d;X,Y,t)
\end{align*}
for any $d\in F^\times$.\\
If $E$ is ramified over $F$, $\Her_n(F)^{nd}=\bigsqcup_{d\in \hana{O}_F^\times/N_{E/F}(\hana{O}_E^\times)}\Her_n(F)_d^{nd}$.
Therefore, we have
\begin{align*}
\hat{\hana{R}}(X,Y,t)={}&\sum_{d\in\hana{O}_F^\times/N_{E/F}(\hana{O}_E^\times)}\hana{R}(d;X,Y,t)
\end{align*}
\end{proof}
Therefore, by computing $\hana{R}$, we can calculate the local factor of the Rankin-Selberg series.
By applying proposition \ref{prop:locden3} to $\wF(A,Y)$, we get the formula 
\begin{align*}
&\hana{R}(d;X,Y,t)=\int_{\Her_n(F)_d^{nd}}\wF(Z;q^{-x})\wF(Z;q^{-y})|\det Z|^{s}dZ\\
={}&\int_{\Her_n(F)_d^{nd}}\wF(Z;q^{-x})
\int_{\GL_n(E)}\bm{1}_n(h)G(Z[h^{-1}];q^{-y})q^{\maf{f}(n-1)y/2}|\det Z|^{s-y}|\det h|_E^{2y}dhdZ.
\end{align*}
Replacing $Z$ by $Z_1[h]$,
\begin{align*}
&q^{-\maf{f}(n-1)y/2}\hana{R}(d;X,Y,t)\\
={}&\int_{\Her_n(F)_d^{nd}}G(Z_1;q^{-n-2y})|\det Z_1|^{s-y}\int_{\GL_n(E)}\bm{1}_n(h)\wF(Z_1[h];q^{-x})|\det h|_E^{s+y+n}dhdZ_1.
\end{align*}
By Proposition \ref{prop:locden4}, we have
\begin{align*}
&q^{-\maf{f}(n-1)y/2}\maf{L}(q^{-x},q^{-s-y-n/2-1/2})^{-1}\hana{R}(d;X,Y,t)\\
={}&\int_{\Her_n(F)_d^{nd}}G(Z_1;q^{-n-2y})|\det Z_1|^{s-y}\maf{B}(Z_1;q^{-3n/2-y-s})\wG(Z_1;q^{-x},q^{-n-y-s})dZ_1.
\end{align*}

Put 
\begin{align*}
&\hana{R}^0(d;X,Y,t)\\
={}&\int_{\Her_n(F)_d^{nd}}G(Z;q^{-n-2y})|\det Z|^{s-y}\maf{B}(Z;q^{-3n/2-y-s})\wG(Z;q^{-x},q^{-n-y-s})dZ.
\end{align*}
Then we have
\begin{equation*}
\hana{R}(d;X,Y,t)=q^{(n-1)y\maf{f}/2}\hana{R}^0(d;X,Y,t)\maf{L}(X,q^{-n/2-1/2}Yt)
\end{equation*}
We put $\wHer_{m,d}=\dHer_m(F;\hana{O}_E)_{d}\cap\Her_{m,*}$ where $\Her_{m,*}$ is defined in \ref{eq:qsemiint}.
Put $P_r(d;X,t)$, $\wP_r(d;X,Y,t)$, $Q_{r}(d;X,Y,t)$ for $d\in\hana{O}_F\cap F^{\times}$ as follows:
\begin{align*}
P_r(d;X,q^{-s})={}&\prod_{i=1}^rL(i,1)L(i,\varepsilon_{E/F})\int_{\dHer_r(F;\hana{O}_E)_d}\wF(Z;X)|\det Z|^sdZ,\\
\wP_r(d;X,Y,q^{-s})={}&\prod_{i=1}^rL(i,1)L(i,\varepsilon_{E/F})\int_{\dHer_r(F;\hana{O}_E)_d}\wG(Z;q^{-x},q^{-s-y})|\det Z|^{s-y}dZ,\\
Q_{r}(d;X,Y,q^{-s})={}&\prod_{i=1}^rL(i,1)L(i,\varepsilon_{E/F})\int_{\wHer_{r,d}}\wG(Z;q^{-x},q^{-s-y})|\det Z|^{s-y}dZ.
\end{align*}

We now relate $\hana{R}^0$, $P_r$, $\wP_r$ and $Q_{r}$. 
We put $\phi_m(t):=\prod_{i=1}^m(1-t^i)$.
\begin{props}\label{prop:KS1}
Put 
\begin{equation*}
\maf{l}(E;t,r)=
\begin{cases}
\prod_{i=1}^r(1-t^4q^{-2r-2+2i})&(\text{$E$ is inert over $F$})\\
\prod_{i=1}^r(1-t^2q^{-r-1+i})^2&(E=F\oplus F)\\
\prod_{i=1}^r(1-t^2q^{-r-1+i})&(\text{$E$ is ramified over $F$})\\
\end{cases}
\end{equation*}
Then, we have
\begin{equation*}
\wP_r(d;X,Y,t)=P_r(d,X,tY^{-1})\maf{l}(E;t,r)
\end{equation*}
\end{props}
\begin{proof}
We can similarly prove this in \cite[Proposition 5.3.1]{K2}.
\end{proof}

Put 
\begin{equation}
\hat{\varepsilon}=\begin{cases}
\sqrt{-1}&\text{if $E$ is inert,}\\
1&\text{otherwise}.
\end{cases} \label{xihat}
\end{equation}
To compute the relation among $\{\wP_r\}_{r=1}^m$, $\{Q_r\}_{r=1}^m$ and $\hat{\hana{R}}^0$, we prove the following lemma:

\begin{lems}\label{lem:RS1} 
Let $m$ be a positive integer.
For every positive integer $r\leqq m$ and $B'\in\Her_{r,*}^{nd}$, the following identity holds whenever either $E/F$ is ramified and $m-r$ is even, or $E/F$ is not ramified.
\begin{equation*}
\wG(\Theta_{m-r}\perp B';X,t)=\wG(B';\hat{\varepsilon}^{m-r}X,\hat{\varepsilon}^{(m-r)}t)\hat{\varepsilon}^{(m-r)\ord(\det B')}.
\end{equation*}
\end{lems}
\begin{proof}
This is a more detailed proof of Katsurada's result \cite[Lemma 5.3.3]{K2}.
Let $\varepsilon'=\hat{\varepsilon}^2$.
Then $\varepsilon'=\varepsilon$ if $E$ is inert over $F$ and $\varepsilon'=1$ otherwise.
By comparing Proposition \ref{prop:locden2} in the case of $G(1_{m-r}\perp B';X)$ and Proposition \ref{prop:locden2} in the case of $G(B';X)$, we have
\begin{equation*}
G(\Theta_{m-r}\perp B';X)=G(B';(\varepsilon'q)^{m-r}X).
\end{equation*}
By Proposition \ref{prop:locden3} and the above equation, we have
\begin{align*}
&F(\Theta_{m-r}\perp B';X)\\
={}&\sum_{W\in \Mat_m(\hana{O}_{E})^{nd}/\K_{m}^0}G((\Theta_{m-r}\perp B')[W];X)(q^mX)^{\ord(N_{E/F}(\det W))}\\
={}&\sum_{W\in \Mat_r(\hana{O}_{E})^{nd}/\K_{r}^0}G(\Theta_{m-r}\perp (B'[W]);X)(q^mX)^{\ord(N_{E/F}(\det W))}.
\end{align*}

Since $\ord(N_{E/F}(x))$ is even for any $x\in E^\times$ if $E$ is inert over $F$ and $\varepsilon'=1$ otherwise, we have
\begin{align*}
&F(\Theta_{m-r}\perp B';X)\\
={}&\sum_{W\in \Mat_r(\hana{O}_{E})^{nd}/\K_{r}^0}G(B'[W];(\varepsilon'q)^{(m-r)}X)(q^r(\varepsilon'q)^{(m-r)}X)^{\ord(N_{E/F}(\det W))}.
\end{align*}
Therefore, we have
\begin{equation*}
F(\Theta_{m-r}\perp B';X)=F(B';(\varepsilon'q)^{m-r}X).
\end{equation*}
Since $\wF(B;X)=X^{-e_v(B)}F(B;X)$ for $B\in\Her_n(F)$, we have
\begin{align*}
\wF(\Theta_{m-r}\perp B';X)={}&X^{-\ord(\det\hat{\gamma}(\Theta_{m-r}\perp B'))}F(\Theta_{m-r}\perp B';q^{-m}X^2)\\
={}&X^{-[m/2]\maf{f}+(m-r)\maf{f}/2-\ord(\det B')}F(B';q^{-r}(\hat{\varepsilon}^{m-r}X)^2)\\
={}&\wF(B';\hat{\varepsilon}^{m-r}X)\hat{\varepsilon}^{(m-r)\ord(\det B')}.
\end{align*}
For $m\geq1$, we put 
\begin{equation*}
    \mb{I}_m=\begin{cases}
        \{0,1,\dots,m\}&\text{ if $E$ is not split,}\\
        \{0,1,\dots,m\}\times\{0,1,\dots,m\}&\text{ if $E$ is split.}
    \end{cases}
\end{equation*}
By definition of $\wG$, we have
\begin{align*}
&\wG(\Theta_{m-r}\perp B';X,t)\\
={}&\sum_{\bm{i}\in\mb{I}_m}\sum_{W\in\K_{m}^0\lsla\D_{m,\bm{i}}}\mt{m}(W)t^{\ord(N_{E/F}(\det W))}\wF((\Theta_{m-r}\perp B')[W^{-1}];X)\\
={}&\sum_{\bm{i}\in\mb{I}_r}\sum_{W\in\K_{r}^0\lsla\D_{r,\bm{i}}}\mt{m}(W)t^{\ord(N_{E/F}(\det W))}\wF(\Theta_{m-r}\perp B'[W^{-1}];X)\\
={}&\sum_{\bm{i}\in\mb{I}_r}\sum_{W\in\K_{r}^0\lsla\D_{r,\bm{i}}}\mt{m}(W)t^{\ord(N_{E/F}(\det W))}\wF(B'[W^{-1}];\hat{\varepsilon}^{m-r}X)\hat{\varepsilon}^{(m-r)\ord(\det(B'[W^{-1}]))}\\
={}&\hat{\varepsilon}^{(m-r)\ord(\det(B'))}\sum_{\bm{i}\in\mb{I}_r}\sum_{W\in\K_{r}^0\lsla\D_{r,\bm{i}}}\mt{m}(W)(\hat{\varepsilon}^{-(m-r)}t)^{\ord(N_{E/F}(\det W))}\wF(B'[W^{-1}];\hat{\varepsilon}^{m-r}X)\\
={}&\hat{\varepsilon}^{(m-r)\ord(\det(B'))}\wG(B';\hat{\varepsilon}^{m-r}X,\hat{\varepsilon}^{-(m-r)}t).
\end{align*}
If $E$ is inert, $\ord(N_{E/F}(\hana{O}_{E}^{nd}))\subset 2\mb{Z}$.
This induces 
\begin{equation*}
\wG(B';\hat{\varepsilon}^{m-r}X,\hat{\varepsilon}^{-(m-r)}t)=\wG(B';\hat{\varepsilon}^{m-r}X,\hat{\varepsilon}^{m-r}t).
\end{equation*}
Therefore, the assertion holds.
\end{proof}

\begin{props}\label{prop:RS2}
Let $\hat{\varepsilon}$ be as in \ref{xihat} and $d\in F^\times$.
\begin{itemize}
\item[(1)]Assume that $E$ is not ramified over $F$. Then
\begin{equation*}
\wP_m(d;\hat{\varepsilon}^{-m} X, Y,\hat{\varepsilon}^{-m}t)=\sum_{r=0}^m\dfrac{1}{\phi_{m-r}(\varepsilon q^{-1})}Q_{r}(d;\hat{\varepsilon}^{-r}X,Y,\hat{\varepsilon}^{-r}t).
\end{equation*}
\item[(2)]Assume that $E$ is ramified over $F$.
Then we have
\begin{align*}
&(tY^{-1})^{\maf{f}\cdot[m/2]}\wP_m((-\varpi^{\maf{f}})^{[m/2]}d;X,Y,t)\\
={}&
\begin{cases}
\disp\sum_{r=0}^{(m-1)/2}\dfrac{1}{\phi_{(m-2r-1)/2}(q^{-2})}(tY^{-1})^{r\maf{f}}Q_{2r+1}((-\varpi^{\maf{f}})^rd;X,Y,t)&(\text{$m$:odd})\\
\disp\sum_{r=0}^{m/2}\dfrac{1}{\phi_{(m-2r)/2}(q^{-2})}(tY^{-1})^{r\maf{f}}Q_{2r}((-\varpi^{\maf{f}})^rd;X,Y,t)&(\text{$m$:even})
\end{cases}
\end{align*}
\end{itemize}
\end{props}
\begin{proof}
\begin{itemize}
\item[(1)]Suppose that $E$ is not ramified.
Let $S_{r,d}$ be a representative set of $\wHer_{r,d}/\K_{r}^0$.
Then, by the structure theory of Hermitian lattices, $\dHer_m(F;\hana{O}_E)_d/\K_{m}^0$ has a representative set $\bigcup_{r=0}^m\{1_{m-r}\perp B|B\in S_{r,d}\}$.
Then since $\wG(Z[k];X,tY)=\wG(Z;X,tY)$ for any $k\in\K_m^0$, we have
\begin{align*}
&\wP_m(d;X,Y,t)\\
={}&\prod_{i=1}^mL(i,1)L(i,\varepsilon_{E/F})\int_{\dHer_m(F;\hana{O}_E)_d}\wG(Z;X,tY)|\det Z|^{s-y}dZ\\
={}&\sum_{Z_1\in\dHer_{m,d}/\K_m^0}\prod_{i=1}^mL(i,1)L(i,\varepsilon_{E/F})\int_{Z_1[\K_m^0]}\wG(Z;X,tY)|\det Z_1|^{s-y}dZ\\
={}&\sum_{r=0}^m\sum_{Z_1\in S_{r,d}}|\det Z_1|^{s-y}\prod_{i=1}^mL(i,1)L(i,\varepsilon_{E/F})\int_{(1_{m-r}\perp Z_1)[\K_m^0]}\wG(Z;X,tY)dZ.
\end{align*}
By Lemma \ref{lem:locden8} and Lemma \ref{lem:RS1}, we have
\begin{align*}
&\wP_m(d;X,Y,t)\\
={}&\sum_{r=0}^m\sum_{Z_1\in S_{r,d}}\dfrac{\hat{\varepsilon}^{(m-r)\ord(\det Z_1)}\wG(Z_1;\hat{\varepsilon}^{m-r}X, \hat{\varepsilon}^{m-r}tY)}{\phi_{m-r}(\varepsilon q^{-1})\alpha(Z_1)}|\det Z_1|^{s-y}\\
={}&\sum_{r=0}^m\phi_{m-r}(\varepsilon q^{-1})^{-1}\sum_{Z_1\in S_{r,d}}\prod_{i=1}^rL(i,1)L(i,\varepsilon_{E/F})\\
&\times\int_{Z_1[\K_r^0]}\wG(Z;\hat{\varepsilon}^{m-r}X, \hat{\varepsilon}^{m-r}tY)\hat{\varepsilon}^{(m-r)\ord(\det Z_1)}|\det Z_1|^{s-y}dZ\\
={}&\sum_{r=0}^m\dfrac{1}{\phi_{m-r}(\varepsilon q^{-1})}Q_{r}(d; \hat{\varepsilon}^{m-r}X, Y, \hat{\varepsilon}^{m-r}t).
\end{align*}
Therefore, we have 
\begin{equation*}
\wP_m(d;\hat{\varepsilon}^{-m}X,Y,\hat{\varepsilon}^{-m}t)=\sum_{r=0}^m\dfrac{1}{\phi_{m-r}(\varepsilon q^{-1})}Q_{r}(d; \hat{\varepsilon}^{-r}X, Y, \hat{\varepsilon}^{-r}t).
\end{equation*}
\item[(2)]Suppose that $E$ is ramified over $F$.
Let $d_a=(-\varpi)^{-a}d$ for $d\in F^\times$ and integer $a$ and $S_{r,d}$ as above.
Then, by structure theory of $\dHer_n(F;\hana{O}_E)$, $\dHer_m(F;\hana{O}_E)_d/\K_{m}^0$ has a representative set $\bigcup_{r=0,\text{$m-r$: even}}^m\{\Theta_{m-r}\perp B|B\in S_r\}$.
Then since $\wG(Z[k];X,tY)=\wG(Z;X,tY)$ for any $k\in\K_m^0$, we have
\begin{align*}
&\wP_m(d;X,Y,t)\\
={}&\prod_{i=1}^mL(i,1)L(i,\varepsilon_{E/F})\sum_{Z\in\dHer_m(F;\hana{O}_E)_d/\K_{m}^0}\int_{Z[\K_m^0]}\wG(Z_1;X,tY)|\det Z_1|^{s-y}dZ_1\\
={}&\sum_{r=0,\text{$m-r$: even}}^m\sum_{Z\in S_{r,d_{(m-r)\maf{f}/2}}}\dfrac{|\det (\Theta_{m-r}\perp Z)|^{s-y}}{\alpha(\Theta_{m-r}\perp Z)}\wG((\Theta_{m-r}\perp Z);X,tY).\end{align*}
By Lemma \ref{lem:locden8} and Lemma \ref{lem:RS1}, we have
\begin{align*}
&\wP_m(d;X,Y,t)\\
={}&\sum_{r=0,\text{$m-r$: even}}^m\sum_{Z\in\wHer_{r,d'(r)}/\K_{r}^0}\dfrac{\wG(Z; X, tY)}{\phi_{(m-r)/2}(q^{-2})\alpha(Z)}|\det\Theta_{m-r}|^{s-y}|\det Z|^{s-y}\\
={}&\sum_{r=0,\text{$m-r$: even}}^m\dfrac{(tY^{-1})^{-(m-r)/2\cdot\maf{f}}}{\phi_{(m-r)/2}(q^{-2})}\sum_{Z\in\wHer_{r,d'(r)}/\K_{r}^0}\prod_{i=1}^rL(i,1)L(i,\varepsilon_{E/F})\\
&\times\int_{Z[\K_r^0]}\wG(Z_1; X, tY)|\det Z_1|^{s-y}dZ_1\\
={}&\sum_{r=0,\text{$m-r$: even}}^m\dfrac{(tY^{-1})^{-(m-r)/2\cdot\maf{f}}}{\phi_{(m-r)/2}(q^{-2})}Q_{r}(d'(r); X, Y, t)
\end{align*}
where $d'(r)=d_{(m-r)\maf{f}/2}$.
Therefore we have 
\begin{align*}
&(tY^{-1})^{[m/2]\cdot\maf{f}}\wP_m((-\varpi^{\maf{f}})^{[m/2]}d;X,Y,t)\\
={}&\sum_{r=0,\text{$m-r$ is even}}^m\dfrac{(tY^{-1})^{[r/2]\cdot\maf{f}}}{\phi_{(m-r)/2}(q^{-2})}Q_{r}((-\varpi^{\maf{f}})^{[r/2]}d; X, Y, t).
\end{align*}
\end{itemize}
\end{proof}
\begin{Cor}\label{cor:propofRS2}
Let $d\in F^\times$.
\begin{itemize}
\item[(1)]Assume that $E$ is inert over $F$.
Then we have
\begin{equation*}
Q_{r}(d,\hat{\varepsilon}^{-r}X,Y,\hat{\varepsilon}^{-r}t)=\sum_{m=0}^r\dfrac{(-1)^{r-m}(\varepsilon q)^{(r-m)(r-m-1)/2}}{\phi_{r-m}(\varepsilon q^{-1})}\wP_{m}(d,\hat{\varepsilon}^{-m}X,Y,\hat{\varepsilon}^{-m}t)
\end{equation*}
for any positive integer $r$.
\item[(2)]Assume that $E$ is ramified over $F$.
Then we have 
\begin{align*}
(tY^{-1})^{-r\maf{f}}Q_{2r+1}((-\varpi^{\maf{f}})^rd,X,Y,t)={}& \sum_{m=0}^r\dfrac{(-1)^m(q)^{m^2-m}}{\phi_m(q^{-2})}\wP_{2r+1-2m}((-\varpi^{\maf{f}})^{r-m}d,X,Y,t)
\end{align*}
for any non negative integer $r$. We also have
\begin{align*}
(tY^{-1})^{-r\maf{f}}Q_{2r}((-\varpi^{\maf{f}})^{r}d,X,Y,t)={}& \sum_{m=0}^r\dfrac{(-1)^m(q)^{m^2-m}}{\phi_m(q^{-2})}\wP_{2r-2m}((-\varpi^{\maf{f}})^{r-m}d,X,Y,t)
\end{align*}
for any positive integer $r$.
\end{itemize}
\end{Cor}
\begin{proof}
\begin{itemize}
\item[(1)] By Proposition \ref{prop:RS2}, we have
\begin{align*}
&\sum_{m=0}^r\dfrac{(-1)^{r-m}(\varepsilon q)^{(r-m)(r-m-1)/2}}{\phi_{r-m}(\varepsilon q^{-1})}\wP_m(d,\hat{\varepsilon}^{-m}X,Y,\hat{\varepsilon}^{-m}t)\\
={}&\sum_{m=0}^r\dfrac{(-1)^{r-m}(\varepsilon q)^{(r-m)(r-m-1)/2}}{\phi_{r-m}(\varepsilon q^{-1})}\sum_{l=0}^{m}\dfrac{1}{\phi_{m-l}(\varepsilon q^{-1})}Q_l(d;\hat{\varepsilon}^{-l}X,Y,\hat{\varepsilon}^{-l}t)\\
={}&\sum_{l=0}^r\dfrac{Q_l(d;\hat{\varepsilon}^{-l}X,Y,\hat{\varepsilon}^{-l}t)}{\phi_{r-l}(\varepsilon q^{-1})}\sum_{m=l}^{r}\dfrac{(-1)^{r-m}(\varepsilon q)^{(r-m)(r-m-1)/2}\phi_{r-l}(\varepsilon q^{-1})}{\phi_{r-m}(\varepsilon q^{-1})\phi_{m-l}(\varepsilon q^{-1})}\\
={}&\sum_{l=0}^r\dfrac{Q_l(d;\hat{\varepsilon}^{-l}X,Y,\hat{\varepsilon}^{-l}t)}{\phi_{r-l}(\varepsilon q^{-1})}\sum_{k=0}^{r-l}\dfrac{(-\varepsilon q)^{-k}(\varepsilon q)^{k(k+1)/2}\phi_{r-l}(\varepsilon q^{-1})}{\phi_k(\varepsilon q^{-1})\phi_{r-l-k}(\varepsilon q^{-1})}\\
={}&Q_{r}(d;\hat{\varepsilon}^{-r}X,Y,\hat{\varepsilon}^{-r}t)+\sum_{l=0}^{r-1}\dfrac{Q_l(d;\hat{\varepsilon}^{-l}X,Y,\hat{\varepsilon}^{-l}t)}{\phi_{r-l}(\varepsilon q^{-1})}\prod_{k=1}^{r-l}(1-(\varepsilon q)^{-1}\cdot(\varepsilon q)^k)\\
={}&Q_{r}(d;\hat{\varepsilon}^{-r}X,Y,\hat{\varepsilon}^{-r}t).
\end{align*}
\item[(2)]Let $r'$ be $2r$ or $2r+1$. Then $r=[r'/2]$. By Proposition \ref{prop:RS2}, we have
\begin{align*}
&\sum_{m=0}^{r}\dfrac{(-1)^m(q)^{m^2-m}}{\phi_m(q^{-2})}(tY^{-1})^{(r-m)\maf{f}}\wP_{r'-2m}((-\varpi^{\maf{f}})^{r-m}d,X,Y,t)\\
={}&\sum_{m=0}^r\dfrac{(-1)^m(q)^{m^2-m}}{\phi_m(q^{-2})}\sum_{l=0,\text{$r'-2m-l$:even}}^{r'-2m}\dfrac{(tY^{-1})^{[l/2]\cdot\maf{f}}}{\phi_{(r'-2m-l)/2}(q^{-2})}Q_l((-\varpi^{\maf{f}})^{[l/2]}d; X, Y, t)\\
={}&\sum_{l=0,\text{$r'-l$:even}}^{r'}\dfrac{Q_l((-\varpi^{\maf{f}})^{[l/2]}d; X, Y, t)}{\phi_{(r'-l)/2}(q^{-2})}(tY^{-1})^{[l/2]\cdot\maf{f}}\sum_{m=0}^{(r'-l)/2}\dfrac{(-q^{-2})^m(q)^{m^2+m}\phi_{(r'-l)/2}(q^{-2})}{\phi_m(q^{-2})\phi_{(r'-2m-l)/2}(q^{-2})}\\
={}&Q_{r'}((-\varpi^{\maf{f}})^rd; X, Y, t)\\
&+\sum_{l=0,\text{$r'-l$:even}}^{r'-2}\dfrac{Q_l((-\varpi^{\maf{f}})^{[l/2]}d; X, Y, t)}{\phi_{(r'-l)/2}(q^{-2})}(tY^{-1})^{[l/2]\cdot\maf{f}}\prod_{m=1}^{(r'-l)/2}(1-q^{-2}\cdot q^{2m})\\
={}&Q_{r'}((-\varpi^{\maf{f}})^rd; X, Y, t).
\end{align*}
\end{itemize}
\end{proof}

\begin{thms}\label{thm:RS1}
Assume that $n$ is odd. 
\begin{itemize}
\item[(1)]Suppose that $E$ is inert over $F$. Then we have
\begin{align*}
&\prod_{i=1}^nL(i,1)L(i,\varepsilon_{E/F})\hana{R}^0(d;X,Y,t)\\
={}&\sum_{r=0}^n(q^r\varepsilon Y^2)^{n-r}\wP_r(d; \hat{\varepsilon}^{n-r}X,q^{-n/2}Y,\hat{\varepsilon}^{n-r}q^{-n/2}t)\\
&\times\dfrac{\prod_{i=1}^{n-r}(1-(\hat{\varepsilon}q)^{-r-n-i}t^2)\prod_{i=0}^{r-1}(1-\hat{\varepsilon}^n(\hat{\varepsilon}q)^iY^2)}{\phi_{n-r}(\varepsilon q^{-1})}.
\end{align*}
\item[(2)]Assume that $E=F\oplus F$.
Then we have
\begin{align*}
&\prod_{i=1}^nL(i,1)L(i,\varepsilon_{E/F})\hana{R}^0(d;X,Y,t)\\
={}&\sum_{r=0}^n(q^rY^2)^{n-r}\wP_r(d; X,q^{-n/2}Y,q^{-n/2}t)\times\dfrac{\prod_{i=1}^{n-r}(1-(\hat{\varepsilon}q)^{-r-n-i}t^2)\prod_{i=0}^{r-1}(1-\hat{\varepsilon}^n(\hat{\varepsilon}q)^iY^2)}{\phi_{n-r}(q^{-1})}.
\end{align*}
\item[(3)]
Suppose that $E$ is ramified over $F$. Then we have
\begin{align*}
&\prod_{i=1}^nL(i,1)L(i,\varepsilon_{E/F})\hana{R}^0(d;X,Y,t)\\
={}&\sum_{r=0}^{(n-1)/2}(tY^{-1})^{(n-2r-1)\maf{f}/2}\wP_{2r+1}((-1)^{(n-2r-1)/2}d; X,q^{-n/2}Y,q^{-n/2}t)\\
&\times\dfrac{(q^{2r+1}Y^2)^{(n-2r-1)/2}\prod_{i=1}^{(n-2r-1)/2}(1-q^{-2r-n-2i-1}t^2)\prod_{i=0}^{r-1}(1-q^{2i+1}Y^2)}{\phi_{(n-2r-1)/2}(q^{-2})}.
\end{align*}
\end{itemize}
\end{thms}
\begin{proof}
\begin{itemize}
\item[(1)]Assume that $E$ is inert over $F$.
Then, as in Lemma \ref{lem:RS1}, we have
\begin{align*}
&\sum_{r=0}^n\dfrac{\prod_{i=0}^{r-1}(1-\varepsilon^{n+i}q^{i}Y^2)\prod_{i=r}^{n-1}(1-\varepsilon^{-n+i}q^{-2n+i}Y^2t^2)}{\phi_{n-r}(\varepsilon q^{-1})}Q_{r}(d;\hat{\varepsilon}^{n-r}X,q^{-n/2}Y,\hat{\varepsilon}^{n-r}q^{-n/2}t)\\
={}&\sum_{r=0}^n\dfrac{\prod_{i=0}^{r-1}(1-\varepsilon^{n+i}q^{i}Y^2)\prod_{i=r}^{n-1}(1-\varepsilon^{-n+i}q^{-2n+i}Y^2t^2)}{\phi_{n-r}(\varepsilon q^{-1})}\\
&\times\prod_{i=1}^rL(i,1)L(i,\varepsilon_{E/F})\int_{\wHer_{r,d}}\wG(Z; \hat{\varepsilon}^{n-r}X, \hat{\varepsilon}^{n-r}q^{-n}tY)(\hat{\varepsilon}^{n-r})^{\ord(\det Z)}|\det Z|^{s-y}dZ\\
={}&\sum_{r=0}^n\prod_{i=0}^{r-1}(1-(\varepsilon q)^{n+i}(q^{-n/2}Y)^2)\prod_{i=r}^{n-1}(1-(\varepsilon q)^{n+i}(q^{-3n/2}Yt)^2)\\
&\times\prod_{i=1}^rL(i,1)L(i,\varepsilon_{E/F})\int_{\wHer_{r,d}}\wG(1_{n-r}\perp Z; X, q^{-n}tY)|\det Z|^{s-y}dZ.
\end{align*}
By Proposition \ref{prop:locden2} and Lemma \ref{lem:locden11}, we have
\begin{align*}
&\sum_{r=0}^n\dfrac{\prod_{i=0}^{r-1}(1-\varepsilon^{n+i}q^{i}Y^2)\prod_{i=r}^{n-1}(1-\varepsilon^{-n+i}q^{-2n+i}Y^2t^2)}{\phi_{n-r}(\varepsilon q^{-1})}Q_{r}(d;\hat{\varepsilon}^{n-r}X,q^{-n/2}Y,\hat{\varepsilon}^{n-r}q^{-n/2}t)\\
={}&\sum_{r=0}^n\sum_{Z\in\wHer_{r,d}/\K_r^0}\prod_{i=1}^nL(i,1)L(i,\varepsilon_{E/F})\int_{(1_{n-r}\perp Z)[\K_n^0]}\wG(Z_1; X, q^{-n}tY)|\det Z_1|^{s-y}\\
&\times G(Z_1,(q^{-n/2}Y)^2)\maf{B}(Z_1; q^{-3n/2}Yt)dZ_1.
\end{align*}
By the structure theory of $\dHer_n(F;\hana{O}_E)$, $\dHer_n(F;\hana{O}_E)_d/\K_{n}^0$ has a representative set $\bigcup_{r=0}^n\{1_{n-r}\perp B|B\in\wHer_{r,d}/\K_r^0\}$. 
\begin{align*}
&\sum_{r=0}^n\dfrac{\prod_{i=0}^{r-1}(1-\varepsilon^{n+i}q^{i}Y^2)\prod_{i=r}^{n-1}(1-\varepsilon^{-n+i}q^{-2n+i}Y^2t^2)}{\phi_{n-r}(\varepsilon q^{-1})}Q_{r}(d;\hat{\varepsilon}^{n-r}X,q^{-n/2}Y,\hat{\varepsilon}^{n-r}q^{-n/2}t)\\
={}&\prod_{i=1}^nL(i,1)L(i,\varepsilon_{E/F})\int_{\wHer_{n,d}}\wG(Z; X, q^{-n}tY)|\det Z|^{s-y}G(Z,(q^{-n/2}Y)^2)\maf{B}(Z; q^{-3n/2}Yt)dZ\\
={}&\prod_{i=1}^nL(i,1)L(i,\varepsilon_{E/F})\hana{R}^0(d;X,Y,t).
\end{align*}
On the other hand, by Proposition \ref{prop:RS1} and the Corollary to Proposition \ref{prop:RS2}, the assertion holds.
\begin{align*}
&\sum_{r=0}^n\dfrac{\prod_{i=0}^{r-1}(1-\varepsilon^{n+i}q^{i}Y^2)\prod_{i=r}^{n-1}(1-\varepsilon^{-n+i}q^{-2n+i}Y^2t^2)}{\phi_{n-r}(\varepsilon q^{-1})}Q_{r}(d;\hat{\varepsilon}^{n-r}X,q^{-n/2}Y,\hat{\varepsilon}^{n-r}q^{-n/2}t)\\
={}&\sum_{r=0}^n\dfrac{\prod_{i=0}^{r-1}(1-\varepsilon^{n+i}q^{i}Y^2)\prod_{i=r}^{n-1}(1-\varepsilon^{-n+i}q^{-2n+i}Y^2t^2)}{\phi_{n-r}(\varepsilon q^{-1})}\\
&\times \sum_{m=0}^r\dfrac{(-1)^{r-m}(\varepsilon q)^{(r-m)(r-m-1)/2}}{\phi_{r-m}(\varepsilon q^{-1})}\wP_{m}(d,\hat{\varepsilon}^{n-m}X,q^{-n/2}Y,\hat{\varepsilon}^{n-m}q^{-n/2}t)\\
={}&\sum_{m=0}^n\wP_{m}(d,\hat{\varepsilon}^{n-m}X,q^{-n/2}Y,\hat{\varepsilon}^{n-m}q^{-n/2}t)\\
&\times\sum_{l=0}^{n-m}\dfrac{\prod_{i=0}^{l+m-1}(1-\varepsilon^{n+i}q^{i}Y^2)\prod_{i=l+m}^{n-1}(1-\varepsilon^{-n+i}q^{-2n+i}Y^2t^2)}{\phi_{n-l-m}(\varepsilon q^{-1})}\dfrac{(-1)^{l}(\varepsilon q)^{l(l-1)/2}}{\phi_{l}(\varepsilon q^{-1})}
\end{align*}
By \cite[Lemma 3.4]{IK}, the above equation equals
\begin{align*}
&\sum_{m=0}^n(q^m\varepsilon Y^2)^{n-m}\wP_{m}(d,\hat{\varepsilon}^{n-m}X,q^{-n/2}Y,\hat{\varepsilon}^{n-m}q^{-n/2}t)\\
&\times
\dfrac{\prod_{i=0}^{m-1}(1-\varepsilon^{m}(\varepsilon q)^{i}Y^2)\prod_{i=m}^{n-1}(1-(\varepsilon q)^{-n-m-i}Y^2t^2)}{\phi_{n-m}(\varepsilon q^{-1})}.
\end{align*}
\item[(2)]Assume that $E=F\oplus F$. Then we have
\begin{align*}
&\sum_{r=0}^n\dfrac{\prod_{i=0}^{r-1}(1-q^{i}Y^2)\prod_{i=r}^{n-1}(1-q^{-2n+i}Y^2t^2)}{\phi_{n-r}(q^{-1})}Q_{r}(d;X,q^{-n/2}Y,q^{-n/2}t)\\
={}&\sum_{r=0}^n\dfrac{\prod_{i=0}^{r-1}(1-q^{i}Y^2)\prod_{i=r}^{n-1}(1-q^{-2n+i}Y^2t^2)}{\phi_{n-r}(q^{-1})}\\
&\prod_{i=1}^rL(i,1)L(i,\varepsilon_{E/F})\int_{\wHer_{r,d}}\wG(Z; X, q^{-n}tY)|\det Z|^{s-y}dZ\\
={}&\sum_{r=0}^n\prod_{i=0}^{r-1}(1-q^{n+i}(q^{-n/2}Y)^2)\prod_{i=r}^{n-1}(1-q^{n+i}(q^{-3n/2}Yt)^2)\\
&\times\sum_{B\in\wHer_{r,d}/\K_r^0}\prod_{i=1}^nL(i,1)L(i,\varepsilon_{E/F})\int_{(1_{n-r}\perp B)[\K_n^0]}\wG(Z; X, q^{-n}tY)|\det Z|^{s-y}dZ.
\end{align*}
By Proposition \ref{prop:locden2} and Lemma \ref{lem:locden11}, we have
\begin{align*}
&\sum_{r=0}^n\dfrac{\prod_{i=0}^{r-1}(1-q^{i}Y^2)\prod_{i=r}^{n-1}(1-q^{-2n+i}Y^2t^2)}{\phi_{n-r}(q^{-1})}Q_{r}(d;X,q^{-n/2}Y,q^{-n/2}t)\\
={}&\sum_{r=0}^n\sum_{B\in\wHer_{r,d}/\K_r^0}\prod_{i=1}^nL(i,1)L(i,\varepsilon_{E/F})\int_{(1_{n-r}\perp B)[\K_n^0]}\wG(Z; X, q^{-n}tY)|\det Z|^{s-y}\\
&\times G(Z,(q^{-n/2}Y)^2)\maf{B}(Z; q^{-3n/2}Yt)dZ\\
={}&\prod_{i=1}^nL(i,1)L(i,\varepsilon_{E/F})\int_{\wHer_{n,d}}\wG(Z; X, q^{-n}tY)|\det Z|^{s-y}G(Z,q^{-n}Y^2)\maf{B}(Z; q^{-3n/2}Yt)dZ\\
={}&\prod_{i=1}^nL(i,1)L(i,\varepsilon_{E/F})\hana{R}^0(d;X,Y,t).
\end{align*}
On the other hand, by Proposition \ref{prop:RS1} and the Corollary to Proposition \ref{prop:RS2} we have
\begin{align*}
&\sum_{r=0}^n\dfrac{\prod_{i=0}^{r-1}(1-q^{i}Y^2)\prod_{i=r}^{n-1}(1-q^{-2n+i}Y^2t^2)}{\phi_{n-r}(q^{-1})}Q_{r}(d;X,q^{-n/2}Y,q^{-n/2}t)\\
={}&\sum_{m=0}^n\wP_{m}(d,X,q^{-n/2}Y,q^{-n/2}t)\\
&\times\sum_{l=0}^{n-m}\dfrac{\prod_{i=0}^{l+m-1}(1-q^{i}Y^2)\prod_{i=l+m}^{n-1}(1-\varepsilon^{-n+i}q^{-2n+i}Y^2t^2)}{\phi_{n-l-m}(\varepsilon q^{-1})}\dfrac{(-1)^{l}(\varepsilon q)^{l(l-1)/2}}{\phi_{l}(\varepsilon q^{-1})}
\end{align*}
By \cite[Lemma 3.4]{IK}, the above equation equals
\begin{equation*}
\sum_{m=0}^n(q^m Y^2)^{n-m}\wP_{m}(d,X,q^{-n/2}Y,q^{-n/2}t)
\dfrac{\prod_{i=0}^{m-1}(1-q^{i}Y^2)\prod_{i=m}^{n-1}(1-q^{-n-m-i}Y^2t^2)}{\phi_{n-m}( q^{-1})}
\end{equation*}
\item[(3)]
Suppose that $E$ is ramified over $F$.
Then we have
\begin{align*}
&\sum_{r=0,\,n-r\in2\mb{Z}}^n\dfrac{\prod_{i=0}^{[r/2]-1}(1-q^{2i+1}Y^2)\prod_{i=[r/2]}^{(n-3)/2}(1-q^{-2n+2i+1}Y^2t^2)}{\phi_{(n-r)/2}(q^{-2})}Q_{r}(d;X,q^{-n/2}Y,q^{-n/2}t)\\
={}&\sum_{r=0,\,n-r\in2\mb{Z}}^n\dfrac{\prod_{i=0}^{[r/2]-1}(1-q^{2i+1}Y^2)\prod_{i=[r/2]}^{(n-3)/2}(1-q^{-2n+2i+1}Y^2t^2)}{\phi_{(n-r)/2}(q^{-2})}\\
&\times\prod_{i=1}^rL(i,1)L(i,\varepsilon_{E/F})\int_{\wHer_{r,d}}\wG(Z; X, q^{-n}tY)|\det Z|^{s-y}dZ\\
={}&\sum_{r=0,\,n-r\in2\mb{Z}}^n\prod_{i=0}^{[r/2]-1}(1-q^{2i+n+1}(q^{-n/2}Y)^2)\prod_{i=[r/2]}^{(n-3)/2}(1-q^{2i+n+1}(q^{-3n/2}Yt)^2)\\
&\times\sum_{B\in\wHer_{r,d}/\K_r^0}\prod_{i=1}^nL(i,1)L(i,\varepsilon_{E/F})\int_{(\Theta_{n-r}\perp B)[\K_n^0]}\wG(Z; X, q^{-n}tY)|\det Z|^{s-y}dZ.
\end{align*}
By Proposition \ref{prop:locden2} and Lemma \ref{lem:locden11}, we have
\begin{align*}
&\sum_{r=0,\,n-r\in2\mb{Z}}^n\dfrac{\prod_{i=0}^{[r/2]-1}(1-q^{2i+1}Y^2)\prod_{i=[r/2]}^{(n-3)/2}(1-q^{-2n+2i+1}Y^2t^2)}{\phi_{(n-r)/2}(q^{-2})}Q_{r}(d;X,q^{-n/2}Y,q^{-n/2}t)\\
={}&\sum_{r=0,\,n-r\in2\mb{Z}}^n
\sum_{B\in\wHer_{r,d}/\K_r^0}\prod_{i=1}^nL(i,1)L(i,\varepsilon_{E/F})\int_{(\Theta_{n-r}\perp B)[\K_n^0]}\wG(Z; X, q^{-n}tY)|\det Z|^{s-y}\\
&\times G(Z,(q^{-n/2}Y)^2)\maf{B}(Z;q^{-3n/2}Yt)dZ\\
={}&\prod_{i=1}^nL(i,1)L(i,\varepsilon_{E/F})\\
&\times\int_{\wHer_{n,d}}\wG(Z; X, q^{-n}tY)|\det Z|^{s-y}G(Z,(q^{-n/2}Y)^2)\maf{B}(Z; q^{-3n/2}Yt)dZ\\
={}&\prod_{i=1}^nL(i,1)L(i,\varepsilon_{E/F})\hana{R}^0(d;X,Y,t).
\end{align*}
On the other hand, substituting the expression for $Q_r$ given by the Corollary to Proposition~\ref{prop:RS2} into the left-hand side and using Proposition~\ref{prop:RS1}, we obtain
\begin{align*}
&\sum_{r=0,\,n-r\in2\mb{Z}}^n\dfrac{\prod_{i=0}^{[r/2]-1}(1-q^{2i+1}Y^2)\prod_{i=[r/2]}^{(n-3)/2}(1-q^{-2n+2i+1}Y^2t^2)}{\phi_{(n-r)/2}(q^{-2})}Q_{r}(d;X,q^{-n/2}Y,q^{-n/2}t)\\
={}&\sum_{r=0,\,n-r\in2\mb{Z}}^n\dfrac{\prod_{i=0}^{[r/2]-1}(1-q^{2i+1}Y^2)\prod_{i=[r/2]}^{(n-3)/2}(1-q^{-2n+2i+1}Y^2t^2)}{\phi_{(n-r)/2}(q^{-2})}\\
&\times\sum_{l=0}^{(r-1)/2}\dfrac{(-1)^lq^{l(l-1)/2}}{\phi_l(q^{-2})}\wP_{r-2l}(d;X,q^{-n/2}Y,q^{-n/2}t).
\end{align*}
Reindexing the double sum by $m=r-2l$, we obtain
\begin{align*}
&\sum_{r=0,\,n-r\in2\mb{Z}}^n\dfrac{\prod_{i=0}^{[r/2]-1}(1-q^{2i+1}Y^2)\prod_{i=[r/2]}^{(n-3)/2}(1-q^{-2n+2i+1}Y^2t^2)}{\phi_{(n-r)/2}(q^{-2})}Q_{r}(d;X,q^{-n/2}Y,q^{-n/2}t)\\
={}&\sum_{m=0,\text{ $n-m$ is even}}^n\wP_{m}(d,X,q^{-n/2}Y,q^{-n/2}t)\\
&\times\sum_{l=0}^{(n-m)/2}\dfrac{\prod_{i=0}^{l+(m-1)/2-1}(1-q^{2i+1}Y^2)\prod_{i=l+(m-1)/2}^{(n-1)/2}(1-q^{-2n+2i+1}Y^2t^2)}{\phi_{(n-m-2l)/2}(q^{-2})}\dfrac{(-1)^{l}q^{l(l-1)/2}}{\phi_{l}(q^{-2})}.
\end{align*}
By \cite[Lemma 3.4]{IK}, the above equation equals
\begin{align*}
&\sum_{m=0,\text{ $n-m$ is even}}^n(q^mY^2)^{n-m}\wP_{m}(d,X,q^{-n/2}Y,q^{-n/2}t)\\
&\times\dfrac{\prod_{i=0}^{(m-1)/2-1}(1-q^{2i+1}Y^2)\prod_{i=(m-1)/2}^{(n-1)/2}(1-q^{-2n+2i+1}Y^2t^2)}{\phi_{(n-m)/2}(q^{-2})}.
\end{align*}
Therefore, the assertion holds.
\end{itemize}
\end{proof}
Then we have the following corollary by Proposition \ref{prop:KS1}.
\begin{Cor}\label{CorofthmRS1}
\begin{itemize}
\item[(1)]Suppose that $E$ is inert over $F$.
Then we have
\begin{align*}
&\prod_{i=1}^nL(i,1)L(i,\varepsilon_{E/F})\hana{R}^0(d;X,Y,t)\\
={}&\sum_{r=0}^n(q^r\varepsilon Y^2)^{n-r}
P_r(d,\hat{\varepsilon}^{n-r}X,\hat{\varepsilon}^{n-r}tY^{-1})\prod_{i=1}^r(1-t^4q^{-2r-2-2n+2i})\\
&\times\dfrac{\prod_{i=1}^{n-r}(1-(\hat{\varepsilon}q)^{-r-n-i}t^2)\prod_{i=0}^{r-1}(1-\hat{\varepsilon}^n(\hat{\varepsilon}q)^iY^2)}{\phi_{n-r}(\hat{\varepsilon}q^{-1})}
\end{align*}
\item[(2)]Assume that $E=F\oplus F$. Then we have
\begin{align*}
&\prod_{i=1}^nL(i,1)L(i,\varepsilon_{E/F})\hana{R}^0(d;X,Y,t)\\
={}&\sum_{r=0}^n(q^rY^2)^{n-r}P_r(d,X,tY^{-1})\prod_{i=1}^r(1-t^2q^{-r-1-n+i})^2\\
&\times\dfrac{\prod_{i=1}^{n-r}(1-q^{-r-n-i}t^2)\prod_{i=0}^{r-1}(1-q^iY^2)}{\phi_{n-r}(q^{-1})}
\end{align*}
\item[(3)]
Suppose that $E$ is ramified over $F$ and $n$ is odd.
Then we have
\begin{align*}
&\prod_{i=1}^nL(i,1)L(i,\varepsilon_{E/F})\hana{R}^0(d;X,Y,t)\\
={}&\sum_{r=0}^{(n-1)/2}(tY^{-1})^{(n-2r-1)\maf{f}/2}P_{2r+1}((-1)^{(n-2r-1)/2}d,X,tY^{-1})\prod_{i=1}^r(1-t^2q^{-r-1-n+i})\\
&\times\dfrac{(q^{2r+1}Y^2)^{(n-2r-1)/2}\prod_{i=1}^{(n-2r-1)/2}(1-q^{-2r-n-2i-1}t^2)\prod_{i=0}^{r-1}(1-q^{2i+1}Y^2)}{\phi_{(n-2r-1)/2}(q^{-2})}.
\end{align*}
\end{itemize}
\end{Cor}

\subsection{The Local Koecher--Maass Integrals}\label{subsec5-3}
In this subsection, we compute the series $P_m$.
We retain the notation
\begin{equation*}
\phi_m(t)=\prod_{i=1}^m(1-t^i)
\end{equation*}
from Section~\ref{subsec5-2}. 
We also put
\begin{equation*}
\zeta_m(d;s)=\int_{\wHer_{m,d}}|\det Z|^s\,dZ,
\qquad
\Phi_m(X,Y)=\prod_{i=1}^m(1-XY^{i-1}).
\end{equation*}
\begin{lems}\label{lem:KS1}
\begin{itemize}
\item[(1)]Assume that $E$ is inert over $F$.
For $d\in\hana{O}_F^\times$, we have
\begin{equation*}
\zeta_{m}(d;s-m)=\dfrac{\Phi_{[(m+1)/2]}(q^{-1},q^{-2})\Phi_{[m/2]}(-q^{-2},q^{-2})q^{-ms}}{\prod_{i=1}^m(1-(-1)^{-m+i}q^{-s+i-1})}
\end{equation*}
\item[(2)]Assume that $E$ is split over $F$. For $d\in\hana{O}_F^\times$, we have
\begin{equation*}
\zeta_m(d;s-m)=\dfrac{\phi_m(q^{-1})q^{-ms}}{\prod_{i=1}^m(1-q^{i-1-s})}
\end{equation*}
\item[(3)]Assume that $E$ is ramified over $F$ and $m$ is odd.
For $d\in\hana{O}_F^\times$, we have
\begin{align*}
\zeta_m(d;s-m)
={}&\begin{cases}
\dfrac{\Phi_{(m+1)/2}(q^{-1},q^{-2})}{2}\dfrac{q^{(\maf{f}-1)(m-1)s/2}}{\prod_{i=1}^{(m+1)/2}(1-q^{2i-2-s})}&\text{if $\maf{f}$ is odd},\\
\dfrac{\Phi_{(m+1)/2}(q^{-1},q^{-2})}{2}\dfrac{q^{(\maf{f}-1)(m-1)s/2}}{\prod_{i=1}^{(m+1)/2}(1-q^{2i-2-s})}
&\text{if $\maf{f}$ is even}.
\end{cases}
\end{align*}
\item[(4)]Assume that $E$ is ramified over $F$ and $m$ is even.
For $d\in\hana{O}_F^\times$, we have
\begin{align*}
&\zeta_m(d;s-m)\\
={}&\begin{cases}
\dfrac{1}{2}\Phi_{m/2}(q^{-1},q^{-2})q^{[\maf{f}/2]m(s-1)}&\\
\times\left(\dfrac{1}{\prod_{i=1}^{m/2}(1-q^{2i-1-s})}+\dfrac{\varepsilon_{E/F}((-1)^{m/2}d)q^{-m/2}}{\prod_{i=1}^{m/2}(1-q^{2i-2-s})}\right)&\text{if $\maf{f}$ is odd,}\\
\dfrac{1}{2}\Phi_{m/2}(q^{-1},q^{-2})q^{[\maf{f}/2]ms}q^{(\maf{f}-1)m/2}&\\
\times\left(\dfrac{1}{\prod_{i=1}^{m/2}(1-q^{2i-1-s})}+\dfrac{\varepsilon_{E/F}((-1)^{m/2}d)q^{-m/2}}{\prod_{i=1}^{m/2}(1-q^{2i-2-s})}\right)
&\text{if $\maf{f}$ is even}.
\end{cases}
\end{align*}
\end{itemize}
\end{lems}
\begin{proof}
If $E$ is unramified, we can easily see $\Her_m(F;\hana{O}_E)_d=\Her_m(F;\hana{O}_E)$.
By definition of $\wHer_{m,d}$ for $d\in\hana{O}_F^{\times}$, we have
\begin{align*}
    \wHer_{m,d}={}&\dHer_m(F;\hana{O}_E)_d\cap\Her_{m,*}\\
    ={}&\begin{cases}
    \maf{p}\Her_m(F;\hana{O}_E)&\text{ if $E$ is unramified over $F$,}\\
    \left\{z=(z_{ij})\in\Her_m(F)\middle|\begin{matrix}z_{ij}\in\maf{q}\maf{d}_{E/F}^{-1}\\
    z_{ii}\in\hana{O}_F\\
    \varepsilon_{E/F}(\det z)=\varepsilon_{E/F}(d)
    \end{matrix}\right\}&\text{if $E$ is ramified over $F$.}
    \end{cases}
\end{align*}
If $E$ is unramified over $F$, we can use the result in \cite[Theorem 4.2 (1),(2)]{SH1} for $\varpi^{-1}\wHer_{m,d}$.
Assume that $E$ is ramified over $F$.

If $\ord_v(2)=0$, $\maf{f}_v=1$ and $\wHer_{m,d}=\Her_m(F;\hana{O}_E)_d$. 
Thus, we can use the result in \cite[Theorem 4.2(3)]{SH1}.

We also assume that $\ord_v(2)>0$. Then, $\maf{f}_v=2l$ for some integer $l$ with $1\leqq l\leqq \ord_v(2)$ or $\maf{f}_v=2\ord_v(2)+1$. Put $l=[\maf{f}/2]$.
Put
\begin{align*}
\bm{X}_d={}&\varpi^{l}\wHer_{m,d}
=\begin{cases}\left\{z=(z_{ij})\in\Her_m(F)\middle|\begin{matrix}z_{ij}\in\maf{q}\\
z_{ii}\in\maf{p}^{l}\\
\varepsilon_{E/F}(\det z)=\varepsilon_{E/F}(d)
\end{matrix}\right\}&\text{if $\maf{f}$ is even},\\
\left\{z=(z_{ij})\in\Her_m(F)\middle|\begin{matrix}z_{ij}\in\hana{O}_E\\
z_{ii}\in\maf{p}^{l}\\
\varepsilon_{E/F}(\det z)=\varepsilon_{E/F}(d)
\end{matrix}\right\}&\text{if $\maf{f}$ is odd.}
\end{cases}
\end{align*}
Although the result in \cite[Theorem 4.2(3)]{SH1} does not apply to the zeta integrals on $\bm{X}_d$, they can be calculated using an argument similar to the proof of \cite[Theorem 4.2(3)]{SH1}.
Put $u=q^{-s}$. 
Then
\begin{align*}
    &\int_{\bm{X}_d}|\det Z|^{s-m}dZ\\
    ={}&\begin{cases}
    \dfrac{\Phi_{(m+1)/2}(q^{-1},q^{-2})}{2}\dfrac{u^{(\maf{f}-1)/2}}{\prod_{i=1}^{(m+1)/2}(1-q^{2i-2}u)}&\text{if $m$ is odd, $\maf{f}$ is odd,}\\
    \dfrac{\Phi_{(m+1)/2}(q^{-1},q^{-2})}{2}\dfrac{u^{(m+\maf{f}-1)/2}}{\prod_{i=1}^{(m+1)/2}(1-q^{2i-2}u)}&\text{if $m$ is odd, $\maf{f}$ is even,}\\
    \dfrac{\Phi_{m/2}(q^{-1},q^{-2})}{2}\\
    \times\left(\dfrac{q^{-(\maf{f}-1)m/2}}{\prod_{i=1}^{m/2}(1-q^{2i-1}u)}+\dfrac{\varepsilon_{E/F}((-1)^{m/2}d)q^{-\maf{f}m/2}}{\prod_{i=1}^{m/2}(1-q^{2i-2}u)}\right)&\text{if $m$ is even, $\maf{f}$ is odd,}\\
    \dfrac{\Phi_{m/2}(q^{-1},q^{-2})}{2}\\
    \times\left(\dfrac{q^{-(\maf{f}-1)m/2}u^{m/2}}{\prod_{i=1}^{m/2}(1-q^{2i-1}u)}+\dfrac{\varepsilon_{E/F}((-1)^{m/2}d)q^{-\maf{f}m/2}u^{m/2}}{\prod_{i=1}^{m/2}(1-q^{2i-2}u)}\right)&\text{if $m$ is even, $\maf{f}$ is even,}\\
    \end{cases}
\end{align*}
By definition, we have
\begin{align*}
    &\zeta_m(d;s-m)=u^{-[\maf{f}/2]m}\int_{\bm{X}_d}|\det Z|^{s-m}dZ\\
    ={}&\begin{cases}
    \dfrac{\Phi_{(m+1)/2}(q^{-1},q^{-2})}{2}\dfrac{u^{-(\maf{f}-1)(m-1)/2}}{\prod_{i=1}^{(m+1)/2}(1-q^{2i-2}u)}&\text{if $m$ is odd, $\maf{f}$ is odd,}\\
    \dfrac{\Phi_{(m+1)/2}(q^{-1},q^{-2})}{2}\dfrac{u^{-(\maf{f}-1)(m-1)/2}}{\prod_{i=1}^{(m+1)/2}(1-q^{2i-2}u)}&\text{if $m$ is odd, $\maf{f}$ is even,}\\
    \dfrac{\Phi_{m/2}(q^{-1},q^{-2})}{2}u^{-(\maf{f}-1)m/2}q^{-(\maf{f}-1)m/2}\\
    \times\left(\dfrac{1}{\prod_{i=1}^{m/2}(1-q^{2i-1}u)}+\dfrac{\varepsilon_{E/F}((-1)^{m/2}d)q^{-m/2}}{\prod_{i=1}^{m/2}(1-q^{2i-2}u)}\right)&\text{if $m$ is even, $\maf{f}$ is odd,}\\
    \dfrac{\Phi_{m/2}(q^{-1},q^{-2})}{2}u^{-\maf{f}m/2}q^{-(\maf{f}-1)m/2}\\
    \times\left(\dfrac{1}{\prod_{i=1}^{m/2}(1-q^{2i-1}u)}+\dfrac{\varepsilon_{E/F}((-1)^{m/2}d)q^{-m/2}}{\prod_{i=1}^{m/2}(1-q^{2i-2}u)}\right)&\text{if $m$ is even, $\maf{f}$ is even,}\\
    \end{cases}
\end{align*}
\end{proof}

Write $X=q^{-x}$ and $t=q^{m-s}$.

\begin{thms}\label{thm:KS}
\begin{itemize}
\item[(1)]Assume that $E$ is inert over $F$. Then for $d\in\hana{O}_F^\times$, we have
\begin{align*}
P_m(d;X,t)={}&\dfrac{\prod_{i=1}^m(1-(-q)^{-i})^{-1}}{\prod_{i=1}^m(1+(-1)^iq^{-i}tX^{-1})\prod_{i=1}^m(1+(-1)^{i+m-1}q^{-i}tX)}
\end{align*}
\item[(2)]Assume that $E$ is split over $F$.
Then for $d\in\hana{O}_F^\times$, we have
\begin{align*}
P_m(d;X,t)&=\dfrac{\prod_{i=1}^m(1-q^{-i})^{-1}}{\prod_{i=1}^m(1-q^{-i}tX^{-1})\prod_{i=1}^m(1-q^{-i}tX)}
\end{align*}
\item[(3)]Assume that $E$ is ramified over $F$. Then for $d\in\hana{O}_F^\times$, if $m$ is even, we have
\begin{align*}
&P_m(d;X,t)\\
={}&\dfrac{q^{\maf{f}m(m-1)/2}}{2\prod_{j=1}^{m/2}(1-q^{-2j})}
\dfrac{t^{-\maf{f}m/2}}{
\prod_{i=1}^{m/2}(1-q^{-2i+1}tX^{-1})\prod_{i=1}^{m/2}(1-tXq^{-2i})}\\
&+\dfrac{q^{\maf{f}m(m-1)/2}}{2\prod_{j=1}^{m/2}(1-q^{-2j})}\dfrac{\varepsilon_{E/F}((-1)^{m/2}d)t^{-\maf{f}m/2}}{
\prod_{i=1}^{m/2}(1-q^{-2i}tX^{-1})\prod_{i=1}^{m/2}(1-tXq^{-2i+1})}.
\end{align*}
If $m$ is odd, we have
\begin{align*}
P_m(d;X,t)={}&\dfrac{q^{\maf{f}m(m-1)/2}}{2\prod_{j=1}^{(m-1)/2}(1-q^{-2j})}\dfrac{t^{-\maf{f}(m-1)/2}}{
\prod_{i=1}^{(m+1)/2}(1-q^{-2i+1}tX^{-1})\prod_{i=1}^{(m+1)/2}(1-tXq^{-2i+1})}.
\end{align*}
\end{itemize}
\end{thms}
\begin{proof}
Put 
\begin{equation*}
\hat{P}_{m}(d;X,t)=\int_{\dHer_m(F;\hana{O}_E)_d}G(Z;q^{-m}X^2)|\det Z|^{s-x}dZ.
\end{equation*}
Then, by Proposition \ref{prop:locden3}, we have
\begin{align*}
&P_m(d;X,t)\\={}&X^{-[m/2]\maf{f}}\int_{\dHer_m(F;\hana{O}_E)_d}\int_{\GL_m(E)}|\det Z|^{-x}G(Z[h^{-1}];q^{-m-2x})|\det h|_E^{2x}dh|\det Z|^sdZ\\
={}&X^{-[m/2]\maf{f}}\hat{P}_{m}(d;X,t)\times\begin{cases} 
\prod_{i=1}^m(1-t^2X^2q^{2i-2-2m})^{-1}&\text{if $E$ is inert over $F$,}\\
\prod_{i=1}^m(1-tXq^{i-1-m})^{-2}&\text{if $E$ is split over $F$,}\\
\prod_{i=1}^m(1-tXq^{i-1-m})^{-1}&\text{if $E$ is ramified over $F$.}
\end{cases}
\end{align*}
Put $d(m-r)=d\cdot(\det\Theta_{m-r})^{-1}$ for $d\in F^\times$.
Then, by the structure of $\dHer_m(F;\hana{O}_E)$, we have
\begin{align*}
&\hat{P}_{m}(d;X,t)\\
={}&\sum_{r=0}^m\sum_{B\in\wHer_{r,d(m-r)}/\K_r^0}\int_{(\Theta_{m-r}\perp B)[\K_m^0]}G(Z;q^{-m}X^2)|\det Z|^{s-x}dZ\\
={}&\begin{cases}
\sum_{r=0}^m\dfrac{\prod_{l=0}^{r-1}(1-(\varepsilon q)^{m+l}\cdot q^{-m}X^2)}{\phi_{(m-r)}(\varepsilon q^{-1})}\sum_{B\in\wHer_{r,d}/\K_r^0}\int_{B[\K_r^0]}|\det Z|^{s-x}dZ&\text{$E$ is unramified over $F$}\\
\sum_{r=0,\text{$m-r$ is even}}^m(tX^{-1})^{-(m-r)/2\cdot\maf{f}}\dfrac{\prod_{i=0}^{[r/2]-1}(1-q^{2i+2[(m+1)/2]}\cdot q^{-m}X^2)}{\phi_{(m-r)/2}(q^{-2})}&\\
\times\sum_{B\in\wHer_{r,d(m-r)}/\K_r^0}\int_{B[\K_r^0]}|\det Z|^{s-x}dZ&\text{$E$ is ramified over $F$}
\end{cases}\\
={}&\begin{cases}
\sum_{r=0}^m\dfrac{\prod_{l=0}^{r-1}(1-(\varepsilon q)^{m+l}\cdot q^{-m}X^2)}{\phi_{(m-r)}(\varepsilon q^{-1})}\zeta_r(d;s-x)&\text{$E$ is unramified over $F$}\\
\sum_{r=0,\text{$m-r$ is even}}^m(tX^{-1})^{-(m-r)/2\cdot\maf{f}}\dfrac{\prod_{i=0}^{[r/2]-1}(1-q^{2i+2[(m+1)/2]}\cdot q^{-m}X^2)}{\phi_{(m-r)/2}(q^{-2})}&\\
\times\zeta_r(d(m-r);s-x)&\text{$E$ is ramified over $F$}.
\end{cases}
\end{align*}

We first assume that $E$ is unramified over $F$.

We put 
\begin{equation*}
    a_{m}(t)=
    \prod_{i=1}^m(1-\varepsilon^{-i+1} q^{-i}t).
\end{equation*}

By Lemma \ref{lem:KS1}, $\zeta_m(d;s-m)a_m(q^{m-s})$ is a monomial of $q^{m-s}$ and the degree of $\zeta_m(d;s-m)a_m(q^{m-s})$ is $m$.

Put $T(d;X,t)=\hat{P}_m(d;X,t)a_m(q^{-m}X^{-1}t)$.

Then we have
\begin{align*}
    &T(d;q^{-x},q^{m-s})\\
    ={}&\sum_{r=0}^m\dfrac{\prod_{l=0}^{r-1}(1-(\varepsilon q)^{m+l}\cdot q^{-m-2x})}{\phi_{(m-r)}(\varepsilon q^{-1})}\zeta_r(d;(s-x+r)-r)a_r(q^{r-(s-x+r)})\cdot\dfrac{a_m(q^{-m+x}\cdot q^{m-s})}{a_r(q^{-r+x}\cdot q^{r-s})}\\
    ={}&\sum_{r=0}^m\dfrac{\prod_{l=0}^{r-1}(1-(\varepsilon q)^{m+l}\cdot q^{-m-2x})}{\phi_{(m-r)}(\varepsilon q^{-1})}c_{E,r}q^{-r(s-x+r)}\cdot\prod_{i=r+1}^m(1-\varepsilon^{-i+1} q^{-i}q^{-s+x})
\end{align*}
where 
\begin{equation*}
    c_{E,m}=\begin{cases}
        \Phi_{[(m+1)/2]}(q^{-1},q^{-2})\Phi_{[m/2]}(-q^{-2},q^{-2})&\text{if $E$ is inert over $F$,}\\
        \phi_m(q^{-1})&\text{if $E$ is split over $F$}.
    \end{cases}
\end{equation*}

Thus, $X^mT(d;X,t)$ is the polynomial in $\mb{C}[X,t]$, and $T(d;X,t)$, view as a polynomial in $t$ over $\mb{C}(X)$ has degree $m$.

Then, we have 
\begin{align*}
P_m(d;X,t)={}&\dfrac{T(d;X,t)}{\prod_{i=1}^m(1-\varepsilon^{i-1}q^{-i}tX^{-1})\prod_{j=1}^m(1-tXq^{-j})(1-\varepsilon tXq^{-j})},
\end{align*}
Since $\wF(B;X^{-1})=\varepsilon_{E/F}(\gamma_0(B))^{m-1}\wF(B;X)$ for any $B\in\Her_m(F)^{nd}$ we have 
\begin{equation*}
    P_m(d;X^{-1},t)=P_m(d;X,\varepsilon^{m-1} t)
\end{equation*}
by definition.
Then, this equation induces the following equation:
\begin{align*}
\dfrac{T(d;X^{-1},t)X^{2m}\prod_{i=1}^m(1-\varepsilon^{i+m-1}q^{-i}tX)}{\prod_{i=1}^m(X-\varepsilon^{i}q^{-i}t)}={}&T(d;X,\varepsilon^{m-1} t)X^m
\end{align*}
and $T(d;X^{-1},t)X^{2m}$ is divisible by $(X+\varepsilon^{i-1}q^{-i}t)$ for $1\leqq i\leqq m$.
Since the degree of $T(d;X^{-1},t)$ in $\mb{C}(X)[t]$ is $m$, the above equation implies that there exists a constant $C$ such that
\begin{align*}
P_m(d;X,t)={}&\dfrac{C}{\prod_{i=1}^m(1-\varepsilon^{i-1}q^{-i}tX^{-1})\prod_{i=1}^m(1-\varepsilon^{i+m}q^{-i}tX)}.
\end{align*}
Now we have
\begin{align*}
    &\lim_{t\to0}\lim_{X\to0}P_m(d;X,Xt)\\
    ={}&\lim_{s\to\infty}\lim_{X\to0}\dfrac{1}{\prod_{i=1}^m(1-q^{-i})(1-\varepsilon q^{-i})}\int_{\dHer_m(F;\hana{O}_E)_d}X^{\ord_v(\det Z)}\wF_v(Z;X)|\det Z|^{s}dZ\\
    ={}&\lim_{s\to\infty}\dfrac{1}{\prod_{i=1}^m(1-q^{-i})(1-\varepsilon q^{-i})}\int_{\dHer_m(F;\hana{O}_E)_d}|\det Z|^sdZ\\
    ={}&\dfrac{\vol(\{x\in\Her_m(F;\hana{O}_E)|\det x\in\hana{O}_F^{\times}\},d\mu_{\Her_m})}{\prod_{i=1}^m(1-q^{-i})(1-\varepsilon q^{-i})}\\
    ={}&\prod_{i=1}^m(1-(\varepsilon q)^{-i})^{-1}.
\end{align*}
Comparing coefficients, we have
\begin{align*}
P_m(d;X,t)={}&\dfrac{\prod_{i=1}^m(1-(\varepsilon q)^{-i})^{-1}}{\prod_{i=1}^m(1-\varepsilon^{i-1}q^{-i}tX^{-1})\prod_{i=1}^m(1-\varepsilon^{i+m}q^{-i}tX)}.
\end{align*}

Second, we assume that $E$ is ramified.
When $m$ is odd, we put 
\begin{align*}
    a_m(t)={}&\prod_{i=1}^{(m+1)/2}(1-q^{2i-2-m}t)\\ 
    q_{m,r}(q^{-x})={}&\dfrac{\prod_{i=0}^{[r/2]-1}(1-q^{2i+2[(m+1)/2]}\cdot q^{-m-2x})}{\phi_{(m-r)/2}(q^{-2})}.
\end{align*} 
Then $\zeta_m(d;s-m)a_m(q^{m-s})$ is a Laurent monomial in $q^{-s}$ with exponent $-(\maf{f}-1)(m-1)/2$.
Let $c_{E,m}\in\mb{C}$ such that $\zeta_m(d;s-m)a_{m}(q^{m-s})=c_{E,m}q^{-(\maf{f}-1)(m-1)(m-s)/2}$.
We put $T(d;X,t)=t^{\maf{f}[m/2]}\hat{P}_m(d;X,t)a_m(q^{-m}Xt)$. 
Then we have
\begin{align*}
    &T(d;q^{-x},q^{m-s})\\
    ={}&\sum_{r=0,\text{$r$ is odd}}^mq^{(m(r-1)-(m-1)x)/2\cdot\maf{f}}q_{m,r}(q^{-x})c_{E,r}q^{(r-1)(x-s)/2}\prod_{i=(r+3)/2}^{(m+1)/2}(1-q^{-2i+1}q^{x-s}),
\end{align*}
$T(d;X,t)\in\mb{C}[X,X^{-1},t]$ and $T(d;X,t)$, view as a polynomial in $t$ over $\mb{C}(X)$ has degree $(m-1)/2$.

Then we have
\begin{align*}
    P_m(d;X,t)=\dfrac{t^{-\maf{f}[m/2]}T(d;X,t)}{\prod_{i=1}^{(m+1)/2}(1-q^{2i-2-m}X^{-1}t)\prod_{j=1}^m(1-q^{j-1-m}Xt)}.
\end{align*}

Since $\wF(B;X^{-1})=\wF(B;X)$ for $B\in\Her_n(F)^{nd}$, $P_m(d;X,t)=P_m(d;X^{-1},t)$ by definition and
\begin{equation*}
    \dfrac{T(d;X^{-1},t)\prod_{j=1}^{(m-1)/2}(1-q^{2j-1}Xt)}{\prod_{j=1}^{(m-1)/2}(1-q^{2j-1}X^{-1}t)}=T(d;X,t).
\end{equation*}
Since the degree of $T(d;X,t)$ is $(m-1)/2$, there exists $C\in\mb{C}$ such that
\begin{align*}
    P_m(d;X,t)=\dfrac{Ct^{-\maf{f}[m/2]}}{\prod_{i=1}^{(m+1)/2}(1-q^{2i-2-m}X^{-1}t)\prod_{j=1}^{(m+1)/2}(1-q^{2j-2-m}Xt)}.
\end{align*}

We fix $d_E\in\hana{O}_F^{\times}$ such that $F^{\times}=N_{E/F}(E^{\times})\sqcup d_EN_{E/F}(E^{\times})$. 
Since we assume that $m$ is odd, $\wHer_{m,1}=d_E^{-1}\wHer_{m,d_E}$ and $\Her_{m,*}=\wHer_{m,1}\sqcup \wHer_{m,d_E}$.
Then 
\begin{align*}
    &\lim_{X\to0}X^{\maf{f}[m/2]}P_m(d;X,Xq^{m-s})\\
    ={}&\lim_{X\to0}\dfrac{1}{\prod_{i=1}^m(1-q^{-i})}\int_{\dHer_m(F;\hana{O}_E)_d}X^{e(Z)}\wF_v(Z;X)|\det Z|^{s-m}dZ\\
    ={}&\dfrac{q^{[(\maf{f}+1)/2]ms}}{\prod_{i=1}^m(1-q^{-i})}\int_{\varpi^{[(\maf{f}+1)/2]}\dHer_m(F;\hana{O}_E)_d}|\det Z|^{s-m}dZ
\end{align*}
where $e(Z)=e_v(Z)$ is defined in Section $\ref{subsec2-3}$.

When $m$ is even, we define the polynomial $T^l(X,t)$ for $l\in\{0,1\}$ instead of $T(d;X,t)$.
Put
\begin{align*}
P_m^l(X,t)={}&\dfrac{1}{2}\sum_{d\in \hana{O}_{F}^\times/N_{E/F}(\hana{O}_{E}^\times)}\varepsilon_{E/F}^l((-1)^{m/2}d)P_{m}(d;X,t)\\
\hat{P}_m^l(X,t)={}&\dfrac{1}{2}\sum_{d\in \hana{O}_{F}^\times/N_{E/F}(\hana{O}_{E}^\times)}\varepsilon_{E/F}^l((-1)^{m/2}d)\hat{P}_{m}(d;X,t)
\end{align*}
for $l\in\{0,1\}$.

Put 
\begin{equation*}
    a_m^l(t)=\begin{cases}
    \prod_{i=1}^{m/2}(1-q^{2i-1-m}t)&\text{if $l=0$}\\
    \prod_{i=1}^{m/2}(1-q^{2i-2-m}t)&\text{if $l=1$}
    \end{cases}=\prod_{i=1}^{m/2}(1-q^{-2i+1-l}t).
\end{equation*}

For $l\in\{0,1\}$, there exists $c_{E,m}^l\in\mb{C}$ such that
\begin{align*}
    c_{E,m}^lq^{[\maf{f}/2]m(s-1)}=\dfrac{1}{2}\sum_{d\in\hana{O}_F^{\times}/N_{E/F}(\hana{O}_E^{\times})}\varepsilon_{E/F}^l((-1)^{m/2}d)\zeta_m(d;s-m)a_m^l(q^{m-s}).
\end{align*}

We put $T^l(X,t)=t^{\maf{f}[m/2]}\hat{P}_m^l(X,t)a_m^l(q^{-m}X^{-1}t)$, $q_{m,r}(X)=\dfrac{\prod_{i=0}^{r/2-1}(1-q^{2i}X^2)}{\phi_{(m-r)/2}(q^{-2})}\in\mb{Q}[X]$ and $\delta_{\bmf{f}}=\maf{f}/2-[\maf{f}/2]$.

Then there exists $0\ne C_{m,r}^l\in\mb{Q}$ for $0\leq r\leq m$ and $l\in\{0,1\}$ such that 
\begin{align*}
    &T^l(q^{-x},q^{m-s})\\
    ={}&\sum_{r=0,\text{$r$ is even}}^mC_{m,r}^lq_{m,r}(q^{-x})q^{-x(\maf{f}m/2-\delta_{\bmf{f}}r)}q^{-s\delta_{\bmf{f}}r}\prod_{i=r/2+1}^{m/2}(1-q^{-2i+1-l}q^{x-s}).
\end{align*}
By the above computation, $T^l(X,t)\in\mb{C}[X,X^{-1},t]$ and $T^l(X,t)$, view as a polynomial in $t$ over $\mb{C}(X)$ has degree $m/2$ for $l\in\{0,1\}$.

Then we have 
\begin{align*}
P_m^l(X,t)=\dfrac{1}{2}\dfrac{t^{-\maf{f}[m/2]}T^l(X,t)}{
\prod_{i=1}^{m/2}(1-q^{-2i+1-l}tX^{-1})\prod_{i=1}^m(1-tXq^{-m+i-1})}
\end{align*}
for $l\in\{0,1\}$.

Since $\wF(B;X^{-1})=\varepsilon_{E/F}(\gamma_0(B))^{m-1}\wF(B;X)$ for any $B\in\Her_m(F)^{nd}$ we have $P_m^0(X^{-1},t)=P_m^1(X,t)$ by definition.

This means that
\begin{align*}
\dfrac{1}{2}\dfrac{T^0(X^{-1},t)}{
\prod_{i=1}^{m/2}(1-q^{-2i+1}tX)\prod_{i=1}^m(1-tX^{-1}q^{-i})}=\dfrac{1}{2}\dfrac{T^1(X,t)}{
\prod_{i=1}^{m/2}(1-q^{-2i}tX^{-1})\prod_{i=1}^m(1-tXq^{-i})}.
\end{align*}

Since the degrees of both are $m/2$, there exists a constant $C_E$ such that
\begin{align*}
&P_m^l(X,t)\\
={}&\dfrac{C_Et^{-\maf{f}[m/2]}}{
\prod_{i=1}^{m/2}(1-q^{-2i+1-l}tX^{-1})\prod_{i=1}^{m/2}(1-tXq^{-2i+l})}. 
\end{align*}
Therefore, we have
\begin{align*}
&P_m(d;X,t)\\
={}&\dfrac{C_Et^{-\maf{f}[m/2]}}{
\prod_{i=1}^{m/2}(1-q^{-2i+1}tX^{-1})\prod_{i=1}^{m/2}(1-tXq^{-2i})}
+\dfrac{\varepsilon_{E/F}((-1)^{m/2}d)C_Et^{-\maf{f}[m/2]}}{
\prod_{i=1}^{m/2}(1-q^{-2i}tX^{-1})\prod_{i=1}^{m/2}(1-tXq^{-2i+1})}
\end{align*}

Then, we determine the $C_E$ if $E$ is ramified.
We fix $d=d_E\in\hana{O}_F^{\times}$ such that $F^{\times}=N_{E/F}(E^{\times})\sqcup dN_{E/F}(E^{\times})$. 
Then 
\begin{align*}
    &\lim_{X\to0}X^{\maf{f}[m/2]}P_m(d;X,Xq^{m-s})\\
    ={}&\lim_{X\to0}\dfrac{1}{\prod_{i=1}^m(1-q^{-i})}\int_{\dHer_m(F;\hana{O}_E)_d}X^{e(Z)}\wF_v(Z;X)|\det Z|^{s-m}dZ\\
    ={}&\dfrac{q^{[(\maf{f}+1)/2]ms}}{\prod_{i=1}^m(1-q^{-i})}\int_{\varpi^{[(\maf{f}+1)/2]}\dHer_m(F;\hana{O}_E)_d}|\det Z|^{s-m}dZ
\end{align*}
where $e(Z)=e_v(Z)$ is defined in Section $\ref{subsec2-3}$.
By \cite[Theorem 4.2(3)]{SH1} for $k=[(\maf{f}+1)/2]$, we have
\begin{align*}
    &\lim_{X\to0}X^{\maf{f}[m/2]}P_m(d;X,Xq^{m-s})\\
    ={}&\begin{cases}
        \dfrac{1}{2\phi_{(m-1)/2}(q^{-2})}\dfrac{q^{(m-1)\maf{f}s/2}}{\prod_{i=1}^{(m+1)/2}(1-q^{2i-2-s})}&\text{if $m$ is odd,}\\
        \dfrac{q^{\maf{f}ms/2}}{2\phi_{m/2}(q^{-2})}\left(\dfrac{q^{-\maf{f}m/2}}{\prod_{i=1}^{m/2}(1-q^{2i-1-s})}+\dfrac{\varepsilon_{E/F}((-1)^{m/2}d)q^{-\maf{f}m/2}}{\prod_{i=1}^{m/2}(1-q^{2i-2-s})}\right)&\text{if $m$ is even.}
    \end{cases}
\end{align*}
Therefore,
\begin{align*}
    &P_m(d;X,t)\\
    ={}&\begin{cases}
        \dfrac{q^{\maf{f}m(m-1)/2}}{2\prod_{j=1}^{(m-1)/2}(1-q^{-2j})}\dfrac{t^{-\maf{f}[m/2]}}{\prod_{i=1}^{(m+1)/2}(1-q^{-2i+1}X^{-1}t)(1-q^{-2i+1}Xt)}&\text{if $m$ is odd,}\\
        \dfrac{q^{\maf{f}m(m-1)/2}}{2\prod_{j=1}^{m/2}(1-q^{-2j})}\dfrac{t^{-\maf{f}[m/2]}}{\prod_{i=1}^{m/2}(1-q^{-2i+1}X^{-1}t)\prod_{i=1}^{m/2}(1-q^{-2i}Xt)}\\
        +\dfrac{q^{\maf{f}m(m-1)/2}}{2\prod_{j=1}^{m/2}(1-q^{-2j})}\dfrac{\varepsilon_{E/F}((-1)^{m/2}d)t^{-\maf{f}[m/2]}}{\prod_{i=1}^{m/2}(1-q^{-2i}X^{-1}t)\prod_{i=1}^{m/2}(1-q^{-2i+1}Xt)}&\text{if $m$ is even.}
    \end{cases}
\end{align*}

\end{proof}

%- Section -%
\section{The Period Formula for the Ikeda--Yamana Lift}\label{sec6}
\subsection{Proof of the Main Theorem}\label{subsec6-1}
Assume that $n$ is odd and put $m=\dfrac{n-1}{2}$. 
The Corollary to Theorem \ref{thm:RS1} and Theorem \ref{thm:KS} give the following local formulas.
\begin{thms}\label{thm:RS2}
\begin{itemize}
\item[(1)]If $E_v$ is inert over $F_v$, then
\begin{align*}
&\R_v(d;X,Y,t)=\hat{\R}_v(X,Y,t)\\={}&\prod_{i=1}^nL_v(i,\varepsilon_{E/F}^{i-1})^{-1}\prod_{i=1}^{n}(1-q_v^{-n+1}(-q_v)^{-(i-1)}t^2)\\
&\times\dfrac{1}{\prod_{i=1}^{n}(1+(-q_v)^{-i}XYt)(1+(-q_v)^{-i}XY^{-1}t)}\\
&\times\dfrac{1}{\prod_{i=1}^{n}(1+(-q_v)^{-i}X^{-1}Yt)(1+(-q_v)^{-i}X^{-1}Y^{-1}t)}
\end{align*}
for any $d\in F^{\times}$
\item[(2)]If $E_v$ is split over $F_v$, 
\begin{align*}
&\hat{\R}_v(X,Y,t)
=\R_v(d;X,Y,t)\\
={}&\prod_{i=1}^nL_v(i,\varepsilon_{E/F}^{i-1})^{-1}\prod_{i=1}^{n}(1-q_v^{-n+1}q_v^{-(i-1)}t^2)\\
&\times\dfrac{1}{\prod_{i=1}^{n}(1-q_v^{-i}XYt)(1-q_v^{-i}XY^{-1}t)}\\
&\times\dfrac{1}{\prod_{i=1}^{n}(1-q_v^{-i}X^{-1}Yt)(1-q_v^{-i}X^{-1}Y^{-1}t)}
\end{align*}
\item[(3)]If $E_v$ is ramified over $F_v$, 
\begin{align*}
&\hat{\R}_v(X,Y,t)\\
={}&q_v^{\maf{f}_vn(n-1)/2}t^{\maf{f}_v(n-1)/2}\prod_{i=1}^nL_v(i,\varepsilon_{E/F}^{i-1})^{-1}\prod_{i=1}^{m+1}(1-q_v^{-2(n+2i-2)}t^2)\\
&\times\dfrac{1}{\prod_{i=1}^{m+1}(1-q_v^{-n+2i-2}XYt)(1-q_v^{-n+2i-2}X^{-1}Y^{-1}t)}\\
&\times\dfrac{1}{\prod_{i=1}^{m+1}(1-q_v^{-n+2i-2}X^{-1}Yt)(1-q_v^{-n+2i-2}XY^{-1}t)}
\end{align*}
\end{itemize}
\end{thms}
\begin{proof}
The proof proceeds as in that of Theorem~\ref{thm:KS}.
Substituting the formulas in the corollary to Theorem~\ref{thm:RS1} and Theorem~\ref{thm:KS}, we express $\widehat{\R}_v(X,Y,t)$ as a rational function in $t$. 
After clearing the denominators, the numerator is a polynomial in $t$ of degree at most $2n$ in the inert and split cases and, after removing the factor $t^{\maf{f}_v(n-1)/2}$, of degree at most $n+1$ in the ramified case. 
Applying the functional equation for the Laurent functions $\wF_v(B;Y)$ in the variable $Y$ determines this numerator up to a constant.
The constant is determined by comparing the initial coefficients in $t$, which gives the stated formulas.
\end{proof}

\begin{thms}\label{thm:period1}
Let $f$ be a Hilbert Hecke cusp form of weight $\{\kappa_v\}$ and level one. The Rankin--Selberg integral of the Ikeda--Yamana lift $\mathcal{I}_n(f)$ has the following Euler product:
\begin{align*}
&\R(s,\I_n(f))\\
={}&2^{-n\kappa-(ns-n^2-n)d_F}|D_E|^{(n-1)(n-s)/2}|D_F|^{-(n-1)(n-s)-2n^2}\\
&\times\prod_{v\in\bmf{a}}\prod_{i=1}^n\Gamma_{\mb{C}}(s+\kappa_v-i)\Gamma_{\mb{C}}(i)^{-1}\prod_{i=1}^n\dfrac{L(s-n+i,f,\Ad,\varepsilon_{E/F}^{i-1})L(s-n+i,\varepsilon_{E/F}^{i-1})}{L(2s-n+i,\varepsilon_{E/F}^{i-1})L(n+i,\varepsilon_{E/F}^{n+i})}.
\end{align*}
Here $\kappa=\sum_{v}\kappa_v$.
\end{thms}
\begin{proof}
By Theorem \ref{thm:RS2}, we have
\begin{align*}
&\prod_{v\in\bmf{f}}(1-q_v^{-1})^{-1}\hat{\R}_v(q_v^{-\maf{s}_v},q_v^{-\overline{\maf{s}_v}},q_v^{n-s})\\
={}&(|D_E|/|D_F|^2)^{(n-1)(n-s)/2+n(n-1)/2}\prod_{i=2}^nL(i,\varepsilon_{E/F}^{i-1})^{-1}\\
&\times\prod_{i=1}^nL(2s-n+i,\varepsilon_{E/F}^{i-1})^{-1}\\
&\times\prod_{i=1}^n\prod_{v\in\bmf{f}}L_v(s-n+i+\maf{s}_v+\overline{\maf{s}_v},\varepsilon_{E/F}^{i-1})L_v(s-n+i+\maf{s}_v-\overline{\maf{s}_v},\varepsilon_{E/F}^{i-1})\\
&\times\prod_{i=1}^n\prod_{v\in\bmf{f}}L_v(s-n+i-\maf{s}_v+\overline{\maf{s}_v},\varepsilon_{E/F}^{i-1})L_v(s-n+i-\maf{s}_v-\overline{\maf{s}_v},\varepsilon_{E/F}^{i-1})
\end{align*}
Since the normalized Satake parameter \(q_v^{\maf{s}_v}\) has absolute value $1$ by \cite{BlD}, we have
\begin{align}
&\prod_{v\in\bmf{f}}(1-q_v^{-1})^{-1}\hat{\R}_v(q_v^{-\maf{s}_v},q_v^{-\overline{\maf{s}_v}},q_v^{n-s})\nonumber\\
={}&(|D_E|/|D_F|^2)^{(n-1)(n-s)/2+n(n-1)/2}\prod_{i=2}^nL(i,\varepsilon_{E/F}^{i-1})^{-1}\nonumber\\
&\times\prod_{i=1}^n\dfrac{L(s-n+i,f,\Ad,\varepsilon_{E/F}^{i-1})L(s-n+i,\varepsilon_{E/F}^{i-1})}{L(2s-n+i,\varepsilon_{E/F}^{i-1})}.\label{eq:locRS}
\end{align}

Combining \eqref{eq:constantofRS}, \eqref{eq:locfactorofRS} and \eqref{eq:locRS}, we obtain
\begin{align*}
&\R(s,\I_n(f))\\
={}&2^{-n\kappa-nd_Fs+n(n+1)d_F}|D_E|^{(n-1)(n-s)/2}|D_F|^{-(n-1)(n-s)-2n^2}\\
&\times\prod_{v\in\bmf{a}}\prod_{i=1}^n\Gamma_{\mb{C}}(s+\kappa_v-i)\Gamma_{\mb{C}}(i)^{-1}\prod_{i=1}^n\dfrac{L(s-n+i,f,\Ad,\varepsilon_{E/F}^{i-1})L(s-n+i,\varepsilon_{E/F}^{i-1})}{L(2s-n+i,\varepsilon_{E/F}^{i-1})L(n+i,\varepsilon_{E/F}^{n+i})}
\end{align*}
\end{proof}

\begin{proof}[Proof of Theorem \ref{MT}]
First, we recall the definition of the four complete L-functions. We put $\Gamma_{\mb{R}}(s,\varepsilon_{E/F})=\Gamma_{\mb{R}}(s+1)$ and for $\chi\in\{1,\varepsilon_{E/F}\}$, we set
\begin{align*}
\Lmd(s,\chi)={}&L(s,\chi)\Gamma_{\mb{R}}(s,\chi)^{d_F},\\
\wLmd(s,\chi)={}&L(s,\chi)\Gamma_{\mb{C}}(s)^{d_F},\\
\Lmd(s,f,\Ad,\chi)={}&L(s,f,\Ad,\chi)\prod_{v\in \bmf{a}}\left(\Gamma_{\mb{C}}(s+\kappa_v-1)\Gamma_{\mb{R}}(s,\chi)\right),\\
\wLmd(s,f,\Ad,\chi)={}&L(s,f,\Ad,\chi)\prod_{v\in \bmf{a}}\left(\Gamma_{\mb{C}}(s+\kappa_v-1)\Gamma_{\mb{C}}(s)\right)
\end{align*}

Then, by Theorem \ref{thm:period1}, we have
\begin{align*}
&\R(s,\I_n(f))\\
={}&2^{-n\kappa-n(s-n-1)d_F}|D_E|^{(n-1)(n-s)/2}|D_F|^{-(n-1)(n-s)-2n^2}\\
&\times\prod_{i=1}^n\dfrac{\Gamma_{\mb{R}}(2s-n+i,\varepsilon_{E/F}^{i-1})^{d_F}\Gamma_{\mb{R}}(n+i,\varepsilon_{E/F}^{n+i})^{d_F}}{\Gamma_{\mb{C}}(i)^{d_F}\Gamma_{\mb{R}}(s-n+i,\varepsilon_{E/F}^{i-1})^{d_F}\Gamma_{\mb{R}}(s-n+i,\varepsilon_{E/F}^{i-1})^{d_F}}\\
&\times\prod_{i=1}^n\dfrac{\Lmd(s-n+i,f,\Ad,\varepsilon_{E/F}^{i-1})\Lmd(s-n+i,\varepsilon_{E/F}^{i-1})}{\Lmd(2s-n+i,\varepsilon_{E/F}^{i-1})\Lmd(n+i,\varepsilon_{E/F}^{n+i})}\\
={}&2^{-n\kappa-n(s-n-1)d_F}|D_E|^{(n-1)(n-s)/2}|D_F|^{-(n-1)(n-s)-2n^2}\\
&\times\prod_{i=1}^n\dfrac{\Gamma_{\mb{C}}(2s-n+i)^{d_F}\Gamma_{\mb{R}}(n+i,\varepsilon_{E/F}^{n+i})^{d_F}}{\Gamma_{\mb{C}}(i)^{d_F}\Gamma_{\mb{C}}(s-n+i)^{d_F}\Gamma_{\mb{R}}(s-n+i,\varepsilon_{E/F}^{i-1})^{d_F}}\\
&\times\prod_{i=1}^n\dfrac{\wLmd(s-n+i,f,\Ad,\varepsilon_{E/F}^{i-1})\Lmd(s-n+i,\varepsilon_{E/F}^{i-1})}{\wLmd(2s-n+i,\varepsilon_{E/F}^{i-1})\Lmd(n+i,\varepsilon_{E/F}^{n+i})}.
\end{align*}
Thus, the residue of the Rankin--Selberg integral at $s=n$ is the following:
\begin{align*}
&\Res_{s=n}\R(s,\I_n(f))\\
={}&2^{-n\kappa+nd_F}|D_F|^{-2n^2}\prod_{i=1}^n\dfrac{\Gamma_{\mb{R}}(n+i,\varepsilon_{E/F}^{i-1})^{2d_F}}{\Gamma_{\mb{C}}(i)^{d_F}\Gamma_{\mb{R}}(i,\varepsilon_{E/F}^{i-1})^{2d_F}}\\
&\times\dfrac{\Lmd(1,f,\Ad)\Res_{\sigma=1}\Lmd(\sigma,1)}{\Lmd(n+1,1)^2}\prod_{i=2}^n\dfrac{\Lmd(i,f,\Ad,\varepsilon_{E/F}^{i-1})\Lmd(i,\varepsilon_{E/F}^{i-1})}{\Lmd(n+i,\varepsilon_{E/F}^{i-1})\Lmd(n+i,\varepsilon_{E/F}^{n+i})}\\
={}&2^{-n\kappa+nd_F}|D_F|^{-2n^2}\prod_{i=1}^n\dfrac{\Gamma_{\mb{C}}(n+i)^{d_F}\Gamma_{\mb{R}}(n+i,\varepsilon_{E/F}^{n+i})^{d_F}}{\Gamma_{\mb{C}}(i)^{2d_F}\Gamma_{\mb{R}}(i,\varepsilon_{E/F}^{i-1})^{d_F}}\\
&\times\dfrac{\wLmd(1,f,\Ad)\Res_{\sigma=1}\Lmd(\sigma,1)}{\wLmd(n+1,1)\Lmd(n+1,1)}\prod_{i=2}^n\dfrac{\wLmd(i,f,\Ad,\varepsilon_{E/F}^{i-1})\Lmd(i,\varepsilon_{E/F}^{i-1})}{\wLmd(n+i,\varepsilon_{E/F}^{i-1})\Lmd(n+i,\varepsilon_{E/F}^{n+i})}.
\end{align*}

By the corollary~\ref{cor:period1}, we have
\begin{align*}
    &\Res_{s=n}\R(s,\I_n(f))\\
    ={}&|D_F|^{-n/2}|D_E|^{-n(n-1)/4}\dfrac{\Res_{\sigma=1}\Lmd(\sigma,1)\prod_{i=2}^n\Lmd(i,\varepsilon_{E/F}^{i-1})}{2\prod_{i=1}^{n}\Lmd(n+i,\varepsilon_{E/F}^{i-1})}\perd{\I_n(f)}{\I_n(f)}.
\end{align*}
Therefore, we have
\begin{align*}
&\perd{\I_n(f)}{\I_n(f)}\\
={}&2^{-n\kappa+nd_F+1}|D_E|^{n(n-1)/4}|D_F|^{-n(4n-1)/2}\\
&\times\prod_{i=1}^n\dfrac{\Gamma_{\mb{R}}(n+i,\varepsilon_{E/F}^{i-1})^{2d_F}}{\Gamma_{\mb{C}}(i)^{d_F}\Gamma_{\mb{R}}(i,\varepsilon_{E/F}^{i-1})^{2d_F}}\prod_{i=1}^n\dfrac{\Lmd(i,f,\Ad,\varepsilon_{E/F}^{i-1})}{\Lmd(n+i,\varepsilon_{E/F}^{i-1})}\\
={}&2^{-n\kappa+nd_F+1}|D_E|^{n(n-1)/4}|D_F|^{-n(4n-1)/2}\\
&\times\prod_{i=1}^n\dfrac{\Gamma_{\mb{C}}(n+i)^{d_F}\Gamma_{\mb{R}}(n+i,\varepsilon_{E/F}^{n+i})^{d_F}}{\Gamma_{\mb{C}}(i)^{2d_F}\Gamma_{\mb{R}}(i,\varepsilon_{E/F}^{i-1})^{d_F}}\prod_{i=1}^n\dfrac{\wLmd(i,f,\Ad,\varepsilon_{E/F}^{i-1})}{\wLmd(n+i,\varepsilon_{E/F}^{i-1})}.
\end{align*}

Since $\Gamma_{\mb{R}}(s,\varepsilon_{E/F})=\Gamma_{\mb{R}}(s+1)$ and $\Gamma_{\mb{R}}(s,1)=\Gamma_{\mb{R}}(s)$, we have
\begin{align*}
\prod_{i=1}^n\dfrac{\Gamma_{\mb{R}}(n+i,\varepsilon_{E/F}^{i-1})}{\Gamma_{\mb{C}}(i)\Gamma_{\mb{R}}(i,\varepsilon_{E/F}^{i-1})}={}&\prod_{i=1}^n\dfrac{\Gamma_{\mb{R}}(n+i,\varepsilon_{E/F}^{n+i})\Gamma_{\mb{R}}(i,\varepsilon_{E/F}^i)}{\Gamma_{\mb{C}}(i)^2}\\
={}&\dfrac{\prod_{i=1}^{n}\Gamma_{\mb{R}}(2i-1,\varepsilon_{E/F}^{2i-1})\Gamma_{\mb{R}}(2i,\varepsilon_{E/F}^{2i})}{\prod_{i=1}^n\Gamma_{\mb{C}}(i)^2}\\
={}&\dfrac{\prod_{i=1}^{n}\Gamma_{\mb{R}}(2i)^2}{\prod_{i=1}^n\Gamma_{\mb{C}}(i)^2}\\
={}&\left(\dfrac{\prod_{i=1}^{n}\pi^{-i}\Gamma(i)}{\prod_{i=1}^n2(2\pi)^{-i}\Gamma(i)}\right)^2\\
={}&2^{n(n-1)}.
\end{align*}
Then, we have
\begin{align*}
&\perd{\I_n(f)}{\I_n(f)}\\
={}&2^{-n\kappa+n(2n-1)d_F+1}|D_E|^{n(n-1)/4}|D_F|^{-n(4n-1)/2}\prod_{i=1}^n\Gamma_{\mb{C}}(i)^{d_F}\prod_{i=1}^n\dfrac{\Lmd(i,f,\Ad,\varepsilon_{E/F}^{i-1})}{\Lmd(n+i,\varepsilon_{E/F}^{i-1})}\\
={}&2^{-n\kappa+n^2d_F+1}|D_E|^{n(n-1)/4}|D_F|^{-n(4n-1)/2}\prod_{i=1}^n\dfrac{\Gamma_{\mb{C}}(n+i)^{d_F}}{\Gamma_{\mb{C}}(i)^{d_F}}\prod_{i=1}^n\dfrac{\wLmd(i,f,\Ad,\varepsilon_{E/F}^{i-1})}{\wLmd(n+i,\varepsilon_{E/F}^{i-1})}.
\end{align*}

\end{proof}

\subsection{Comparison with Previous Work}\label{subsec6-2}
In this subsection, we compare the formula obtained by specializing our main theorem to $F=\mb{Q}$ with the previously known formula for the period of the classical Ikeda lift.
To compare the two formulas, we need to account for two differences: the normalization of the measure 
used to define the period and the scalar relating the Ikeda--Yamana lift to the classical Ikeda lift.

First, we compare the definition of our period with Katsurada's period.
Our period is defined by the integral on $\Z(\A)\wGe_n(F)\lsla\wGe_n(\A)$, but the period in \cite{K3} is essentially defined as an integral on $\Ge_n(F)\lsla\Ge_n(\A)$. 
We let $\mb{G}^1$ be the kernel of the norm map $\Res_{E/F}\mb{G}_m\to\mb{G}_m$ and $\Z\cap\Ge_n=\Z_{\Ge_n}$.
Now we have the quotient fibration
\begin{align*}
    \Z_{\Ge_n}(\A)\Ge_n(F)\lsla\Ge_n(\A)&\to\Z(\A)\wGe_n(F)\lsla\wGe_n(\A)\xrightarrow{\nu} F^{\times}N_{E/F}(\A_E^{\times})\lsla\A_F^{\times},\\
    \Z_{\Ge_n}(F)\lsla\Z_{\Ge_n}(\A)\to&\Ge_n(F)\lsla\Ge_n(\A)\to\Z_{\Ge_n}(\A)\Ge_n(F)\lsla\Ge_n(\A).
\end{align*} 
By the facts $\tau(\mb{G}^1)=2$ and $\Z_{\Ge_n}\cong\mb{G}^1$, we have
\begin{align*}
    \int_{\Ge_n(F)\lsla\Ge_n(\A)}\phi(g)d\mu_{\Ge_n}(g)={}&\dfrac{1}{2}\int_{\Z_{\Ge_n}(F)\lsla\Z_{\Ge_n}(\A)}d\mu_{\Z_{\Ge_n}}(z)\int_{\Z(\A)\wGe_n(F)\lsla\wGe_n(\A)}\phi(g)d\mu_{\Z\lsla\wGe_n}(g)\\
    ={}&\int_{\Z(\A)\wGe_n(F)\lsla\wGe_n(\A)}\phi(g)d\mu_{\Z\lsla\wGe_n}(g)
\end{align*}
for any smooth integrable function $\phi$ on $\Z(\A)\wGe_n(F)\lsla\wGe_n(\A)$ where $\mu_{\Ge_n}$ is the Tamagawa measure on $\Ge_n(F)\lsla\Ge_n(\A)$.

Let $\iota_{\bmf{f}}'$ be the diagonal embedding of $\Ge_n(F)\to\prod_{v\in\bmf{f}}\Ge_n(F_v)$ and $\iota_{\bmf{a}}'$ be the diagonal embedding of $\Ge_n(F)\to\prod_{v\in\bmf{a}}\Ge_n(F_v)$.
Let $\{\gamma'_1,\cdots,\gamma'_{h_E/h_{F}}\}$ be the representative set of $\Ge_n(F)\lsla\Ge_n(\A)/\K_{n,\bmf{f}}\Ge_n(F\otimes\mb{R})$ such that $\gamma'_i\in\Ge_n(\A_{\bmf{f}})$ and $\gamma'_1=1_{2n}$. We put $h=h_{E}/h_{F}$ and put
\begin{equation*}
    \Gamma_{\K,i}^{(n)}=\iota_{\bmf{a}}'((\iota_{\bmf{f}}')^{-1}(\gamma'_i\K_{n,\bmf{f}}(\gamma'_i)^{-1}))\text{ and }\Gamma^{(n)}=\iota_{\bmf{a}}'\left((\iota_{\bmf{f}}')^{-1}\left(\prod_{v\in\bmf{f}}\left(\GL_{2n}(\hana{O}_{E_v})\cap\Ge_n(F_v)\right)\right)\right).
\end{equation*}

Let $\hana{Y}_{i}$ be the fundamental domain of $\Gamma_{\K,i}^{(n)}\lsla\maf{H}_n^{\bmf{a}}$ and $\hana{X}$ be the inverse image of $\bigsqcup_{i=1}^{h}\hana{Y}_i$ under the map $\prod_{v\in\bmf{a}}\Ge_n(F_v)\to\prod_{v\in\bmf{a}}\left(\Ge_n(F_v)/\K_{n,v}\right)\cong\maf{H}_n^{\bmf{a}}$. 
Then we have
\begin{equation*}
    \Ge_n(F)\lsla\Ge_n(\A)/\left(\K_{n,\bmf{f}}\times\prod_{v\in\bmf{a}}\K_{n,v}\right)\cong\bigsqcup_{i=1}^{h}\Gamma_{\K,i}^{(n)}\lsla\maf{H}_n^{\bmf{a}}
\end{equation*}

As in \ref{ap:Tamagawa}, the Tamagawa measure $\mu_{\Ge_n}$ induces the invariant measure on the hermitian upper half-plane and we can compute the volume 
\begin{equation*}
\sum_{i=1}^h\int_{\hana{Y}_i}\prod_{v\in\bmf{a}}d\mu_{\maf{H}_n}.
\end{equation*}

As in \ref{ap:Tamagawa}, we compute the volume $\vol(\prod_{v\in\bmf{f}}\K_v\times\hana{X};\mu_{\Ge_n,\A})$.

We let $\chi_{\Ge_n}$ be the character corresponding to the Galois representation on $\Hom(\Ge_n,\mb{G}_m)$ and $L(s,\chi_{\Ge_n})$ be the Artin L-function of $\chi_{\Ge_n}$. 
Then $L(s,\chi_{\Ge_n})=L(s,\varepsilon_{E/F})$ and we have
\begin{equation*}
    \mu_{\Ge_n,\A}=\dfrac{|D_F|^{-2n^2}}{L(1,\varepsilon_{E/F})}\prod_{v\in\bmf{f}}\dfrac{1}{1-\varepsilon_vq_v^{-1}}\mu_{\Ge_n,v}\times\prod_{v\in\bmf{a}}\mu_{\Ge_n,v}
\end{equation*}
where $\mu_{\Ge_n,v}$ is the Haar measure induced by the highest differential form on $\Ge_n$ and $\varepsilon_v$ is defined in \ref{ap:Tamagawa}.

Then we obtain the volume $\vol\left(\prod_{v\in\bmf{f}}\K_v\times\hana{X};\mu_{\Ge_n,\A}\right)$ as in \ref{ap:Tamagawa}:
\begin{align*}
    &\vol\left(\prod_{v\in\bmf{f}}\K_v\times\hana{X};\mu_{\Ge_n,\A}\right)\\
    ={}&\dfrac{|D_F|^{-(4n^2)/2}}{L(1,\chi_{\Ge_n})}\prod_{v\in\bmf{f}}L_v(1,\chi_{\Ge_n})\vol(\K_v,\mu_{\Ge_n,v})\times\vol(\hana{X};\prod_{v\in\bmf{a}}\mu_{\Ge_n,v})\\
    ={}&\dfrac{|D_F|^{-2n^2}}{L(1,\varepsilon_{E/F})}\prod_{v\in\bmf{f}}\dfrac{\prod_{i=1}^{2n}(1-\varepsilon_v(i)q_v^{-i})\times\delta_v(\gamma)^{-n(2n-1)}}{(1-\varepsilon_vq_v^{-1})}\\
    &\times\left(\sum_{i=1}^h\int_{\hana{Y}_i}\prod_{v\in\bmf{a}}2^{-n}\alpha_{E_v}^{-n(n-1)/2}(\gamma)d\mu_{\maf{H}_n}\right)\times\prod_{v\in\bmf{a}}\int_{\K_{n,v}}d\mu_{\K_{n,v}}
\end{align*}
where $\gamma$ is the integral element of $E$ such that $\overline{\gamma}=-\gamma$, $\delta_v(\gamma)=[\hana{O}_{E_v}:\hana{O}_v[\gamma]]^{-1}=\alpha_{E_v}(2\gamma)^{1/2}\alpha_{E_v}(\maf{d}_{E_v/F_v})^{-1/2}$ and $\varepsilon_v(i)$ for $1\leqq i\leqq 2n$ are defined in \ref{ap:Tamagawa}.

Then we have
\begin{align*}
    &\vol\left(\prod_{v\in\bmf{f}}\K_v\times\hana{X};\mu_{\Ge_n,\A}\right)\\
    ={}&\dfrac{|D_F|^{-2n^2}}{L(1,\varepsilon_{E/F})}\prod_{v\in\bmf{f}}\dfrac{\prod_{i=1}^{2n}(1-\varepsilon_v(i)q_v^{-i})\times\left(\frac{\alpha_{E_v}(2\gamma)}{\alpha_{E_v}(\maf{d}_{E_v/F_v})}\right)^{-n(2n-1)/2}}{(1-\varepsilon_vq_v^{-1})}\\
    &\times\prod_{v\in\bmf{a}}\left(2^{-n}\alpha_{E_v}^{-n(n-1)/2}(\gamma)\times\alpha_{E_v}(\gamma_{v,0})^{-n^2/2}2^{3n}\prod_{i=1}^n\dfrac{1}{\Gamma_{\mb{C}}(i)^2}\right)\times\left(\sum_{i=1}^h\int_{\hana{Y}_i}\prod_{v\in\bmf{a}}d\mu_{\maf{H}_n}\right)\\
    ={}&\dfrac{|D_F|^{-2n^2}}{L(1,\varepsilon_{E/F})}\prod_{v\in\bmf{f}}\alpha_{E_v}^{n(2n-1)/2}(\maf{d}_{E_v/F_v})\times2^{d_F(n(2n-1)-n+3n)}\\
    &\times\prod_{v\in\bmf{f}}\prod_{i=2}^{2n}(1-\varepsilon_v(i)q_v^{-i})\prod_{v\in\bmf{a}}\prod_{i=1}^n\Gamma_{\mb{C}}(i)^{-2}\left(\sum_{i=1}^h\int_{\hana{Y}_i}\prod_{v\in\bmf{a}}d\mu_{\maf{H}_n}\right)\\
    ={}&|D_E|^{-n(2n-1)/2}|D_F|^{-n}2^{(2n^2+n)d_F}\prod_{i=1}^{2n}L(i,\varepsilon_{E/F}^i)^{-1}\prod_{i=1}^n\Gamma_{\mb{C}}(i)^{-2d_F}\times\left(\sum_{i=1}^h\int_{\hana{Y}_i}\prod_{v\in\bmf{a}}d\mu_{\maf{H}_n}\right)
\end{align*}
where $\gamma\in E^{\times}$ such that $\rho_{E/F}(\gamma)=-\gamma$ and $\gamma_{v,0}$ is defined in \ref{ap:Tamagawa}.

On the other hand, $\prod_{v\in\bmf{f}}\K_v\times\hana{X}$ is the domain of $\Ge_n(F)\lsla\Ge_n(\A)$ with multiplicity $w_E$, $\tau(\Ge_n)=2$ and 
\begin{equation*}
    \vol\left(\prod_{v\in\bmf{f}}\K_v\times\hana{X};\mu_{\Ge_n,\A}\right)=\tau(\Ge_n)\times\#\Z_{\Ge_n}(F)=2w_E.
\end{equation*}
where $w_E$ is the number of roots of unity in $E$.
Then, we obtain
\begin{align*}
    &\sum_{i=1}^h\int_{\hana{Y}_i}d\mu_{\maf{H}_n}
    ={}|D_E|^{(2n^2-n)/2}|D_F|^{n}2^{-(2n^2+n)d_F+1}w_E\prod_{i=1}^{2n}L(i,\varepsilon_{E/F}^i)\times\prod_{i=1}^n\Gamma_{\mb{C}}(i)^{2d_F}
\end{align*}
and
\begin{align*}
    &\int_{\hana{Y}_i}d\mu_{\maf{H}_n}
    ={}|D_E|^{(2n^2-n)/2}|D_F|^{n}2^{-(2n^2+n)d_F+1}w_Eh^{-1}\prod_{i=1}^{2n}L(i,\varepsilon_{E/F}^i)\times\prod_{i=1}^n\Gamma_{\mb{C}}(i)^{2d_F}
\end{align*}
for any $i$.

When $F=\mb{Q}$, $h=h_E$ and we have 
\begin{equation*}
    \sum_{i=1}^h\int_{\hana{Y}_i}d\mu_{\maf{H}_n}=|D_E|^{n(2n-1)/2}2^{-(2n^2+n)+1}w_E\prod_{i=1}^{2n}L(i,\varepsilon_{E/F}^i)\times\prod_{i=1}^n\Gamma_{\mb{C}}(i)^{2}
\end{equation*}
and 
\begin{align*}
    &\int_{\hana{Y}_i}d\mu_{\maf{H}_n}
    ={}|D_E|^{n(2n-1)/2}2^{-(2n^2+n)+1}w_Eh_E^{-1}\prod_{i=1}^{2n}L(i,\varepsilon_{E/F}^i)\times\prod_{i=1}^n\Gamma_{\mb{C}}(i)^{2}
\end{align*}

In \cite{K3}, Katsurada normalizes the measure $\mu_{\Gamma,\mr{Kat}}$ on $\Gamma\lsla\maf{H}_n$ for congruence subgroup $\Gamma$ by 
\begin{equation*}
    \int_{\Gamma\lsla\maf{H}_n}d\mu_{\Gamma,\mr{Kat}}=\int_{\Gamma^{(n)}\lsla\maf{H}_n}d\mu_{\maf{H}_n}.
\end{equation*}

He defines the period for the tuple of the Hilbert-Hermitian cusp forms $F=(F_1,\dots,F_h)$, $G=(G_1,\cdots,G_h)$ of weight $k$ which correspond to each automorphic form $\phi_F,\phi_G$ on $\Ge_n(\A)$ by
\begin{align*}
    &\perd{F}{G}_{\mr{Kat}}=\dfrac{1}{h_E}\sum_{i=1}^{h_E}\int_{\Gamma_i\lsla\maf{H}_n}F_i(Z)\overline{G_i(Z)}|\det Y|^{k}d\mu_{\Gamma,\mr{Kat}}(Z).
\end{align*}
It is well known that $\int_{\Gamma_i\lsla\maf{H}_n}d\mu_{\maf{H}_n}=\int_{\Gamma_1\lsla\maf{H}_n}d\mu_{\maf{H}_n}$ for any $i$. 
We set $v(\Gamma)=\int_{\Gamma\lsla\maf{H}_n}d\mu_{\maf{H}_n}$ for a congruence subgroup $\Gamma$. 
Then we have
\begin{align*}
    \perd{F}{G}_{\mr{Kat}}={}&\dfrac{v(\Gamma^{(n)})}{h_Ev(\Gamma_1)}\sum_{i=1}^{h_E}\int_{\Gamma_i\lsla\maf{H}_n}F_i(Z)\overline{G_i(Z)}|\det Y|^{k}d\mu_{\maf{H}_n}(Z).
\end{align*}

Therefore, if $\phi_F(g)\overline{\phi_G(g)}$ is trivial on $\Z(\A)\cap\Ge_n(\A)$, by classical-adelic correspondence we have
\begin{align*}
    \perd{F}{G}_{\mr{Kat}}={}&\dfrac{v(\Gamma^{(n)})}{2}\dfrac{\tau(\Ge_n)}{h_Ev(\Gamma_1)}\sum_{i=1}^{h_E}\int_{\Gamma_i\lsla\maf{H}_n}F_i(Z)\overline{G_i(Z)}|\det Y|^{k}d\mu_{\maf{H}_n}(Z)\\
    ={}&v(\Gamma^{(n)})2^{-1}\int_{\Ge_n(\mb{Q})\lsla\Ge_n(\A_{\mb{Q}})}\phi_F(g)\overline{\phi_G(g)}d\mu_{\Ge_n}(g)\\
    ={}&v(\Gamma^{(n)})2^{-1}\perd{\phi_F}{\phi_G}.
\end{align*}

A computation similar to $\int_{\hana{Y}_i}d\mu_{\maf{H}_n}$ gives
\begin{equation*}
    v(\Gamma^{(n)})=|D_E|^{(2n^2-n)/2}2^{-(2n^2+n)+1}w_Eh_E^{-1}\prod_{i=1}^{2n}L(i,\varepsilon_{E/F}^i)\times\prod_{i=1}^n\Gamma_{\mb{C}}(i)^{2}.
\end{equation*}

Therefore,
\begin{align*}
    \perd{F}{G}_{\mr{Kat}}={}&|D_E|^{(2n^2-n)/2}2^{-2n^2-n}\dfrac{L(1,\varepsilon_{E/F})w_E}{h_E}\prod_{i=2}^{2n}L(i,\varepsilon_{E/F}^i)\times\prod_{i=1}^n\Gamma_{\mb{C}}(i)^{2}\perd{\phi_F}{\phi_G}.
\end{align*}

Next, we next determine the scalar relating the Ikeda--Yamana lift when $F=\mb{Q}$ to the classical Ikeda lift.

We denote by $F_f$ the Ikeda lift defined in \cite{Ike2}.
In \cite[Theorem 5.2]{Ike2}, the Fourier coefficient at $B$ for $B\in\Her_n(\mb{Q})^+$ of the classical Ikeda lift associated with the Hecke eigenform $f\in S_{2k}(SL_2(\mb{Z}))$ has the form
\begin{equation*}
    \left((-D_E)^{(n-1)/2}\det B\right)^{k-1/2}\prod_{p|(-D_E)^{(n-1)/2}\det B}\wF_p(B;p^{-\maf{s}_p}).
\end{equation*}

In \cite[Corollary 11.3.]{Y}, Yamana shows that the tuple of the functions $(H_1,\cdots,H_{h_{E}})$ defined by
\begin{equation*}
    \I_n(f)\prod_{v\in\bmf{f}}\alpha_{E_v}(\maf{d}_{E_v/F_v})^{-2n\maf{s}_v}=(H_1,\cdots,H_{h_{E}})
\end{equation*}
defines a tuple of Hilbert-Hermitian modular forms of weight $(\kappa_v+n-1)_{v\in\bmf{a}}$ with respect to the character $\prod_{v\in\bmf{a}}\epsilon_v^{(\kappa_v+n-1)/2}$.
Fourier coefficient at $B$ of the function $H_i$ is of the form
\begin{equation*}
    |\det B|^{(\kappa-1)/2}\prod_{v\in S(f)}\alpha_{F_v}^{n/2}(\lambda_{i,v}^{-1}N_{E/F}(\det A_{i,v}))\prod_{v\notin S(f)}\wF_v(\lambda_{i,v}^{-1}B[A_{i,v}];q_v^{-\maf{s}_v})
\end{equation*}
where $H_i$ is the function on $\wGe_n(F)\left(\m(A_{i,v})\bm{d}_n(\lambda_{i,v})\right)_{v\in\bmf{f}}\wK_{\bmf{f}}\wGe_n^+(F\otimes\mb{R})$, $\kappa=\sum_{v\in\bmf{a}}\kappa_v$ and $S(f)$ is defined in Section \ref{subsec3-1}.

Then by setting $\kappa=2k$ and comparing the Fourier coefficients, we have $|F_{f}|=|D_E|^{(n-1)/2\cdot (\kappa-1)/2}|\I_n(f)|$.
Then we have 
\begin{align*}
    &\perd{F_f}{F_f}_{\mr{Kat}}\\
    ={}&|D_E|^{(n-1)(\kappa-1)/2+(2n^2-n)/2}2^{-2n^2-n}    \dfrac{L(1,\varepsilon_{E/F})w_E}{h_E}\prod_{i=2}^{2n}L(i,\varepsilon_{E/F}^i)\times\prod_{i=1}^n\Gamma_{\mb{C}}(i)^2\perd{\I_n(f)}{\I_n(f)}\\
    ={}&|D_E|^{(n-1)(\kappa-1)/2+(2n^2-n)/2}2^{-2n^2-n}\dfrac{\wLmd(1,\varepsilon_{E/F})w_E}{h_E}\prod_{i=2}^{2n}\wLmd(i,\varepsilon_{E/F}^i)\times\prod_{i=1}^n\dfrac{\Gamma_{\mb{C}}(i)}{\Gamma_{\mb{C}}(n+i)}\perd{\I_n(f)}{\I_n(f)}.
\end{align*}

Then we have
\begin{align*}
    &\perd{F_f}{F_f}_{\mr{Kat}}\\
    ={}&|D_E|^{(n-1)(\kappa-1)/2+(2n^2-n)/2}2^{-2n^2-n}\dfrac{\wLmd(1,\varepsilon_{E/F})w_E}{h_E}\prod_{i=2}^{2n}\wLmd(i,\varepsilon_{E/F}^i)\times\prod_{i=1}^n\dfrac{\Gamma_{\mb{C}}(i)}{\Gamma_{\mb{C}}(n+i)}\\
    &\times2^{-n\kappa+n^2+1}|D_E|^{n(n-1)/4}\prod_{i=1}^n\dfrac{\Gamma_{\mb{C}}(n+i)}{\Gamma_{\mb{C}}(i)}\prod_{i=1}^n\dfrac{\wLmd(i,f,\Ad,\varepsilon_{E/F}^{i-1})}{\wLmd(n+i,\varepsilon_{E/F}^{i-1})}\\
    ={}&2^{-n\kappa-n^2-n+1}|D_E|^{(n-1)\cdot\kappa/2+5(n^2-n)/4}    
    \dfrac{\wLmd(1,\varepsilon_{E/F})|D_E|^{1/2}w_E}{h_E}\prod_{i=1}^n\wLmd(i,f,\Ad,\varepsilon_{E/F}^{i-1})\prod_{i=2}^{n}\wLmd(i,\varepsilon_{E/F}^i).
\end{align*}

Now $\wLmd(1,\varepsilon_{E/F})=\dfrac{2h_E}{w_E|D_E|^{1/2}}$.

Therefore,
\begin{align*}
    &\perd{F_f}{F_f}_{\mr{Kat}}\\
    ={}&2^{-n\kappa-n^2-n+2}|D_E|^{(n-1)\cdot\kappa/2+5(n^2-n)/4}\prod_{i=1}^n\wLmd(i,f,\Ad,\varepsilon_{E/F}^{i-1})\prod_{i=2}^{n}\wLmd(i,\varepsilon_{E/F}^i).
\end{align*}

After substituting $n=2m+1$, this agrees exactly \cite[Theorem 2.2(2)]{K3}.

\appendix\label{ap}
\section{Tamagawa measures and Iwasawa decomposition}\label{ap:Tamagawa}

In this appendix, we determine the constant relating the Tamagawa measure on
$\wGe_n$ to the measure arising from the Iwasawa decomposition
\begin{equation*}
\wGe_n^+(\A)=\wPe_n^+(\A)\wK_n.
\end{equation*}

The purpose is to make explicit the constant $c_0$ appearing in Section~\ref{sec4}.

\subsection{Generalities on Tamagawa measures}

We use Ono's normalization of Tamagawa measures \cite{Ono}. Let $G$ be a connected reductive group over $F$, let $\omega_G$ be a non-zero left invariant differential form of highest degree, and let $\mu_{G,v}$ be the induced Haar measure on $G(F_v)$. The Galois module $X^*(G)\otimes_{\mb Z}\mb Q$ has trace character $\chi_G$; write $r_G$ for the multiplicity of its trivial constituent. Ono's convergence factors are $L_v(1,\chi_G)^{-1}$, and the corresponding adelic measure is
\begin{equation*}
\mu_{G,\A}=\frac{|D_F|^{-\frac{1}{2}\dim G}}
{\lim_{s\to1}(s-1)^{r_G}L(s,\chi_G)}
\prod_{v\in\bmf{a}}\mu_{G,v}
\prod_{v\in\bmf{f}}
L_v(1,\chi_G)\mu_{G,v}.
\end{equation*}
The quotient measure induced on $G(F)\backslash G(\A)$ is the Tamagawa measure.

We now apply this construction to the Tamagawa measure on $\wGe_n(F)\lsla\wGe_n(\A)$ and the Tamagawa measure on $\wPe_n(F)\lsla\wPe_n(\A)$.

When $G=\wGe_n$, $X^*(\wGe_n)$ is generated by $\det$ and $\nu_n$, and we have $X^*(\wGe_n)\otimes_{\mb{Z}}\mb{Q}\cong\mb{Q}[\Gal(E/F)]$. Then we have $L_v(s,\chi_{\wGe_n})=L_v(s,1)L_v(s,\varepsilon_{E/F})$. 
Thus, we define the measure on $\wGe_n(\A)$ by
\begin{equation*}
    \mu_{\wGe_n,\A}=\frac{|D_F|^{-(4n^2+1)/2}}
{\lim_{s\to1}(s-1)L(s,\chi_{\wGe_n})}
\prod_{v\in\bmf{a}}\mu_{\wGe_n,v}
\prod_{v\in\bmf{f}}
L_v(1,\chi_{\wGe_n})\mu_{\wGe_n,v}.
\end{equation*}

We define the Tamagawa measure $\mu_{\wGe_n}$ as the quotient measure induced by $\mu_{\wGe_n,\A}$ and the canonical discrete measure on $\wGe_n(F)$.

To define the Tamagawa measure on $\wPe_n$, we first consider its Levi part $\wM_n$, since $\wPe_n$ is not a reductive algebraic group and $\wPe_n=\wM_n\ltimes U_{\Pe_n}$.
When $G=\wM_n$, $X^*(\wM_n)\cong\mb{Z}[\Gal(E/F)]\oplus\mb{Z}$ as $\Gal(E/F)$-modules and we have $L_v(s,\chi_{\wM_n})=L_v(s,1)^2L_v(s,\varepsilon_{E/F})$.

Then we put
\begin{equation*}
    \mu_{\wM_n,\A}=\frac{|D_F|^{-(2n^2+1)/2}}
{\lim_{s\to1}(s-1)^2L(s,\chi_{\wM_n})}
\prod_{v\in\bmf{a}}\mu_{\wM_n,v}
\prod_{v\in\bmf{f}}
L_v(1,\chi_{\wM_n})\mu_{\wM_n,v}.
\end{equation*}

We choose the differential form induced by an $F$-basis of $\mr{Lie}\,(U_{\Pe_n})\simeq \Her_n$ and let $\omega_{U}$ be a non-zero left invariant differential form of highest degree on $U_{\Pe_n}$ and this also induces a Haar measure on $U_{\Pe_n}(F_v)$.

Since $U_{\Pe_n}$ is unipotent, it has no
non-trivial rational characters. Hence, no convergence factor occurs for $U_{\Pe_n}$. 

Thus we define the Haar measure $\mu_{U,\A}$ on $U_{\Pe_n}(\A)$ by
\begin{equation*}
    |D_F|^{-n^2/2}
\prod_{v\in\bmf{a}}\mu_{U,v}
\prod_{v\in\bmf{f}}\mu_{U,v}
\end{equation*}
and define the Tamagawa measure $\mu_U$ by the induced quotient measure on $U_{\Pe_n}(F)\lsla U_{\Pe_n}(\A)$ by $\mu_{U,\A}$. Then $\vol(U_{\Pe_n}(F)\lsla U_{\Pe_n}(\A),\mu_U)=1$.

We define the measure $\mu_{\wPe_n,\A}$ on $\wPe_n(\A)$ to be the push-forward of $\mu_{\wM_n,\A}\times\mu_{U,\A}$ under the bijection $\wM_n(\A)\times U_{\Pe_n}(\A)\cong \wPe_n(\A),\,(m,u)\mapsto mu$. 
Then this is a left Haar measure on $\wPe_n(\A)$.
We denote the induced quotient measure on $\wPe_n(F)\backslash \wPe_n(\A)$ by $\mu_{\Pe_n}$ and call $\mu_{\Pe_n}$ the Tamagawa measure of $\wPe_n$.

\subsection{Comparison with the Iwasawa decomposition}
We consider the decomposition of the measure with the Iwasawa decomposition.

We put $\wK_{\wPe_n,v}=\wK_{n,v}\cap\wPe_n(F_v)$ for $v\in\bmf{f}$, $\K_{\wPe_n,v}=\K_{n,v}\cap\wPe_n(F_v)$ and $\wK_{\wPe_n}=\wK_n\cap\wPe_n(\A)$.

We also define a measure $\mu_{\wPe_n,v}$ on $\wPe_n(F_v)$ to be the push-forward of $\mu_{\wM_n,v}\times\mu_{U,v}$ under the bijection of $\wM_n(F_v)\times U_{\wPe_n}(F_v)\cong\wPe_n(F_v),\,(m,u)\mapsto mu$ for $v\in\bmf{v}$. 
This is also a left Haar measure on $\wPe_n(F_v)$.
For $v\in\bmf f$, the Iwasawa decomposition gives $\wGe_n(F_v)=\wPe_n(F_v)\wK_{n,v}$.
We define the positive constant $c_{0,v}$ by the identity
\begin{equation*}
\int_{\wGe_n(F_v)}
\phi_v(g)\,d\mu_{\wGe_n,v}(g)
=
c_{0,v}
\int_{\wPe_n(F_v)}
\phi_v(p)\,d\mu_{\wPe_n,v}(p)
\end{equation*}
for any right $\wK_{n,v}$-invariant compactly supported function $\phi_v$.

For $v\in\bmf{a}$, the Iwasawa decomposition of $\Ge_n(F_v)$ gives the equation $\wGe_n(F_v)=\wPe_n(F_v)\K_{n,v}$ as we explaned in Section \ref{sec4}. 
We define a positive constant $c_{0,v}$ such that  we have
\begin{equation*}
\int_{\wGe_n(F_v)}\phi_{v}(g_v)d\mu_{\wGe_n,v}(g_v)=c_{0,v}\int_{\wPe_n(F_v)}\phi_v(p_v)d\mu_{\wPe_n,v}(p_v)
\end{equation*}
for any right $\K_{n,v}$-invariant compactly supported function $\phi_v$.

Under the Iwasawa decomposition $\bm{I}:\wPe_n(\A)\times \wK_{n}\to\wGe_n(\A)$, we also have a constant $c_0$ such that for any right $\wK_n$-invariant integrable function $\phi$
\begin{equation*}
\int_{\wGe_n(\A)}\phi(g)d\mu_{\wGe_n,\A}(g)=c_{0}\int_{\wPe_n(\A)}\phi(p)d\mu_{\wPe_n,\A}(p).
\end{equation*}

Taking the restricted product of the local identities and inserting the global normalizing factors, we have
\begin{equation*}
    c_0=\dfrac{c_{\wGe_n}}{c_{\wPe_n}}\prod_{v\in\bmf{f}}\dfrac{L_v(1,\chi_{\wGe_n})}{L_v(1,\chi_{\wM_n})}c_{0,v}\times\prod_{v\in\bmf{a}}c_{0,v}
\end{equation*}
where 
\begin{equation*}
    c_{\wGe_n}=\frac{|D_F|^{-(4n^2+1)/2}}
{\lim_{s\to1}(s-1)L(s,\chi_{\wGe_n})}\text{ and }c_{\wPe_n}=\frac{|D_F|^{-(3n^2+1)/2}}
{\lim_{s\to1}(s-1)^2L(s,\chi_{\wM_n})}.
\end{equation*}

The product 
\begin{equation*}
\prod_{v\in\bmf{f}}\dfrac{L_v(1,\chi_{\wGe_n})}{L_v(1,\chi_{\wM_n})}c_{0,v}
\end{equation*}
will be evaluated below and is absolutely convergent by the local computation.

By the definition of $L(s,\chi_{\wGe_n})$ and $L(s,\chi_{\wM_n})$, we have 
\begin{equation*}
    \dfrac{c_{\wGe_n}}{c_{\wPe_n}}=|D_F|^{-n^2/2}\lim_{s\to1}(s-1)L(s,1)
\end{equation*}

Thus,
\begin{equation*}
    c_0=|D_F|^{-n^2/2}\lim_{s\to1}(s-1)L(s,1)\prod_{v\in\bmf{f}}\dfrac{L_v(1,\chi_{\wGe_n})}{L_v(1,\chi_{\wM_n})}c_{0,v}\times\prod_{v\in\bmf{a}}c_{0,v}.
\end{equation*}

Therefore, we need to fix the highest degree differential forms of $\wGe_n$, $\wM_n$ and $U_{\Pe_n}$ and determine the local constants $c_{0,v}$ for $v\in\bmf{v}$.

\subsection{Choice of differential forms}

We now choose explicit $F$-bases of the Lie algebras of $\wGe_n$, $U_{\Pe_n}$, and $\wM_n$, which determine the corresponding highest degree differential forms.

Let $\gamma$ be an integral element of $E$ such that $\rho_{E/F}(\gamma)=-\gamma$ where $\rho_{E/F}$ is the non-trivial element of $\Gal(E/F)$. Let $E_{ij}$ be the $(i,j)$-elementary matrix of size $n$. We put 
\begin{equation*}
S_{ij}=\begin{cases}
E_{ij}+E_{ji}, & i<j,\\
E_{ii}, & i=j,
\end{cases}
\end{equation*}
and for $1\leq i<j\leq n$, put
\begin{equation*}
    A_{ij}=E_{ij}-E_{ji}.
\end{equation*}
Then $\{S_{ij}\}_{1\leq i\leq j\leq n}$ and
$\{\gamma A_{ij}\}_{1\leq i<j\leq n}$ give an $F$-basis of the space $\Her_n$ of hermitian matrices.

We put 
\begin{align*} 
X_{ij}={}&\begin{pmatrix}
E_{ij}&0\\
0&-E_{ji}
\end{pmatrix},&
W_{ij}={}&\gamma\begin{pmatrix}
E_{ij}&0\\
0&E_{ji}
\end{pmatrix}
\end{align*}
for $1\leqq i,j\leqq n$, put
\begin{align*}
Y_{ij}={}&\begin{pmatrix}
0&S_{ij}\\
0&0
\end{pmatrix},&
Y'_{ij}={}&\begin{pmatrix}
0&0\\
S_{ij}&0
\end{pmatrix}
\end{align*}
for $1\leqq i\leqq j\leqq n$ and put
\begin{align*}
V_{ij}={}&\gamma\begin{pmatrix}
0&A_{ij}\\
0&0
\end{pmatrix},&
V'_{ij}={}&\gamma\begin{pmatrix}
0&0\\
A_{ij}&0
\end{pmatrix}
\end{align*}
for $1\leqq i<j\leqq n$.
We put 
\begin{equation*}
    D=\begin{pmatrix}
    0&0\\
    0&1_n
\end{pmatrix}.
\end{equation*}

Then we also put the set
\begin{align*}
\hana{B}_{\wM_n}={}&\{X_{ij}|1\leqq i,j\leqq n\}\cup\{W_{ij}|1\leqq i,j\leqq n\}\cup\{D\},\\
\hana{B}_{U_{\Pe_n}}={}&\{Y_{ij}|1\leqq i\leqq j\leqq n\}\cup\{V_{ij}|1\leqq i<j\leqq n\},\\
\hana{B}_{\wGe_n}={}&\hana{B}_{\wM_n}\cup\hana{B}_{U_{\Pe_n}}\cup\{Y_{ij}'|1\leqq i\leqq j\leqq n\}\cup\{V_{ij}'|1\leqq i<j\leqq n\}.
\end{align*}

One checks directly from the defining equation of $\wGe_n$ that the elements in $\hana{B}_{\wM_n}$, $\hana{B}_{U_{\Pe_n}}$, and $\hana{B}_{\wGe_n}$ form
$F$-bases of
\begin{equation*}
\mr{Lie}\,(\wM_n),\qquad
\mr{Lie}\,(U_{\Pe_n}),\qquad
\mr{Lie}\,(\wGe_n),
\end{equation*}
respectively. In particular, their cardinalities are
\begin{equation*}
    \#\hana{B}_{\wM_n}=2n^2+1,
\qquad
\#\hana{B}_{U_{\Pe_n}}=n^2,
\qquad
\#\hana{B}_{\wGe_n}=4n^2+1.
\end{equation*}

We define the left invariant highest degree differential forms
$\omega_{\wM_n}$, $\omega_{U_{\Pe_n}}$, and $\omega_{\wGe_n}$ to be the forms
corresponding to the ordered bases
\begin{equation*}
\hana{B}_{\wM_n},\qquad
\hana{B}_{U_{\Pe_n}},\qquad
\hana{B}_{\wGe_n},
\end{equation*}
respectively. We put $\omega_{\wPe_n}=\omega_{\wM_n}\wedge\omega_{U_{\Pe_n}}$. For each place $v$, these forms induce Haar measures
\begin{equation*}
\mu_{\wM_n,v},\qquad
\mu_{U_{\Pe_n},v},\qquad
\mu_{\wPe_n,v},\qquad
\mu_{\wGe_n,v}
\end{equation*}
on
\begin{equation*}
    \wM_n(F_v),\qquad
    U_{\Pe_n}(F_v),\qquad
\wPe_n(F_v),\qquad
\wGe_n(F_v).
\end{equation*}

\subsection{Local Haar measures determined by differential forms}

For $H=\wGe_n,\wM_n$, or $U_{\Pe_n}$, let $\mu_{\hana{B}_H,v}$ denote the local Haar measure determined by the basis $\hana{B}_H$. If $\hana{B}'_H$ is another $F_v$-basis and $A$ is the change-of-basis matrix from $\hana{B}_H$ to $\hana{B}'_H$, then
\begin{equation*}
d\mu_{\hana{B}'_H,v}=|\det A|_v^{-1}d\mu_{\hana{B}_H,v}.
\end{equation*}
Thus the local constants below are obtained by comparing the Lie algebra lattices defined by the chosen bases. At finite places this gives the ratio of the volumes of $\wK_{n,v}$ and $\wK_{\wPe_n,v}$; at archimedean places we compare the induced measures on the compact groups and on the Hermitian symmetric space.

\subsection{Local computations of $c_{0,v}$}
To compute the constant $c_{0,v}$ for $v\in\bmf{v}$, we fix the suitable function $\phi_v$ and compute the ratio \begin{equation*}
    \frac{\int_{\wGe_n(F_v)}\phi_v(g)d\mu_{\wGe_n,v}(g)}{\int_{\wPe_n(F_v)}\phi_v(p)d\mu_{\wPe_n,v}(p)}.
\end{equation*}

Let $v\in\bmf f$. We put the characteristic function $\phi_v=\bm{1}_{\wK_{n,v}}$ on $\wK_{n,v}$ and we have
\begin{equation*}
    c_{0,v}=\frac{\vol(\wK_{n,v})}{\vol(\wK_{\wPe_n,v})}.
\end{equation*}

Write $z=a+b\gamma$ with $a,b\in F_v$. The measure defined by these $F_v$-coordinates differs from the additive Haar measure $dz_v$, normalized by $\vol(\hana O_{E_v},dz_v)=1$, by the lattice factor
\begin{equation*}
\delta_v(\gamma)= \vol(\hana{O}_v[\gamma];dz_v),
\end{equation*}
so that the coordinate measure is $\delta_v(\gamma)^{-1}dz_v$. The chosen basis of $\mr{Lie}(\wM_n)$ contains $n^2$ such $E_v$-coordinates, while the Hermitian unipotent radical contains $n(n-1)/2$ off-diagonal $E_v$-coordinates. Consequently,
\begin{equation*}
\mu_{\wPe_n,v}^{(0)}=\delta_v(\gamma)^{n(3n-1)/2}\mu_{\wPe_n,v},
\qquad
\mu_{\wGe_n,v}^{(0)}=\delta_v(\gamma)^{n(2n-1)}\mu_{\wGe_n,v}.
\end{equation*}
Put 
\begin{equation*}    
L_{\wPe_n,v}^{(0)}=\left\{\begin{pmatrix}
    X_{11}&X_{12}\\
    0&X_{22}
\end{pmatrix}\in\Mat_{2n}(\hana{O}_{E_v})\mid X_{11}+X_{22}^*\in \hana{O}_v\cdot1_n,X_{12}=X_{12}^*\right\}
\end{equation*}
and 
\begin{equation*}
    L_{\wGe_n,v}^{(0)}=\left\{X\in\Mat_{2n}(\hana{O}_{E_v})\mid J_nX+X^*J_n\in\hana{O}_v\cdot J_n\right\}.
\end{equation*}

Then, the measure $d\mu_{H,v}^{(0)}$ is characterized infinitesimally by the condition that the lattice $L_{H,v}^{(0)}$ has volume $1$ in the tangent space for $H=\wPe_n$ and $\wGe_n$.

Put 
\begin{equation*}
\wK_{\maf{p},v}=\left\{\begin{pmatrix}
    x_{11}&x_{12}\\
    x_{21}&x_{22}
\end{pmatrix}\in\wGe_n(F_v)\middle|\begin{matrix}x_{11},x_{22}\in1_n+\Mat_n(\maf{p}_{E_v}),\\
x_{12}\in\maf{p}_v^{-f_v}\Mat_n(\maf{p}_{E_v}),x_{21}\in\maf{p}_v^{f_v}\Mat_n(\maf{p}_{E_v})
\end{matrix}\right\}.
\end{equation*}

We define
\begin{align*}
\eta_v:\;&\Mat_n(\maf{p}_{E_v})
\times\Her_n(F_v;\maf{p}_v^{-f_v}\maf{p}_{E_v})
\times\Her_n(F_v;\maf{p}_v^{f_v}\maf{p}_{E_v})\\
&\longrightarrow\wK_{n,v}
\end{align*}
by
\begin{equation*}
\eta_v(X,Z_1,Z_2)=\begin{pmatrix}
    1_n&0\\
    Z_2&1_n
\end{pmatrix}\m(1_n+X)\n_n(Z_1).
\end{equation*}
Then $\eta_v$ defines the local chart at $1_{2n}$ and by the above characterization, we have 
\begin{align*}
   \int_{\wK_{\maf{p},v}}d\mu_{\wGe_n,v}^{(0)}={}&q_v^{-2n-1}\times q_{E_v}^{-n(2n-1)} \\\int_{\wK_{\maf{p},v}\cap\wPe_n(F_v)}d\mu_{\wPe_n,v}^{(0)}={}&q_v^{-n-1}\times q_{E_v}^{-n(3n-1)/2}\times q_v^{f_vn^2}.
\end{align*}

Therefore, we need only the order of $\wK_{n,v}/\wK_{\maf{p},v}$ and $\wK_{\wPe_n,v}/(\wK_{\maf{p},v}\cap\wPe_n(F_v))$.

Put $k_{F_v}=\hana{O}_v/\maf{p}_{v}$ and $k_{E_v}=O_{E_v}/\maf{p}_{E_v}$.

To state the orders of the reductive quotients uniformly, put
\begin{equation*}
    \varepsilon_v=\begin{cases}
    -1&\text{if $E_v$ is inert over $F_v$},\\
    1&\text{if $E_v$ is split over $F_v$},\\
    0&\text{if $E_v$ is ramified over $F_v$}
\end{cases}
\end{equation*}
and
\begin{equation*}
    \varepsilon_v(i)=\begin{cases}
    \varepsilon_v&\text{if $i$ is odd},\\
    1&\text{if $i$ is even}.
\end{cases}
\end{equation*}
The standard order formulas for finite general linear, unitary, and symplectic similitude groups give
\begin{align*}
\#(\wK_{n,v}/\wK_{\maf p,v})
={}&q_{E_v}^{n(2n-1)}q_v^{2n+1}
\prod_{i=1}^{2n}(1-\varepsilon_v(i)q_v^{-i})(1-q_v^{-1}),\\
&\#(\wK_{\wPe_n,v}/(\wK_{\maf{p},v}\cap\wPe_n(F_v)))\\
={}&q_{E_v}^{n(3n-1)/2}q_v^{n+1}\prod_{i=1}^n(1-q_v^{-i})(1-\varepsilon_vq_v^{-i})\times(1-q_v^{-1}).
\end{align*}

Therefore, we have
\begin{align*}
   \int_{\wK_{n,v}}d\mu_{\wGe_n,v}={}&\prod_{i=1}^{2n}(1-\varepsilon_v(i)q_v^{-i})\times(1-q_v^{-1})\times\delta_v(\gamma)^{-n(2n-1)} 
\end{align*}
and
\begin{align*}
   &\int_{\wK_{\wPe_n,v}}d\mu_{\wPe_n,v}\\={}&\prod_{i=1}^{n}(1-\varepsilon_{v}q_v^{-i})(1-q_v^{-i})\times(1-q_v^{-1}) \times\delta_v(\gamma)^{-\frac{n(3n-1)}{2}}q_v^{f_vn^2}.
\end{align*}

Taking into account the base change matrix of $E_v$ over $F_v$, we can easily obtain $\delta_v(\gamma)=[\hana{O}_{E_v}:\hana{O}_v[\gamma]]^{-1}=\alpha_{E_v}(2\gamma)^{1/2}\alpha_{E_v}(\maf{d}_{E_v/F_v})^{-1/2}$. Then we have
\begin{align*}
    c_{0,v}={}&\dfrac{\prod_{i=1}^{2n}(1-\varepsilon_{v}(i)q_v^{-i})\times(1-q_v^{-1})\times\delta_v(\gamma)^{-n(2n-1)}}{\prod_{i=1}^{n}(1-\varepsilon_{v}q_v^{-i})(1-q_v^{-i})\times(1-q_v^{-1}) \times\delta_v(\gamma)^{-\frac{n(3n-1)}{2}}q_v^{f_vn^2}}\\
    ={}&\prod_{i=1}^n(1-\varepsilon_v(i-1)q_v^{-i})^{-1}
    \prod_{i=n+1}^{2n}(1-\varepsilon_v(i)q_v^{-i})\\
    &\times\alpha_{E_v}(2\gamma)^{-n(n-1)/4}
    \alpha_{E_v}^{n(n-1)/4}(\maf{d}_{E_v/F_v})q_v^{-f_vn^2}
\end{align*}
for $v\in\bmf{f}$.

For $v\in\bmf{a}$, we compute the measure by comparing the measure on the hermitian upper half-plane and one on $\K_{n,v}$.

By Iwasawa decomposition of $\Ge_n(F_v)$, we have $\Pe_n(F_v)\K_{n,v}=\Ge_n(F_v)$. 
Since we have an exact sequence $1\to\Ge_n(F_v)\to\wGe_n(F_v)\to\mb{R}^{\times}\to1$ and have a section $\mb{R}^{\times}\to\wPe_n(F_v)\,t\mapsto\bm{d}_n(t)$, we obtain $\wGe_n(F_v)=\wPe_n(F_v)\K_{v}$.
Moreover, inclusion $\wPe_n(F_v)\to\wGe_n(F_v)$ induces the isomorphism
\begin{align*}    
\wPe_n(F_v)/\K_{\wPe_n,v}\cong\wGe_n(F_v)/\K_{v}\cong&\maf{H}_n\times\mb{R}^{\times}.
\end{align*}

We first specify the Haar measures on $\K_{n,v}$ and $\K_{\wPe_{n,v}}$, and then define the corresponding quotient measures.

Let $A$ be $\dfrac{1}{\sqrt{2}}\begin{pmatrix}
1_n&\sqrt{-1}1_n\\
1_n&-\sqrt{-1}1_n
\end{pmatrix}$. 
Since the map
\begin{align*}
\iota_{\K_{n,v}}:\;&U(n)\times U(n)\longrightarrow \K_{n,v},\\
&(u_1,u_2)\longmapsto\begin{pmatrix}
    \dfrac{u_1+u_2}{2}&\sqrt{-1}\dfrac{u_1-u_2}{2}\\
    -\sqrt{-1}\dfrac{u_1-u_2}{2}&\dfrac{u_1+u_2}{2}
\end{pmatrix}=A^{-1}\begin{pmatrix}
    u_1&0\\
    0&u_2
\end{pmatrix}A
\end{align*}
gives an isomorphism. 
Its restriction to the diagonal subgroup identifies $\K_{\wPe_n,v}$ with $ U(n)$. 
Put $\iota_{\K_{\wPe_n,v}}=\iota_{\K_n,v}|_{\K_{\wPe_n,v}}$.
We then can apply the Lie theory for $\K_{n,v}$ and $\K_{\wPe_n,v}$.

By the above isomorphism, we realize the Lie algebra of $\K_{n,v}$ by
\begin{equation*}
    \mr{Lie}\,(\K_{n,v})=A^{-1}\left\{\begin{pmatrix}
        X_1&0\\
        0&X_2
    \end{pmatrix}\mid X_1+X_1^*=0,X_2+X_2^*=0\right\}A
\end{equation*}

We put $\mr{Ad}(A)X=AXA^{-1}$.
Put 
\begin{align*}
&\hana{B}_{\K_{\wPe_n,v}}\\
={}&\{\mr{Ad}(A)(X_{ij}-X_{ji})|1\leqq i<j\leqq n\}\cup\{\mr{Ad}(A)(W_{ij}+W_{ji})|1\leqq i< j\leqq n\}\\
&\cup\{\mr{Ad}(A)(W_{ii})|1\leqq i\leqq n\}
\end{align*}
and
\begin{equation*}
\hana{B}_{\K_{n,v}}=\hana{B}_{\K_{\wPe_n,v}}\cup\{\mr{Ad}(A)(Y_{ij}-Y_{ij}')|1\leqq i\leqq j\leqq n\}\cup\{\mr{Ad}(A)(V_{ij}+V_{ij}')|1\leqq i< j\leqq n\}.
\end{equation*}
Then the elements in $\hana{B}_{\K_{n,v}}$ form the basis of the Lie algebra of $\K_{n,v}$ and the elements in $\hana{B}_{\K_{\wPe_n,v}}$ form the basis of the Lie algebra of $\K_{\wPe_n,v}$.
We define the Haar measure $\mu_{\K_{n,v}}$ on $\K_{n,v}$ determined by $\hana{B}_{\K_{n,v}}$ and the Haar measure $\mu_{\K_{\wPe_n,v}}$ on $\K_{\wPe_n,v}$ determined by $\hana{B}_{\K_{\wPe_n,v}}$.

Let $e_n={}^t(0,0,\cdots,1)\in\mb{C}^n$. 
Define the map $\pi_{U(n)}:U(n)\to S^{2n-1}$ by $\pi_{U(n)}(g)=g\cdot e_n$. 
Then $\mr{Stab}_{e_n}(U(n)):=\{g\in U(n)|g\cdot e_n=e_n\}\cong U(n-1)$. 
We regard $S^{2n-1}$ as the unit sphere in $\mb{C}^n\cong\mb{R}^{2n}$ and equip it with the standard Euclidean volume form
\begin{equation*}
\omega_{S^{2n-1}}=\sum_{i=1}^{2n}(-1)^{i-1}x_i\,dx_1\wedge\cdots\wedge\widehat{dx_i}\wedge\cdots\wedge dx_{2n}.
\end{equation*}
Let $\mu_{S^{2n-1}}$ be the associated measure. By the standard formula for the volume of the unit sphere, we have
\begin{equation*}
\vol(S^{2n-1},\mu_{S^{2n-1}})=\dfrac{2\pi^n}{\Gamma(n)}=\dfrac{2^{-n+2}}{\Gamma_{\mb{C}}(n)}.
\end{equation*}

Put $\pi_{U(n)\times U(n)}=\pi_{U(n)}\times\pi_{U(n)}$, $u_{\K_{n,v}}=\pi_{U(n)\times U(n)}\circ\iota_{\K_{n,v}}^{-1}$, and $u_{\K_{\wPe_n,v}}=\pi_{U(n)}\circ\iota_{\K_{\wPe_n,v}}^{-1}$. We equip $S^{2n-1}\times S^{2n-1}$ with the product volume form
\begin{equation*}
\omega_{S^{2n-1}\times S^{2n-1}}=\operatorname{pr}_1^*\omega_{S^{2n-1}}\wedge\operatorname{pr}_2^*\omega_{S^{2n-1}}.
\end{equation*}
We denote by $\mu_{S^{2n-1}\times S^{2n-1}}$ the measure on $S^{2n-1}\times S^{2n-1}$ determined by $\omega_{S^{2n-1}\times S^{2n-1}}$.

Let $\gamma_{v,0}\in\mb{R} $ such that $\gamma=\gamma_{v,0}\sqrt{-1}\in E_v$ for $v\in\bmf{a}$.

A direct comparison of the above bases under the fibrations
\begin{equation*}
U(n-1)\longrightarrow U(n)\longrightarrow S^{2n-1}
\end{equation*}
gives the recursive volume formulas
\begin{align*}
\vol(\K_{\wPe_n,v})
={}&|\gamma_{v,0}|^{-n}\vol(S^{2n-1})\vol(\K_{\wPe_{n-1},v}),\\
\vol(\K_{n,v})
={}&2^{2n-1}|\gamma_{v,0}|^{-(2n-1)}\vol(S^{2n-1})^2\vol(\K_{n-1,v}).
\end{align*}
Using the sphere-volume formula above and induction, we obtain
\begin{align*}
    \vol(\K_{\wPe_n,v},d\mu_{\K_{\wPe_n,v}})=&\prod_{i=1}^n|\gamma_{v,0}|^{-i}\dfrac{2^{-i+2}}{\Gamma_{\mb{C}}(i)}\\
    =&|\gamma_{v,0}|^{-n(n+1)/2}2^{-n(n-3)/2}\prod_{i=1}^n\dfrac{1}{\Gamma_{\mb{C}}(i)}\\
    \vol(\K_{n,v},d\mu_{\K_{n,v}})=&\prod_{i=1}^n2^{2i-1}|\gamma_{v,0}|^{-2i+1}\dfrac{2^{-2i+4}}{\Gamma_{\mb{C}}(i)^2}\\
    =&|\gamma_{v,0}|^{-n^2}2^{3n}\prod_{i=1}^n\dfrac{1}{\Gamma_{\mb{C}}(i)^2}.\\
\end{align*}

Let $\mu_{\maf{H}_n}$ be the $\Ge_n(F_v)$-invariant measures on $\maf{H}_n$ such that for $Z=X+\sqrt{-1}Y$ 
\begin{equation*}
    \int_{\maf{H}_n}\phi(Z)d\mu_{\maf{H}_n}(Z)=\int_{\maf{H}_n}\phi(X+\sqrt{-1}Y)|\det Y|^{-2n}dXdY
\end{equation*}
where $dX,dY$ are defined in Section \ref{subsec2-2}.
Second, we compare the quotient measure $\mu_{\wGe_n,v}/\mu_{\K_{n,v}}$ and $\mu_{\wPe_n,v}/\mu_{\K_{\wPe_n,v}}$ and  $\mu_{\maf{H}_n}$.
We define the action $\wGe_n(\mb{R})\times(\maf{H}_n\times\mb{R}^{\times})\to \maf{H}_n\times\mb{R}^{\times}$ by $(g,(Z,t))\mapsto ((\bm{d}_n(\nu(g)^{-1})g)Z,\nu(g)t)$ where $gZ$ is defined in Section \ref{subsec2-2}. Then this action induces the Lie map 
\begin{equation*}
    \mr{Lie}\,(\wGe_n)(\mb{R})\to T_{\bm{i}_n}\maf{H}_n\oplus\mb{R},\,\left(\begin{pmatrix}
    X_{11}&X_{12}\\
    X_{21}&X_{22}
    \end{pmatrix},t\right)\mapsto ((X_{11}-X_{22})\sqrt{-1}+(X_{12}+X_{21}), t)
\end{equation*}
where we regard $\mr{Lie}\,(\wGe_n)(\mb{R})=\mr{Lie}\,(\Ge_n)(\mb{R})\oplus\mb{R}$ and  consider
\begin{equation*}
    \mr{Lie}\,(\Ge_n)(\mb{R})=\left\{\begin{pmatrix}
        X_{11}&X_{12}\\
        X_{21}&X_{22}
    \end{pmatrix}\in\Mat_{2n}(\mb{C})\mid X_{11}+X_{22}^*=0\text{, }X_{12}^*=X_{12}\text{ and }X_{21}^*=X_{21}\right\}.
\end{equation*}
Therefore, comparing with the coefficients, both the Haar measure $\mu_v^{\wGe_n}$ and $\mu_v^{\wPe_n}$ induce the measure 
\begin{equation*}
2^{-n}\alpha_{E_v}(\gamma)^{-n(n-1)/2}\mu_{\maf{H}_n} d^{\times}t
\end{equation*}
on $\maf{H}_n\times\mb{R}^{\times}$ where $\int_{[1,2]}d^{\times}t=\log2$.

Therefore, with the Haar measures fixed above, for every right $K_{n,v}$-invariant integrable function on $\wGe_n(F_v)$, we have
\begin{align*}
    \int_{\wGe_n(F_v)}\phi_v(g)d\mu_{\wGe_n,v}(g)={}&2^{-n}\alpha_{E_v}(\gamma)^{-n(n-1)/2}\int_{\maf{H}_n\times\mb{R}^{\times}}\phi_{\maf{H}_n}(Z,t)d\mu_{\maf{H}_n}(Z)d^{\times}t \vol(\K_{n,v},\mu_{\K_{n,v}})\\
    \int_{\wPe_n(F_v)}\phi_v(p)d\mu_{\wPe_n,v}(p)={}&2^{-n}\alpha_{E_v}(\gamma)^{-n(n-1)/2}\int_{\maf{H}_n\times\mb{R}^{\times}}\phi_{\maf{H}_n}(Z,t)d\mu_{\maf{H}_n}(Z)d^{\times}t \vol(\K_{\wPe_n,v},\mu_{\K_{\wPe_n,v}})
\end{align*}
where we put $\phi_{\maf{H}_n}(g\cdot\bm{i}_n,\nu(g))=\phi_v(g)$ and we have
\begin{equation*}
    c_{0,v}=\dfrac{\vol(\K_{n,v},\mu_{\K_{n,v}})}{\vol(\K_{\wPe_n,v},\mu_{\K_{\wPe_n,v}})}.
\end{equation*}

Therefore, we have
\begin{align*}
    c_{0,v}={}&\dfrac{2^{-n}\alpha_{E_v}(\gamma)^{-n(n-1)/2} \vol(\K_{n,v},\mu_{\K_{n,v}})}{2^{-n}\alpha_{E_v}(\gamma)^{-n(n-1)/2} \vol(\K_{\wPe_n,v},\mu_{\K_{\wPe_n,v}})}\\
    ={}&\alpha_{E_v}(\gamma)^{-(n^2-n)/4}2^{n(n+3)/2}\prod_{i=1}^n\Gamma_{\mb{C}}(i)^{-1}.
\end{align*}

Then since $\bm{I}^{-1}(\wK_{n,\bmf{f}}\wGe_n(\mb{R}\otimes F))=(\wK_{\wPe_n,\bmf{f}}\wPe_n(\mb{R}\otimes F))\times\wK_n$, 
\begin{equation*}
    c_0=\dfrac{\int_{\wGe_n(\A)}\phi(g)d\mu_{\wGe_n,\A}(g)}{\int_{\wPe_n(\A)}\phi(p)d\mu_{\wPe_n,\A}(p)}=\prod_{v\in\bmf{v}}\dfrac{\int_{\wGe_n(F_v)}\phi(g)d\mu_{\wGe_n,v}(g)}{\int_{\wPe_n(F_v)}\phi(p)d\mu_{\wPe_n,v}(p)}=\prod_{v\in\bmf{v}}c_{0,v}
\end{equation*}
for any factorizable test function $\phi=\prod_{v\in\bmf{v}}\phi_v$ on $\wGe_n(\A)$.

Since $\prod_{v\in\bmf{v}}\alpha_{E_v}(\gamma)=1$, we therefore have 
\begin{align*}
    c_0={}&|D_F|^{-n^2/2}\lim_{s\to1}(s-1)L(s,1)\\
    &\times\prod_{v\in\bmf{f}}\dfrac{\prod_{i=n+1}^{2n}(1-\varepsilon_v(i)q_v^{-i})}{\prod_{i=1}^n(1-\varepsilon_v(i-1)q_v^{-i})}\alpha_{E_v}(2\gamma)^{-n(n-1)/4}\alpha_{E_v}^{n(n-1)/4}(\maf{d}_{E_v/F_v})q_v^{-f_vn^2}\\
    &\times\prod_{v\in\bmf{a}}\alpha_{E_v}(\gamma)^{-(n^2-n)/4}2^{n(n+3)/2}\prod_{i=1}^n\Gamma_{\mb{C}}(i)^{-1}\\
    ={}&2^{n(n+1)d_F}|D_E|^{-n(n-1)/4}|D_F|^{-n^2-n/2}\lim_{s\to1}(s-1)L(s,1)\dfrac{\prod_{i=2}^nL(i,\varepsilon_{E/F}^{i-1})}{\prod_{i=n+1}^{2n}L(i,\varepsilon_{E/F}^i)}\times\prod_{i=1}^n\Gamma_{\mb{C}}(i)^{-d_F}.
\end{align*}

\section{Appendix: Normalization of the volume}\label{ap:locden}
Now we assume $n$ is odd.

Put $t=\det B$, which is totally positive since $B\in\Her_n(F)^+$. Since $n$ is odd, we have $\varepsilon_v(t\cdot1_n)=\varepsilon_{E_v/F_v}(t^n)=\varepsilon_v(B)$ for every $v\in\bmf{f}$. Hence, Proposition~\ref{prop:pre2} yields an element $h\in\GL_n(E)$ such that $B[h]=t\cdot1_n$.

In order to consider the quotient measure $\mu_{\wM_B\lsla\wM_n}$, we recall the adelic measure on $\wM_B(\A)$ defined as in \ref{ap:Tamagawa}.

We fix a top-degree differential form $\omega_{\wM_B}$ on $\wM_B$ and let $\mu_{\wM_B,v}$ be the Haar measure on $\wM_B(F_v)$ induced by $\omega_{\wM_B}$ for any $v\in\bmf{v}$.

Similarly for $\wGe_n$, $X^*(\wM_B)$ is generated by the similitude character and the determinant and $X^*(\wM_B)\otimes\mb{Q}\cong\mb{Q}[\Gal(E/F)]$.

Thus, we have 
\begin{equation*}
\mu_{\wM_B,\A}=\dfrac{|D_F|^{-(n^2+1)/2}}{\lim_{s\to1}(s-1)L(s,\chi_{\wM_B})}\prod_{v\in\bmf{f}}L_v(1,\chi_{\wM_B})\mu_{\wM_B,v}\prod_{v\in\bmf{a}}\mu_{\wM_B,v}
\end{equation*}
where $L_v(s,\chi_{\wM_B})=L_v(s,1)L_v(s,\varepsilon_{E/F})$ and $L(s,\chi_{\wM_B})=\prod_{v\in\bmf{f}}L_v(s,\chi_{\wM_B})$.

The adelic Haar measures $\mu_{\wM_n,\A}$ and $\mu_{\wM_B,\A}$ induce an invariant quotient measure $\mu_{\wM_B\lsla\wM_n}=\frac{\mu_{\wM_n,\A}}{\mu_{\wM_B,\A}}$ on $\wM_B(\A)\lsla\wM_n(\A)$. 
For each place $v\in\bmf{v}$, let $\mu_{\wM_B\lsla\wM_n,v}
=\frac{\mu_{\wM_n,v}}{\mu_{\wM_B,v}}$ denote the corresponding local quotient measure. 
Comparing the local factorizations of the two adelic Haar measures, we obtain
\begin{equation*}
\mu_{\wM_B\lsla\wM_n}
=
\dfrac{|D_F|^{-n^2/2}}
{\displaystyle\lim_{s\to1}(s-1)L(s,1)}
\prod_{v\in\bmf f}
(1-q_v^{-1})^{-1}
\mu_{\wM_B\lsla\wM_n,v}
\prod_{v\in\bmf a}
\mu_{\wM_B\lsla\wM_n,v}.
\end{equation*}

To describe these local quotient measures explicitly, we now realize the corresponding local homogeneous spaces as more concrete spaces. 

If \(v\in\bmf f\), then, since \(n\) is odd, the orbit map
\begin{equation*}
\mr{orb}_B:\wM_B(F_v)\lsla\wM_n(F_v)\rightarrow\Her_n(F_v)^{\mathrm{nd}},\qquad m\mapsto B[m],
\end{equation*}
is an isomorphism. We may therefore regard
\(\mu_{\wM_B\lsla\wM_n,v}\) as a
\(\wM_n(F_v)\)-invariant measure on
\(\Her_n(F_v)^{\mathrm{nd}}\).
The archimedean local quotients will be treated separately using the Iwasawa decomposition.

We first determine the normalization of the local quotient measure at each finite place. 
For this purpose, we fix bases of the relevant Lie algebras and compare the induced measures.

Let $\hana{B}_{\wM_n}$ be as in \ref{ap:Tamagawa}.
Put 
\begin{equation*}
\hana{B}_{\wM_n}^h=\left\{\begin{pmatrix}
    hXh^{-1}&0\\
    0&(h^*)^{-1}(-X^*+x1_n)h^*
\end{pmatrix}\middle|\begin{pmatrix}
    X&0\\
    0&-X^*+x1_n
\end{pmatrix}\in \hana{B}_{\wM_n}\right\}.
\end{equation*}
Let $\maf{L}^{\wM_n,h}$ be the $O_F$-lattice generated by $\hana{B}_{\wM_n}^h$.

Put $\maf{L}^{\wM_B}=\mr{Lie}(\wM_B)\cap\maf{L}^{\wM_n,h}$. 
The relation $B[h]=t\cdot 1_n$ identifies $\maf{L}^{\wM_B}$, by conjugation, with the standard stabilizer lattice, which is free over $\hana{O}_F$. 
We fix an $O_F$-basis $\hana{B}_{\wM_B}$ of this lattice.
Then, for $v\in\bmf{v}$, $\mu_{\wM_B,v}$ is the Haar measure on $\wM_B(F_v)$ determined by $\hana{B}_{\wM_B}$ and $\mu_{\wM_n,v}$ can be regarded as the Haar measure determined by $\hana{B}_{\wM_n}^h$.

We denote by $\mu_{\Her_{n,v}}$ the Haar measure on $\Her_n(F_v)$ normalized by
\begin{equation*}
\mu_{\Her_{n,v}}(\Her_n(F_v;\hana{O}_{E_v}))=1.
\end{equation*}
Then $\mu_{\Her_{n,v}}^*=\alpha_{F_v}^{-n}(\det Z_v)\mu_{\Her_{n,v}}$ is the $\wM_n(F_v)$-invariant measure on $\Her_n(F_v)^{nd}$. 
Then, there exists a constant $c_v^B>0$ such that $d\mu_{\wM_B\lsla\wM_n,v}(Z_v)=c_v^B\cdot d\mu_{\Her_{n,v}}^*(Z_v)$.

We determine $c_v^B$ for $v\in\bmf f$ by comparing the Lie algebra lattices under the differential of the orbit map. Multiplying the invariant measure by $\alpha_{F_v}^n(\det Z)$ gives the corresponding additive measure on the Hermitian space.

Put 
\begin{align*}
\maf{L}_v^{\wM_n,h}={}&\maf{L}^{\wM_n,h}\otimes_{\hana{O}_F}\hana{O}_{F_v}\\
\maf{L}_v^{\wM_B}={}&\maf{L}^{\wM_B}\otimes_{\hana{O}_F}\hana{O}_{F_v}.
\end{align*}

Identifying $\mr{Lie}\,\wM_n(F_v)=\Mat_n(E_v)\times F_v$ and $T_B\Her_n(F_v)^{nd}=\Her_n(F_v)$, the differential of the orbit map is
\begin{equation*}
d\mr{orb}_B(X,t)=BX+X^*B-tB.
\end{equation*}
Put $\maf{L}_v^B=d\mr{orb}_B(\maf{L}_v^{\wM_n,h})$. 
Then for any $1\leqq i,j \leqq n$, we have
\begin{align*}
d\mr{orb}_B(hE_{ij}h^{-1},0)={}&t\cdot(E_{ij}+E_{ji})[h^{-1}],\\
d\mr{orb}_B(h\gamma E_{ij}h^{-1},0)={}&t\gamma\cdot (E_{ij}-E_{ji})[h^{-1}],\\
d\mr{orb}_B(0,1)={}&t\cdot1_n[h^{-1}].
\end{align*}
Thus, $t^{-1}\maf{L}_v^B[h]=\{Z\in\Her_n(F_v;\hana{O}_{F_v}[\gamma])|Z_{ii}\in 2O_{F_v}\}+\hana{O}_{F_v}1_n$.

Let $S_{ij}$, $A_{ij}$ as in \ref{ap:Tamagawa}.
Let $\maf{M}_{\Her_n,v}$ be the $\hana{O}_{F_v}$-lattice generated by $\hana{B}_{\Her_n}=\{S_{ij}|1\leqq i\leq j\leqq n\}\cup\{\gamma A_{ij}|1\leqq i<j\leqq n\}$ and $\mu'_v$ be the Haar measure on $\Her_n(F_v)$ determined by $\hana{B}_{\Her_n}$. 
As in \ref{ap:Tamagawa}, by comparing the $E_v$-basis $\{1\}$ and $F_v$-basis $\{1,\gamma\}$, we have $\mu'_v=\alpha_{E_v}(2\gamma)^{-n(n-1)/4}\alpha_{E_v}(\maf{d}_{E_v/F_v})^{n(n-1)/4}\mu_{\Her_{n,v}}$.

On the other hand, we have $[\maf{M}_{\Her_n,v}:t^{-1}\maf{L}_v^B[h]]=|2|_v^{-n+1}$ and 
\begin{equation*}
    \int_{\hana{L}_v^B}d\mu'_{v}(Z)=\alpha_{F_v}^{-n}(N_{E/F}(\det h)t^{-n})\int_{t^{-1}\hana{L}_v^B[h]}d\mu'_{v}(Z)=\alpha_{F_v}^{n}(\det B)\alpha_{F_v}^{n-1}(2).
\end{equation*}

Thus, we have 
\begin{equation*}
    \mu_{\wM_B\lsla\wM_n,v}=\alpha_{E_v}(2\gamma)^{-n(n-1)/4}\alpha_{E_v}(\maf{d}_{E_v/F_v})^{n(n-1)/4}\alpha_{F_v}^{-n+1}(2)\mu_{\Her_{n,v}}^*
\end{equation*}
and 
\begin{align*}
    c_v^B={}&\alpha_{E_v}(2\gamma)^{-n(n-1)/4}\alpha_{E_v}(\maf{d}_{E_v/F_v})^{n(n-1)/4}\alpha_{F_v}^{-n+1}(2)\\
    ={}&\alpha_{E_v}(\gamma)^{-n(n-1)/4}\alpha_{F_v}(2)^{-(n+2)(n-1)/2}\alpha_{E_v}(\maf{d}_{E_v/F_v})^{n(n-1)/4}.
\end{align*}

Here the Jacobian factor $\alpha_{F_v}^{n}(\det B)$ is canceled by the factor $\alpha_{F_v}^{-n}(\det B)$ obtained by evaluating the invariant density at $B$.

For $v\in\bmf{a}$, 
\begin{align*}
\wM_B(F_v)={}&\{\m(\sqrt{x}A)\bm{d}_n(x)| A\in hU(n)h^{-1},x\in\mb{R}_{>0}\}
\end{align*}
and
\begin{align*}
\hana{B}_v^{\wM_B}={}&\m(h)\hana{B}_{\K_{\wPe_n,v}}\m(h)^{-1}\cup\m(h)\left\{\sum_{i=1}^n X_{ii}+2D\right\}\m(h)^{-1}.
\end{align*}

Then $\hana{B}_v^{\wM_B}$ is the basis of $\mr{Lie}\wM_B(F_v)$. 
Let $\mu_v'$ be the measure on $\wM_B(F_v)$ determined by $\hana{B}_v^{\wM_B}$. 

Put $Q=\{b:\text{upper-triangular matrix in $\GL_n(\mb{C})$}|b_{ii}\in \mb{R}_{>0}\}$. Then we fix the set of $\mb{R}$-basis $\hana{B}^Q$ of $\mr{Lie}$ $Q$ defined by
\begin{equation*}
\{E_{ij}|1\leqq i\leqq j\leqq n\}\cup\{\sqrt{-1}E_{ij}|1\leqq i<j\leqq n\}.
\end{equation*}
Let $\mu_{Q,v}$ be the right Haar measure on $Q$ determined by $\hana{B}^Q$.
Then we have 
\begin{align*}
\wM_n^+(F_v)={}&\wM_B(F_v)\{\m(A)|A\in hQh^{-1}\}
\end{align*}
by Iwasawa decomposition of $\GL_n(\mb{C})$ and
\begin{align*}
    \{1_{2n}\}={}&\wM_B(F_v)\cap\{\m(A)|A\in hQh^{-1}\}.
\end{align*}
Then $\mu_{Q,v}$ and $\mu_v'$ determine the measure $\mu_{\wM_n,v}'$ on $\wM_n(F_v)$.

We can consider $\mu_{\wM_n,v}'$ as the measure formed by the dual basis of $\hana{B}_v^{\wM_B}\cup h\hana{B}^Qh^{-1}$. 
By comparing the coefficient of basis which determines the measure $\mu_{\wM_n,v}'$ and $\mu_{\wM_n,v}$, we have $\mu_{\wM_n,v}=2|\gamma_{v,0}|^{-n(n-1)/2}\mu_{\wM_n,v}'$ where $\gamma_{v,0}=-\sqrt{-1}\gamma$ in $E_v$ as in \ref{ap:Tamagawa}.

Similarly, comparing with the basis $\hana{B}_v^{\wM_B}$ and $\hana{B}_{\wM_B}$ under complexation of $\mr{Lie}(\wM_B)(F_v)$, we have $\mu_{\wM_B,v}=\mu_v'$.

Thus, $\frac{\mu_{\wM_n,v}}{\mu_{\wM_B,v}}=2\gamma_{v,0}^{-n(n-1)/2}\mu_{Q,v}$ under the topological isomorphism $\wM_B(F_v)\lsla\wM_n^+(F_v)\cong Q$ and $c_v^B=2\alpha_{E_v}^{-n(n-1)/4}(\gamma)$.

Since $\gamma\in E^{\times}$, $\prod_{v} \alpha_{E_v}(\gamma)=1$.
Therefore, we have 
\begin{align*}    
&\prod_{v\in\bmf{v}}c_v^B\\
={}&\prod_{v\in\bmf{f}}\alpha_{E_v}(\gamma)^{-n(n-1)/4}\alpha_{E_v}(\maf{d}_{E_v/F_v})^{n(n-1)/4}\alpha_{F_v}^{-(n+2)(n-1)/2}(2)\times\prod_{v\in\bmf{a}}2\alpha_{E_v}^{-n(n-1)/4}(\gamma)\\
={}&2^{n(n+1)d_F/2}|D_F^2/D_E|^{n(n-1)/4}.
\end{align*}

Then we obtain 
\begin{equation*}
    \mu_{\wM_B\lsla\wM_n}=\dfrac{2^{n(n+1)d_F/2}|D_E|^{-n(n-1)/4}|D_F|^{-n/2}}{\lim_{s\to1}(s-1)L(s,1)}\prod_{v\in\bmf{f}}(1-q_v^{-1})^{-1}\mu^*_{\Her_n,v}\prod_{v\in\bmf{a}}\mu_{Q,v}
\end{equation*}

%-
% References
%-

\section*{Declaration of generative AI and AI-assisted technologies in the manuscript preparation process}
During the preparation of this work, the author used ChatGPT (OpenAI) to assist with language editing and to identify possible issues in the presentation and notation. The author reviewed and edited all suggestions and takes full responsibility for the content of the article.

\nocite{C}
\nocite{F}
\nocite{II}
\nocite{Ike1}
\nocite{Ike2}
\nocite{IKK}
\nocite{K1}
\nocite{KK1}
\nocite{Lan1}
\nocite{M}
\nocite{MS}
\nocite{S1}
\nocite{S2}

\bibliography{references}
\bibliographystyle{abbrv}

\end{document}